\documentclass[12pt, reqno]{amsart}
\usepackage{amsmath,amsopn,amssymb,amsthm}
\usepackage[T2A]{fontenc}
\usepackage[cp1251]{inputenc}
\usepackage[english]{babel}
\usepackage[backref=page, breaklinks=true,colorlinks=true,linkcolor=blue,citecolor=blue,urlcolor=blue]{hyperref}

\newcommand{\arccot}{\mathop{\rm arccot}\nolimits}
\usepackage{graphicx}

\renewenvironment{proof}[1][Proof]{\textbf{#1.} }{\ \rule{0.5em}{0.5em}}

\DeclareMathOperator{\diam}{diam}

\DeclareMathOperator{\bd}{bd}

\DeclareMathOperator{\co}{co}
\DeclareMathOperator{\intt}{int}

\renewenvironment{proof}[1][Proof]{\textbf{#1.} }
{\ \rule{0.5em}{0.5em}}
\newtheorem{theorem}{Theorem}
\newtheorem{prop}{Proposition}
\newtheorem{lemma}{Lemma}
\newtheorem{corollary}{Corollary}

\newtheorem{conjecture}{Conjecture}

\theoremstyle{definition}

\newtheorem{remark}{Remark}

\begin{document}

\title
[The extreme polygons for the self Chebyshev radius \dots]
{The extreme polygons for the self \\ Chebyshev radius of the boundary}
\author{E.V.~Nikitenko, Yu.G.~Nikonorov}

\address{Nikitenko Evgeni\u\i\  Vitalievich\newline
Rubtsovsk Industrial Institute of\newline
Altai State Technical University \newline
after I.I.~Polzunov \newline
Rubtsovsk, Traktornaya st., 2/6, \newline
658207, Russia}
\email{evnikit@mail.ru}

\address{Nikonorov Yuri\u{\i} Gennadievich \newline
Southern Mathematical Institute of \newline
the Vladikavkaz Scientific Center of \newline
the Russian Academy of Sciences, \newline
Vladikavkaz, Vatutina st., 53, \newline
362025, Russia}
\email{nikonorov2006@mail.ru}

\begin{abstract}

The paper is devoted to some extremal problems for convex polygons on the Euclidean plane, related to the concept of
self Chebyshev radius for the polygon boundary. We consider a general problem of minimization of the perimeter among all $n$-gons with a fixed
self Chebyshev radius of the boundary. The main result of the paper is the complete solution of the mentioned problem for $n=4$:
We proved that the quadrilateral of minimum perimeter is a so called magic kite, that verified the corresponding conjecture by Rolf~Walter.

\vspace{2mm}

\vspace{2mm} \noindent Key words and phrases: approximation by polytopes, convex curve, convex polygon, Chebyshev radius, minimum perimeter, self Chebyshev radius.
\end{abstract}

\maketitle

\section{Introduction and the main result}

For a given metric space $(X, d)$, we denote by
$B(x, r)$ the closed ball with
center $x$ and radius $r$.
If $M$ is a nonempty bounded subset of a metric space $(X, d)$, then by
$\diam(M) = \sup_{x,y\in M} d(x, y)$ we denote {\it the diameter of $M$}, and by
$$
r(M) := \inf
\{a \geq 0 \,|\, \exists \, x \in X,\, M\subset B(x, a)\}
= \inf_{x\in X} \sup_{y\in M} d(x,y)
$$
we denote {\it the Chebyshev radius} of $M$. A point $x_0 \in X$ for which $M\subset B(x_0, r(M))$ is called {\it a Chebyshev center} of $M$.
In general, a Chebyshev center of a set is not unique, and therefore we denote by $Z(M)$
{\it the set of all Chebyshev centers} of a bounded set $M$. The mapping $M \mapsto Z(M)$
is known as the Chebyshev-center map.

Given a nonempty bounded subset $M$ of $X$ and a nonempty set $Y \subset X$, {\it the relative Chebyshev radius} (of the set $M$ with respect to $Y$)
is defined by the formula
$$
r_Y(M) :=
\inf
\{a \geq 0 \,|\, \exists \, x \in Y,\, M\subset B(x, a)\}
= \inf_{x\in Y} \sup_{y\in M} d(x,y).
$$

In the case $Y=M$, we get the definition of {\it the relative Chebyshev radius of $M$ with respect to $M$ itself}
~i.e. {\it the self Chebyshev radius of the set $M$}:
\begin{equation}\label{chebir1}
\delta(M):=r_M(M)=\inf
\{a \geq 0 \,|\, \exists \, x \in M,\, M\subset B(x, a)\}=\inf\limits_{x\in M}\, \sup\limits_{y\in M} \,d(x,y),
\end{equation}
see, e.g., \cite{AlTsa2019}, \cite[p.~119]{Am1986}, and  \cite{AmZig1980}. \label{defchebrad}
\smallskip

It is clear that this value depends only on the set $M$ and the restriction of the metric $d$ to this set, hence it is an intrinsic characteristic
of the (bounded) metric space $(M, d|_{M\times M})$.
It has also the following obvious geometric meaning for compact $M$: $\delta(M)$ is the smallest radius of a ball having its center in $M$ and covering $M$.
\smallskip

The study of extremal problems for convex curves in the Euclidean plane with a given self Chebyshev radius started with paper
\cite{Walter2017} by Rolf~Walter.
In particular, he conjectured that $L(\Gamma)\geq \pi \cdot \delta(\Gamma)$ for any closed convex curve $\Gamma$ in the Euclidean plane,
where $L(\Gamma)$ is the length of $\Gamma$ and $d$ is the Euclidean metric.
In~\cite{Walter2017}, this conjecture is proved for the case when
$\Gamma$ is a convex curve of class $C^2$ and all curvature centers of $\Gamma$ lie in the interior of $\Gamma$.
It is also shown that the equality $L(\Gamma)=\pi \cdot \delta(\Gamma)$ in this case
holds if and only if $\gamma$ is of constant width; for sets of constant width the reader is
referred to the monograph \cite{MaMoOl}.

It is also proved in \cite{Walter2017} that all $C^2$-smooth convex curves have good approximations by polygonal chains
in terms of the self Chebyshev radius \eqref{chebir1}.
This observation leads to natural extremal problems for convex polygons.
In particular, the following result was also proved in \cite{Walter2017}.
For each triangle  $P$ in the Euclidean plane,  one has
$L(\Gamma)\geq 2\sqrt{3}\,\cdot \delta(\Gamma)$
with equality exactly for equilateral triangles, where $\Gamma$ is the boundary of~$P$.
The proof of this result was simplified by Balestro, Martini, Nikonorova, and one of the present authors in \cite{BMNN2021},
where the authors determined the self Chebyshev radius for the boundary of an arbitrary triangle.
Moreover, some related problems were considered in detail in \cite{BMNN2021}.
In particular, the maximal possible perimeter for  convex curves and boundaries of convex $n$-gons with a given self Chebyshev radius were found.
\smallskip

As the authors of \cite{BMNN2021} discovered somewhat later, the original problem from \cite{Walter2017},
which was to prove the inequality $L(\Gamma)\geq \pi \cdot \delta(\Gamma)$ for any closed convex curve $\Gamma$ in the Euclidean plane,
had already been solved many years ago in paper \cite{Fal1977} by  K.J.~Falconer
(the results of that paper were formulated in other terms, without using the self Chebyshev radius).
An exposition of the corresponding result can also be found in \cite[Theorem 4.3.2]{MaMoOl}.
It should be noted that the following stronger result is also proved in \cite{Fal1977}.
Instead of convexity, it requires the curve to be rectifiable
and instead of the plane it is in $\mathbb{R}^n$.
{\sl Let $\Gamma$ be a closed rectifiable curve in $\mathbb{R}^n$ {\rm(}with the Euclidean metric{\rm)} such that for every point $x$ on
$\Gamma$ there is a point of $\Gamma$ at distance at least $1$ from $x$. Then $\Gamma$ has length at least $\pi$, this
value being attained if and only if $\Gamma$ bounds a plane convex set of constant width~$1$.}
\medskip

\begin{figure}[t]
\center{\includegraphics[width=0.47\textwidth]{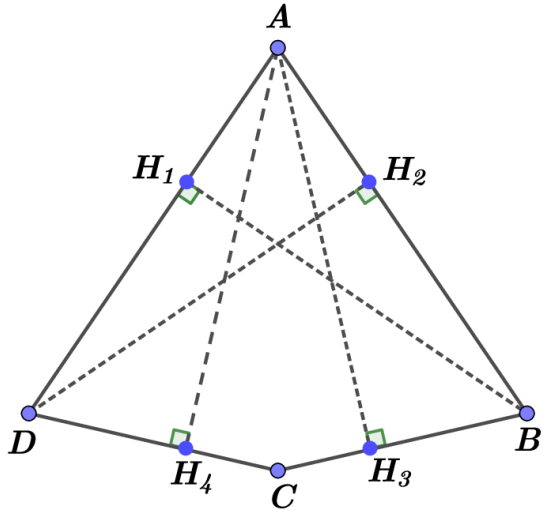}}\\
\caption{A magic kite.}
\label{Fig1}
\end{figure}

This paper is devoted to the proof of Rolf Walter's conjecture in \cite{Walter2017}
on quadrilaterals of minimum perimeter among all quadrilaterals with a given self Chebyshev radius of their boundaries
that is as follows:
$L(\Gamma)\geq \frac{4}{3}\sqrt{2\sqrt{3}+3}\,\cdot \delta(\Gamma)$
for any convex quadrilateral $P\subset \mathbb{R}^2$ with the boundary~$\Gamma$.

Note that this inequality becomes an equality for quadrilaterals $P$ called {\it magic kites} (squares are not extremal in this sense). \label{defmagkite}
This definition is
taken from \cite{Walter2017} and means
convex quadrilaterals which are hypothetically extreme with respect to
the self Chebyshev radius.
Up to similarity, such a quadrilateral could be represented by its vertices, that are as follows (see  Fig. \ref{Fig1}):
$$
(-1, 0), \quad (1, 0), \quad
\left(0, \frac{\sqrt{3}}{3}\sqrt{2\sqrt{3}+3}\right), \quad
\left(0, -\frac{1}{3}\sqrt{2\sqrt{3}-3}\,\right).
$$
The heights drawn from all the vertices to the sides farthest to them have the same length ($H_i$, $i=1,2,3,4$,
denotes the bases of the corresponding heights on the sides of the considered quadrilateral), which coincides with the self Chebyshev radius
of the boundary $\Gamma$ of this magic kite. Note that two such heights emanate from vertex $A$, one such height emanates from vertices $B$ and $D$,
while the heights from vertex $C$ are shorter than the proper Chebyshev radius of the boundary of this polygon.

Note also that  we have $L(\Gamma_\square)=\frac{8}{\sqrt{5}}\,\cdot \delta(\Gamma_\square)$, where $\Gamma_\square$ is a boundary of a square, and
$\frac{8}{\sqrt{5}}>\frac{4}{3}\sqrt{3+2\sqrt{3}}= 3.389946...$.

\medskip
Our main result is the following theorem.

\begin{theorem}\label{t.main}
For any quadrilateral $P$ with the boundary $\Gamma$ on Euclidean plane, we have the inequality
$L(\Gamma)\geq \frac{4}{3}\sqrt{3+2\sqrt{3}}\,\cdot \delta(\Gamma)$, with equality exactly for magic kites.
In other words, a magic kite has the minimal perimeter
among all quadrilaterals with a given self Chebyshev radius of their boundaries.
\end{theorem}

\begin{figure}[t]
\center{\includegraphics[width=0.41\textwidth]{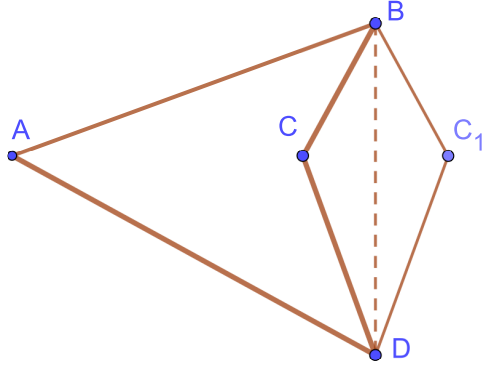}}\\
\caption{The transition from a non-convex quadrilateral to a convex one.}
\label{Fig1.5}
\end{figure}

If a quadrilateral $P$ is not convex, then one can easy to find a convex quadrilateral $P_1$ with the same perimeter and such that
the self Chebyshev radius of the boundary of $P_1$ is not less than the self Chebyshev radius of the boundary of $P$.
Indeed, if the convex hull of a quadrilateral $ABCD$ is the triangle $ABD$, we can consider the point $C_1$ that is symmetric to $C$ with respect the straight line $BD$, see
Fig.~\ref{Fig1.5}. The quadrilateral $ABC_1D$ is convex and has the same perimeter as $ABCD$ has.
For any point $M\in [B,C]\bigcup [C,D]$, denote by $M_1$ the point on $[B,C_1]\bigcup [C_1,D]$ that is symmetric to $M$ with respect the straight line $BD$.
If we take any point $L\in [A,B]\bigcup[A,D]$, then the distance from $L$ to $M_1$ is not less that the distance from   $L$ to $M$.
This observation implies that the self Chebyshev radius of the boundary of $ABCD$ is not greater than the self Chebyshev radius of the boundary of $ABC_1D$.
Therefore, the above theorem must be proved only for convex quadrilaterals.

We use mainly standard methods of Euclidean geometry and analysis to prove this theorem.
Sometimes we have to resort to the help of computer algebra systems.
In what follows we repeatedly use some well known properties of convex polygons.
As a source on the properties of polygons, one can advise, in particular, books \cite{Olymp, HaWa, Pra2006}.
There is also a special book on the properties of quadrilaterals: \cite{AlNel}.
\smallskip

The proof of the main theorem is technically quite difficult. We have to consider many special cases, and in each case we apply different methods.
Despite all our attempts to simplify the proofs, we have not succeeded in doing so.
We would like to believe that someone will find a shorter and more conceptual reasoning. But while this is not known, we decided to share our version of the proof.
\smallskip

For symbolic and numeric calculations, we used the well-established Maple, a general purpose computer algebra system (CAS).
Wherever possible, we have used symbolic calculations.
The Maple program provides a wide range of commands for working with symbolic expressions.
Numerical calculations were used only where it was not possible to use symbolic calculations effectively
and where there was no doubt about the correctness of the approximations
(because there is some margin of magnitude with which to compare the result).
The precision used was controlled by the Digits global variable, which has a default value of 10.
However, it can be set to almost any other value. For more detailed information, see e.~g. \cite{Landau}.
\smallskip

The paper is organized as follows. In Section \ref{sect.0} we consider some  important information on Chebyshev radii and centers for bounded convex sets
in an arbitrary Hilbert space. In Section \ref{sect.1} we prove some important results
on self Chebyshev radii and centers for boundaries of convex polygons on the Euclidean plane.
In Section \ref{sect.2} we describe a general plan of the study and obtain some auxiliary results.
Sections \ref{2.5}, \ref{sect.3}, \ref{sect.4}, and \ref{sect.5} are devoted to the study of important partial cases of the general problem.
We prove our main Theorem \ref{t.main} on the basis of the results obtained in these sections.
In the final section we consider some additional conjectures and outline the prospects for further research on the topic.
The list of some important notations and definitions is given at the end of the paper.

\smallskip

The authors are grateful to Endre Makai, Jr., Horst Martini, and Yusuke Sakane for useful discussions and valuable advices.
We are also grateful for the advices of both of the anonymous reviewers, whose comments and suggestions allowed us to improve the
presentation of this article.

\section{Some general results}\label{sect.0}

All propositions in this section can be found in various sources. We give them here in the form we need and provide brief proofs for the sake of completeness.

Let us consider any Hilbert space $H$ with the inner product $(\cdot, \cdot)$ and the corresponding norm $\|\cdot \|$
(in particular any $\mathbb{R}^n$ with some Euclidean norm).
For any set $M \subset H$, the symbols $\co(M)$, $\bd(M)$, and $\intt(M)$ denote respectively the convex hull, the boundary, and the interior of $M$.

For given $c\in H$ and $r \geq 0$, we consider $S(c,r)=\{x\in H\,|\, \|x-c\|=r\}$,  $U(c,r)=\{x\in H\,|\, \|x-c\|< r\}$,
and  $B(c,r)=\{x\in H\,|\, \|x-c\|\leq r\}$, respectively
the sphere, the open ball, and the closed ball with the center $c$ and radius $r$. For $x,y \in H$, the symbol $[x,y]$ means the closed interval
with the ends $x$ and $y$ on the line through $x$ and $y$.

\begin{prop}\label{prop_2}
Let $M$ be a non-empty  bounded set in $H$, and a function $t:H \rightarrow \mathbb{R}$ such that
$t(x)=\min \{r \,|\, M\subset B(x,r)\}$.
Then $t$ is a convex function and there is a real number $C >0$ such that
$\|x\|-C \leq t(x) \leq \|x\|+C$ for all $x\in H$.
As a corollary, the function $x\mapsto t(x)$ has a unique point of global minimum value in $H$.
\end{prop}

\begin{proof} Let us consider any $x_1, x_2 \in H$ and any $\alpha,\beta >0$, $\alpha+\beta=1$.
Then for every $z\in M$, we have
$$
\|z-(\alpha x_1+\beta x_2)\|=\|\alpha (z-x_1)+\beta (z-x_2)\|\leq \alpha \|z-x_1\|+\beta \|z-x_2\|\leq \alpha t(x_1)+\beta t(x_2).
$$
Therefore,
$t(\alpha x_1+\beta x_2)\leq \alpha t(x_1)+\beta t(x_2)$,
hence, the function $t$ is convex. Since $M$ is bounded, then there is a real number $C>0$ such that $\|z\|\leq C$ for all $z\in M$.
Consequently,
$$
\|x\|-C\leq \|x\|-\|z\|\leq \|x-z\|\leq  \|x\|+\|z\|\leq \|x\|+C
$$
for every $x\in H$ and every $z\in M$.
This means that $t$ is coercive, hence, $t$ has a point of global minimum value, see e.~g. Corollary 11.30 in \cite{BaCo}.

Now, we suppose that there are points $x_2\neq x_1$, where $t$ attains its minimum value, say $m_t$.
Therefore, $M \subset B(x_1,m_t)$, $M \subset B(x_2,m_t)$, and $M \subset B(x_1,m_t)\cap B(x_2,m_t)$.
It is clear that $B(x_1,m_t)\cap B(x_2,m_t) \subset B\left( \frac{x_1+x_2}{2}, \sqrt{m_t^2-\|x_2-x_1\|^2/4}\right)$.
Indeed, for any $z\in B(x_1,m_t)\cap B(x_2,m_t)$, we have
$$
4\|z-(x_1+x_2)/2\|^2+\|x_2-x_1\|^2=2\|z-x_1\|^2+2\|z-x_2\|^2\leq 4m_t^2.
$$
Hence,
$t\left( \frac{x_1+x_2}{2}\right)<m_t$, that is impossible.
\end{proof}

\begin{prop}\label{prop_3}
Let $M$ be a non-empty  convex bounded closed set in $H$, and a function $t:H \rightarrow \mathbb{R}$ such that
$t(x)=\min \{r \,|\, M\subset B(x,r)\}$.
Then for any $y\in H \setminus M$, there is a point $z\in \bd (M)$ such that $t(z) \leq t(y)$. Moreover, if $H$ is finite-dimensional, then  $t(z) <t(y)$.
\end{prop}

\begin{proof} Let $z \in M$ be the closest point of $M$ to the point $y$. It is clear that $M \subset \{x \in H\, |\, \|x-z\|<\|x-y\|\}$.
Hence if $M \subset B(y,r)$ for some $r>0$, then $M \subset B(z,r)$. It implies that $t(z)\leq t(y)$.
If $H$ is finite-dimensional, then $M$ is compact, hence there is $u\in M$ such that $\|u-z\|\geq\|x-z\|$ for all $x \in M$.
Therefore, $t(z)=\|u-z\|<\|u-y\|\leq t(y)$.
\end{proof}
\smallskip

For any non-empty bounded closed set $M$ in $H$, a point $y\in H$ ($y\in M$) with minimal values of the function $x \mapsto t(x)$ on $H$ (respectively, on $M$)
is the Chebyshev center (respectively, the self Chebyshev center) for $M$. Two above propositions imply

\begin{corollary}\label{corol_cent1}
Let $M$ be a non-empty  convex bounded closed set in $H$, then it has a unique Chebyshev center, which by necessity lies in $M$.
\end{corollary}

Obviously $\bd(M)$ has the same Chebyshev center as $M$ from the above corollary. It is not the case for self Chebyshev centers, of course.
\smallskip

\begin{prop}\label{prop_4}
Suppose that $M_1$ and $M_2$ are convex bounded closed set in $H$ and $M_1 \subset M_2$. If $r_i$ is the self Chebyshev radius of the set $\bd(M_i)$, $i=1,2$,
then $r_1\leq r_2$.
\end{prop}

\begin{proof} For any $\varepsilon >0$, there is a point $y_{\varepsilon}\in \bd(M_2)$ such that
$\bd(M_2)\subset M_2 \subset B(y_{\varepsilon}, r_2+\varepsilon)$. Since $M_1 \subset M_2$, we have
$\bd(M_1)\subset M_1 \subset M_2 \subset B(y_{\varepsilon}, r_2+\varepsilon)$.
If $y_{\varepsilon}\in \bd(M_1)$, then we put $z_{\varepsilon}:=y_{\varepsilon}$. Now, suppose that $y_{\varepsilon}\not\in \bd(M_1)$ and consider the function
$t_1:H \rightarrow \mathbb{R}$ such that
$t_1(x)=\min \{r \,|\, M_1\subset B(x,r)\}$.
We see that $t_1(y_{\varepsilon})\leq r_2+\varepsilon$. By Proposition \ref{prop_3}, there is
$z_{\varepsilon}\in \bd (M_1)$ such that $t_1(z_{\varepsilon}) \leq t_1(y_{\varepsilon})$.
In both cases we have $t_1(z_{\varepsilon}) \leq t_1(y_{\varepsilon}) \leq r_2+\varepsilon$. Since $\varepsilon >0$ could be as small as we want, we get
$r_1 = \inf \{ t_1(z)\,|\, z \in \bd(M_2)\}\leq r_2$.
\end{proof}
\smallskip

\begin{remark}
Let us consider one more proof of Proposition \ref{prop_4} (we put  $\Gamma_1=\bd(M_1)$ and $\Gamma_2=\bd(M_2)$).
We know that $\delta (\Gamma_2)=\inf\limits_{x\in \Gamma_2} \min\{r\,|\,\Gamma_2\subset M_2\subset B(x,r)\}$, $\Gamma_1 \subset M_1 \subset M_2$,
and for any $x \in \Gamma_2$ there is $u \in \Gamma_1$ such that $\min\{r\,|\, \Gamma_1\subset B(u,r)\} \leq \min\{r\,|\, \Gamma_1\subset B(x,r)\}$ by
Proposition \ref{prop_3}. Therefore,
$$
\delta (\Gamma_1)=\inf\limits_{u\in \Gamma_1} \min\{r\,|\,\Gamma_1 \subset B(u,r)\}
\leq \inf\limits_{x\in \Gamma_2} \min\{r\,|\,\Gamma_1 \subset B(x,r)\}\leq \delta (\Gamma_2),
$$
that proves the proposition. It should be noted that Proposition \ref{prop_4} holds for non-convex $M_2$, but the proof is too
complicated for non-convex case.
\end{remark}

\section{Some auxiliary results for boundaries of convex polygons}\label{sect.1}

In what follows, the symbol $\|A\|$ means the cardinality of a given set $A$.
We identify the Euclidean plane with $\mathbb{R}^2$ supplied with the standard Euclidean metric~$d$, where $d(x,y)=\sqrt{(x_1-y_1)^2+(x_2-y_2)^2}$.

We call $\Gamma$ {\it a convex curve} if it is the boundary of some convex compact set in the Euclidean plane $\mathbb{R}^2$.
Important examples of convex curves are {\it  convex closed polygonal chains},
i.e. {\it boundaries of convex polygons}.
A polygon $P$ is called an {\it $n$-gon} if it has exactly $n$ vertices.
For $n=2$ we get line segments. The perimeter $L(P)$ of any $2$-gon
is defined as the double length of the line segment $P$. Such definition is justified,
since it leads to the continuity of the perimeter as a functional on the set of convex polygons with respect to the Hausdorff distance.

Let $\Gamma \subset \mathbb{R}^2$ be a compact set (in particular, a convex curve). We define the function $\mu:\Gamma \rightarrow \mathbb{R}$ as follows:
\begin{equation}\label{mufunk}
\mu (x)=\max\limits_{y \in \Gamma} d(x,y).
\end{equation}

For any polygon $P$ and any given point~$x\in P$,  $\max\limits_{y \in P} d(x,y)$ is achieved at a vertex of $P$ (see, e.g., Lemma 4.1 in  \cite{Walter2017}),
i.e., $\mu(x)$ is equal to the maximal distance from $x$ to vertices of $P$.
Note also that the diameter $\diam(P)=\max\limits_{x,y \in P} d(x,y)$  of a polygon $P$ always connects two vertices.

\smallskip

The following property (monotonicity of the perimeter) of convex curves is well known (see, e.g., \cite[\S 7]{BoFe1987}).

\begin{prop}\label{monotper}
If a convex curve $\Gamma_1$ is inside another convex curve $\Gamma_2$ in the Euclidean plane,
then the length of $\Gamma_1$ is less than or equal to the length of $\Gamma_2$, and equality holds if and only if $\Gamma_1=\Gamma_2$.
\end{prop}

The following simple result is very useful.

\begin{prop}[\cite{BMNN2021}]\label{semcircle}
Let $\Gamma$ be a convex curve in $\mathbb{R}^2$ such that $\Gamma$ contains the line segment $[p,q]$ with the property
$\Gamma \subset \left\{ x\in \mathbb{R}^2\,|\, d(x,o)\leq \frac{1}{2}d(p,q) \right\}$, where $o$ is the  midpoint of $[p,q]$.
Then $\delta(\Gamma)= \frac{1}{2}d(p,q)=\mu(o)$. Moreover, $o$ is a unique self Chebyshev center for $\Gamma$.
\end{prop}

\medskip
Let us consider $n \geq 4$ and let $\Gamma$ be a boundary of a convex $n$-gon $P$.
We will show how to estimate $\delta(\Gamma)$. Let us denote $A_i$, $1\leq i \leq n$, consecutive vertices of a polygonal line $\Gamma$ (vertices of the polygon $P$),
assuming that indices are taken $\!\!\!\mod n$.
The direction of traversing the curve $\Gamma$
in the direction $A_1, A_2,...,$ is called {\it positive}, and the opposite direction to it is called {\it negative}.
\smallskip

For a given point $A\in \Gamma$, we consider the set
\begin{equation}\label{eq.fath1}
{\mathcal F}(A)=\{B\in \Gamma\, \mid \, d(A,B)=\mu(A)\},
\end{equation}
see \eqref{mufunk}. It is clear that ${\mathcal F}(A)$ contains only vertices of $\Gamma$.

\begin{figure}[t]
\begin{minipage}[h]{0.35\textwidth}
\center{\includegraphics[width=0.89\textwidth]{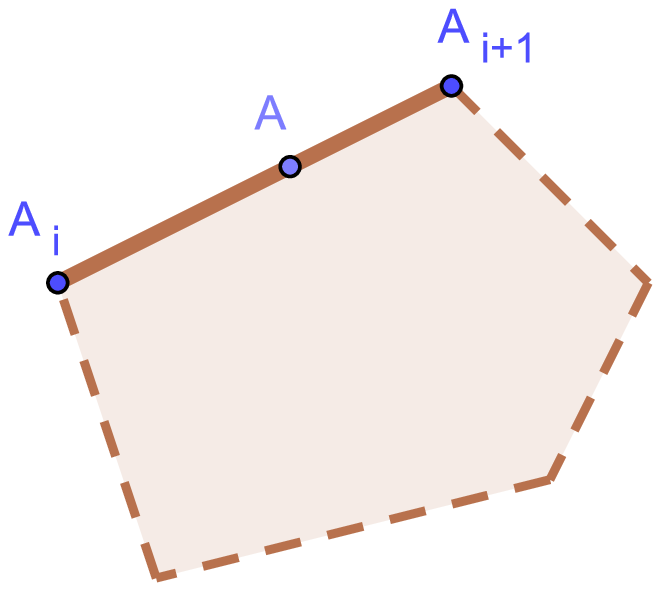}} \\ a)                     \\
\end{minipage}
\quad\quad\quad\quad
\begin{minipage}[h]{0.42\textwidth}
\center{\includegraphics[width=0.89\textwidth]{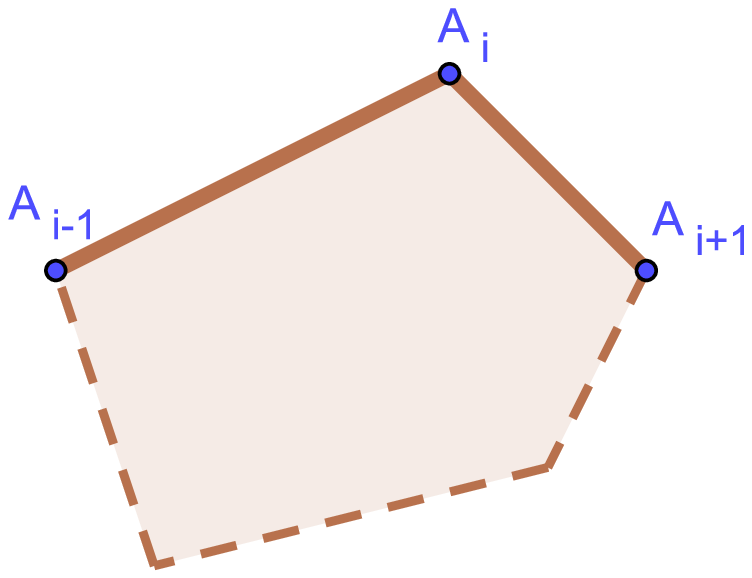}}   \\ b) \\
\end{minipage}
\caption{a) $N(A)$ for an interior point of $[A_i,A_{i+1}]$; b) $N(A)$ for a vertex $A=A_i$.}
\label{N(A)}
\end{figure}

For a given point $A\in \Gamma$, we denote by $T_+(A)$ (respectively, $T_-(A)$) the limit $\lim\limits_{B \to A} \frac{\overrightarrow{AB}}{d(A,B)}$,
where the points $B\in \Gamma$ tend to $A$ via the negative (respectively, the positive) direction.
These vectors have length $1$ and are called one-sided tangent vectors at the point $A$.
If $A$ is an interior point of some segment $[A_i,A_{i+1}]$, then $T_+(A)=- T_-(A)$.

\bigskip

For a given point $A\in [A_i,A_{i+1}] \subset \Gamma$ we put $N(A):=[A_i,A_{i+1}]$ if $A$ is in the interior of the segment $[A_i,A_{i+1}]$ and
$N(A):=[A_{i-1},A_{i}]\cup [A_i,A_{i+1}]$ if $A=A_i$ for some $i \in \mathbb{Z}_n$, see Fig.~\ref{N(A)}.

\begin{prop}\label{locmin_mu1}
In the above notations, for a given point $A\in  \Gamma$, the following conditions are equivalent:

1) The function $\mu$ {\rm(}see \eqref{mufunk}{\rm)} attains its minimal value on the set $N(A)$ at the point~$A$;

2) The point $A$ is the point of local minimum of the function $\mu: \Gamma \rightarrow \mathbb{R}$;

3) For any vector $\vec{v}\in \{T_+(A), T_-(A)\}$, there is a point $B \in {\mathcal F}(A)\subset \Gamma$ such that $(\vec{v}, \overrightarrow{AB})\leq 0$.
\end{prop}

\begin{proof} Obviously, 1) implies 2).
Let us prove $2)\Rightarrow 3)$. Suppose that 2) holds but 3) does not hold. Choose a vector $\vec{v}\in \{T_+(A), T_-(A)\}$, such that $(\vec{v}, \overrightarrow{AB})> 0$ for any $B \in {\mathcal F}(A)$,
and consider the points $A(\varepsilon)=A+\varepsilon \vec{v} \in \Gamma$.
Since $(\vec{v}, \overrightarrow{AB})> 0$, then $d(A(\varepsilon), B)< d(A, B)=\mu(A)$
for sufficiently small $\varepsilon >0$.
If $C$ is any vertex of $P$ that is not in ${\mathcal F}(A)$, then $d(A,C)<\mu(A)$, hence, $d(A(\varepsilon), C)< \mu(A)$ for sufficiently small $\varepsilon >0$.
Therefore, $d(A(\varepsilon), D)< \mu(A)$ for any vertex $D$ of $P$, that implies $\mu(A(\varepsilon))< \mu(A)$ for sufficiently small $\varepsilon >0$.
Hence, we get the contradiction with the condition 2) and, therefore, 2) implies 3).

Now, let us prove $3)\Rightarrow 1)$. We suppose that 3) holds.
Let us choose any vector $\vec{v}\in \{T_+(A), T_-(A)\}$
and consider the point $A(\varepsilon)=A+\varepsilon \vec{v}$ for any $\varepsilon >0$.
By the condition 3), there is $B \in {\mathcal F}(A)\subset \Gamma$ such that $(\vec{v}, \overrightarrow{AB})\leq 0$.
It implies $d(A(\varepsilon), B)> d(A, B)=\mu(A)$. Hence, $\mu(A(\varepsilon))\geq d(A(\varepsilon), B) > \mu(A)$ for all $\varepsilon >0$.
Since we can take both $T_+(A)$ and $T_-(A)$ as $\vec{v}$, we get that
$A$ is the minimal value point of $\mu$ on the set $N(A)$, hence, the condition 1) holds.
\end{proof}

\begin{remark}
This proposition also follows from Proposition \ref{prop_2}. Indeed, the restriction of a convex function to any line segment is also a convex function.
For every convex function, any point of local minimum is also a point of global minimum.
\end{remark}
\smallskip

Using the equivalence of conditions 1) and 3) in the previous proposition, we obtain the following

\begin{corollary}\label{cor.angle1}
If the angle  at the vertex $A_i$ of the $n$-gon $P$ is less than $\pi/2$, then $A_i$ is not a minimal value point of the function $\mu$ on the set
$N(A_i)=[A_{i-1},A_{i}]\cup [A_i,A_{i+1}]$.
\end{corollary}

For a given point $A\in \Gamma$ we have $d(A,A_i)=d(A,A_j)$ if and only if $A$ is on the perpendicular bisector of $A_i$ and $A_j$.
Hence, we have only finite numbers of points $A\in \Gamma$ such that the set ${\mathcal F}(A)$ has more that one point.
Proposition \ref{locmin_mu1} implies

\begin{corollary}\label{cor.far1}
Suppose that the point $A\in \Gamma$  is such that the set ${\mathcal F}(A)$ has only one point. If $A$ is the local minimum point of the function $\mu$, then
$A$ is an interior point of some segment $[A_i,A_{i+1}]$ and $\overrightarrow{A_iA_{i+1}}$ is orthogonal to $\overrightarrow{AB}$, where $\{B\}={\mathcal F}(A)$.
\end{corollary}

\begin{prop}\label{pr.simple1}
For a given $n$-gon $P$ with $\Gamma =\bd(P)$, we consider the set
$$
G_{\operatorname{locmin}}=\{x \in \Gamma\,|\, x\mbox{ is a local minimum point of the function }\mu\}.\label{deflocmin}
$$
If there is $i \in \mathbb{Z}_n$ such that $A_i \in G_{\operatorname{locmin}}$ and $\mu(A_i)= \max \left\{d(A_i, A_{i-1}), d(A_i, A_{i+1})\right\}$,
then $L(\Gamma)\geq (2+\sqrt{2})\cdot \delta(\Gamma)$.
\end{prop}

\begin{proof}
By Corollary \ref{cor.far1} we see that the set ${\mathcal F}(A_i)$ has at least two elements.
Let us consider the set $J=\{j \in \mathbb{Z}_n\,|\, A_j \in F(A_i)\}=\{j \in \mathbb{Z}_n\,|\,d(A_i, A_j)=\mu(A_i)\}$.

Now, we suppose that for every $k,l \in J$ we have $\angle A_k A_i A_l < \pi/2$.
Let us fix $j\in J$ such that $j=i-1$ or $j=i+1$ (it is possible by the assumptions) and consider the points
$A_i(\varepsilon)=\varepsilon A_j +(1-\varepsilon) A_i\in \Gamma$ for $\varepsilon \in (0,1)$.
It is easy to see that
$$
d(A_i(\varepsilon), A_k)<d(A_i,A_k)=\mu(A_i)
$$
for all $k \in J$ and for sufficiently small $\varepsilon >0$ (due to the inequality $\angle A_j A_i A_k < \pi/2$).
Moreover, for all $k \in \mathbb{Z}_n \setminus J$, $k \neq i$, we have $d(A_i(\varepsilon), A_k)<\mu(A_i)$
for sufficiently small $\varepsilon >0$ (due to the inequality $d(A_i,A_k)<\mu(A_i)$).
So we have $\mu(A_i(\varepsilon)) <\mu(A_i)$ for all sufficiently small $\varepsilon >0$, that is contradicts to $A_i \in G_{\operatorname{locmin}}$.

Therefore, there are $k,l \in J$ such that $\angle A_k A_i A_l \geq \pi/2$. It means that $P$ contains the triangle $\triangle A_k A_i A_l$,
which has the perimeter at least
$(2+\sqrt{2})\cdot \delta(\Gamma)$ (we have this minimal value for $\angle A_k A_i A_l = \pi/2$).
Hence, $L(\Gamma)\geq (2+\sqrt{2})\cdot \delta(\Gamma)$.
\end{proof}

\medskip

In what follows, we call an $n$-gon $P$ {\it extremal} if it has the smallest perimeter among all convex $n$-gons whose boundaries have a given
value of the self Chebyshev radius \eqref{chebir1} (i.~e. $L(\Gamma)/\delta(\Gamma)$ has a minimal possible value, where $\Gamma = \bd(P)$). \label{defextrem}

\begin{remark}
The existence of an extremal $n$-gon for every $n \geq 3$ can be proved using a limiting argument.
The limit polygon cannot be an $m$-gon with  $m<n$, since for any $m$-gon  we can easily construct an $n$-gon with the same self Chebyshev radius for its boundary and with
a bigger perimeter. Indeed, it suffices to replace one side of a given $m$-gon to a suitable almost straight polygonal chain of $n-m+1$ line segments.
\end{remark}

The self Chebyshev radius $\delta(\Gamma)$ of $\Gamma$ coincides with the minimal value of the function~$\mu$ (see \eqref{mufunk}).


\smallskip

For fixed $n$-gon $P$ with $\Gamma=\bd(P)$, we consider the function $t:\mathbb{R}^2 \rightarrow \mathbb{R}$ such that
\begin{equation}\label{eq_fuct1}
t(x)=\min \{r \,|\, \Gamma \subset B(x,r)\}.
\end{equation}
We know that this function is convex with an unique point of the global minimum $x_0$ (Proposition \ref{prop_2}).
The point $x_0$ is the Chebyshev center of $\Gamma$ (and $P$), $t(x_0)$ is the Chebyshev radius $r_{\Gamma}$ of $\Gamma$ (and $P$).
We know also that $x_0 \in P$ by Corollary \ref{corol_cent1}.

Since the function $t$ is convex, for any $r >r_{\Gamma}$, the set $U_r:=\{x\in \mathbb{R}^2\,|\, t(x)\leq r\}$ is a convex set, and
$L_r:=\{x\in \mathbb{R}^2\,|\, t(x)=r\}$ is a closed convex curve.

\begin{lemma} \label{le_loc1}
For every $r >r_{\Gamma}$, the convex curve $L_r$ is the union of finite circular arcs of radius $r$ with the center at some vertices of $\Gamma$.
A common point of two arcs with different centers $A_i$ and $A_j$ is on the same distance from these two points.
\end{lemma}

\begin{proof}
If $x \in L_r$, then $\Gamma \subset B(x,r)$ and there is a point $z \in \Gamma$ such that $d(x,z)=r$. Obviously, $z$ is a vertex of $\Gamma$
(and $P$). Since we have only finite number of the vertices, we get the lemma.
\end{proof}

It is clear that
\begin{equation}\label{eq_fuct2}
\delta(\Gamma)=\min \{r \in [r_{\Gamma},\infty)\,|\, \Gamma \cap L_r \neq \emptyset\}.
\end{equation}
Hence, we have $U_{\delta(\Gamma)}\subset P$ and
\begin{equation}\label{eq_fuct3}
G_{\min}:=L_{\delta(\Gamma)}\cap \Gamma =U_{\delta(\Gamma)}\cap \Gamma\neq \emptyset.
\end{equation}

If $z \in G_{\min}$ and $z$ is a smooth point of $L_{\delta(\Gamma)}$, then $z$ is necessarily in the interior of some segment $[A_i, A_{i+1}]\subset \Gamma$.
If $z \in G_{\min}$ and $z$ is a non-smooth point (an intersection point of two circular arcs with different centers) of $L_{\delta(\Gamma)}$,
then $z$ could be some vertex $A_i$ (in this case $\angle A_{i-1}A_iA_{i+1} \geq \pi/2$ by Corollary \ref{cor.angle1}).
\smallskip

It is clear that the set $G_{\min} \cap [A_i, A_{i+1}]$ is either empty or has exactly one point for any line segment $[A_i, A_{i+1}]\subset \Gamma$.
\bigskip

{\bf Let us suppose that $P$ is extremal}. Fix the vertex $A_i$ and consider two variations of $\Gamma$: $\Gamma_{\varepsilon}^+$ and $\Gamma_{\varepsilon}^-$,
where $\varepsilon \in (0,1)$.
This polygonal lines have the same vertices $A_j$ as $\Gamma$ has with one exception: instead of $A_i$ we take respectively
$A_i^+({\varepsilon})=\varepsilon A_{i+1} +(1-\varepsilon)A_i$ and $A_i^-({\varepsilon})=\varepsilon A_{i-1} +(1-\varepsilon)A_i$.
Since $A_i^+({\varepsilon})\in [A_i,A_{i+1}]$ and $A_i^-({\varepsilon})\in [A_{i-1},A_{i}]$, then $L(\Gamma_{\varepsilon}^+)<L(\Gamma)$ and
$L(\Gamma_{\varepsilon}^-)<L(\Gamma)$ for all $\varepsilon \in (0,1)$. Since $P$ is extremal, then $\delta(\Gamma_{\varepsilon}^+)<\delta(\Gamma)$
and $\delta(\Gamma_{\varepsilon}^-)<\delta(\Gamma)$ (by Proposition \ref{prop_4} we know that $\delta(\Gamma_{\varepsilon}^+)\leq \delta(\Gamma)$
and $\delta(\Gamma_{\varepsilon}^-)\leq \delta(\Gamma)$).

Now, we deal with the {\bf variation $\Gamma_{\varepsilon}^-$}.  There is a point $z_{\varepsilon} \in \Gamma_{\varepsilon}^-$ such that
$d(z_{\varepsilon}, x)\leq \delta(\Gamma_{\varepsilon}^-)$ for any $x\in \Gamma_{\varepsilon}^-$, in particular, for all vertices $A_j$, $j\neq i$.
Moreover, $d(z_{\varepsilon}, A_i^-({\varepsilon}))\leq \delta(\Gamma_{\varepsilon}^-)$.
If we can choose $\varepsilon$ arbitrarily close to zero and such that $z_{\varepsilon} \in [A_i^-({\varepsilon}),A_{i+1}]$, then we get that
$[A_i, A_{i+1}]\cap G_{\min} \neq \emptyset$.

If $z_{\varepsilon} \not\in [A_i^-({\varepsilon}),A_{i+1}]$, then $z_{\varepsilon} \in \Gamma$, hence, $d(z_{\varepsilon}, A_i)>\delta(\Gamma_{\varepsilon}^-)$
(otherwise $\delta(\Gamma_{\varepsilon}^-)\geq \delta(\Gamma)$).
If we can choose $\varepsilon$ arbitrarily close to zero and such that $z_{\varepsilon} \in \Gamma$, then we get some $z\in G_{\min}$ such that $d(z,A_i)=\delta(\Gamma)$.
Indeed, if $z_{\varepsilon} \to z$ as $\varepsilon \to 0$, then
$d(z_{\varepsilon}, A_i^-({\varepsilon}))\leq \delta(\Gamma_{\varepsilon}^-)< d(z_{\varepsilon}, A_i)$, $ A_i^-({\varepsilon}) \to A_i$ and
$\delta(\Gamma_{\varepsilon}^-) \to \delta(\Gamma)$ as $\varepsilon \to 0$.

The {\bf variation $\Gamma_{\varepsilon}^+$} could be considered similarly. Hence we get the following

\begin{prop}\label{pr_extr1}
Suppose that an $n$-gon $P$ is extremal.
If $[A_i, A_{i+1}]\cap G_{\min} = \emptyset$ for some $i \in \mathbb{Z}_n$, then there are points
$z_1,z_2 \in G_{\min}$ such that $d(z_1, A_i)=d(z_2, A_{i+1})=\delta(\Gamma)$.
\end{prop}

\section{A plan how to find extremal quadrilaterals and some preliminary considerations}\label{sect.2}

The main idea is to determine all extremal $4$-gons $P$ and to prove that this set coincides with magic kites.
We use the notations from the previous sections.
\medskip

Since $2+\sqrt{2}=3.414213...> \frac{4}{3}\sqrt{2\sqrt{3}+3}= 3.389946...$, then
we get the following simple consequence of Corollary \ref{cor.far1} and Proposition \ref{pr.simple1}.

\begin{corollary}\label{cor.simple1}
If there is $i \in \mathbb{Z}_4$ such that $A_i \in G_{\min}$,
then $L(\Gamma)\geq (2+\sqrt{2})\cdot \delta(\Gamma)$.
Consequently, for any extremal $4$-gon $P$, we have $A_i \not\in G_{\min}$ for all $i \in \mathbb{Z}_4$.
\end{corollary}

\begin{prop}\label{pr_extr2}
Suppose that a quadrilateral $P$ is extremal. If $A\in G_{\min}$ and $A\in [A_i,A_{i+1}]$, then
$A_i, A_{i+1} \not\in {\mathcal F}(A)$. In particular, ${\mathcal F}(A)\subset \{A_{i+2}, A_{i+3}\}$.
\end{prop}

\begin{proof}
By Corollary \ref{cor.simple1} we see that $A \not\in \{A_i, A_{i+1}\}$. Suppose that $A_i \in {\mathcal F}(A)$, i.~e. $d(A,A_i)= \delta(\Gamma)$.
By
Proposition \ref{locmin_mu1}, for the vector
$\vec{v}=\frac{1}{\|\overrightarrow{AA_{i}}\|}\overrightarrow{AA_{i}}\in \{T_+(A), T_-(A)\}$, there is a point $B \in {\mathcal F}(A)\subset \Gamma$
such that $(\vec{v}, \overrightarrow{AB})\leq 0$. Hence, $\angle A_i A B \geq \pi/2$ and $d(A,A_i)=d(A,B)=\delta(\Gamma)$.
Since the triangle $A_iAB$ is a subset of the quadrilateral $P$ and has the perimeter $\geq (2+\sqrt{2})\cdot \delta(\Gamma)$, then
$L(\Gamma)\geq (2+\sqrt{2})\cdot \delta(\Gamma)$. This means that $P$ is not extremal. Therefore, $A_i \not\in {\mathcal F}(A)$.
Similar arguments proves $A_{i+1} \not\in {\mathcal F}(A)$.
\end{proof}
\medskip

We are going to prove that
$L(\Gamma)\geq \frac{4}{3}\sqrt{2\sqrt{3}+3}\,\cdot \delta(\Gamma)$ for any extremal quadrilateral~$P$.
For this goal, we consider some auxiliary problems.
\smallskip

\begin{figure}[t]
\begin{minipage}[h]{0.32\linewidth}
\center{\includegraphics[width=1\linewidth]{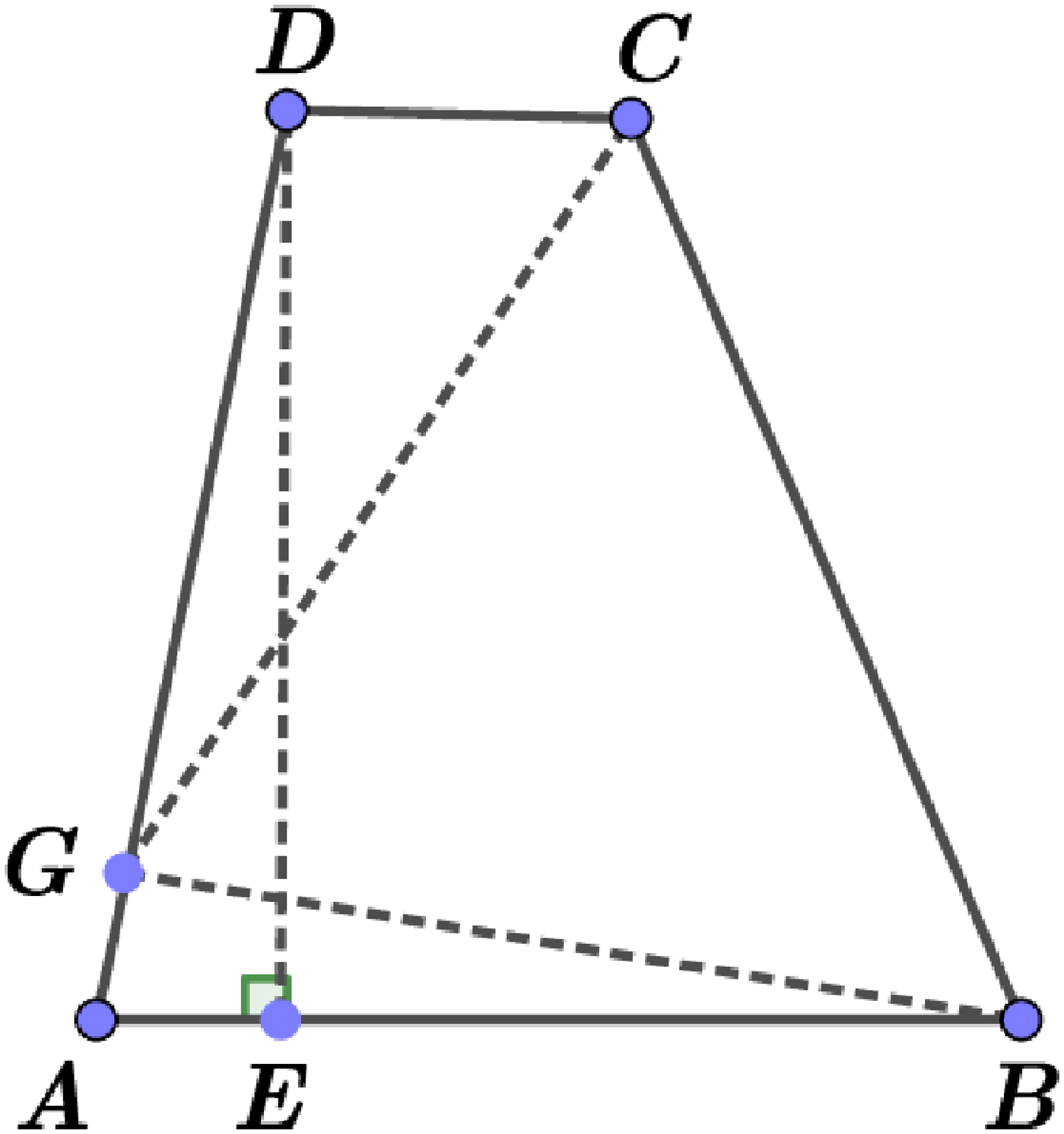}} a) \\
\end{minipage}
\quad\!\!\!\!\!
\begin{minipage}[h]{0.33\linewidth}
\center{\includegraphics[width=1\linewidth]{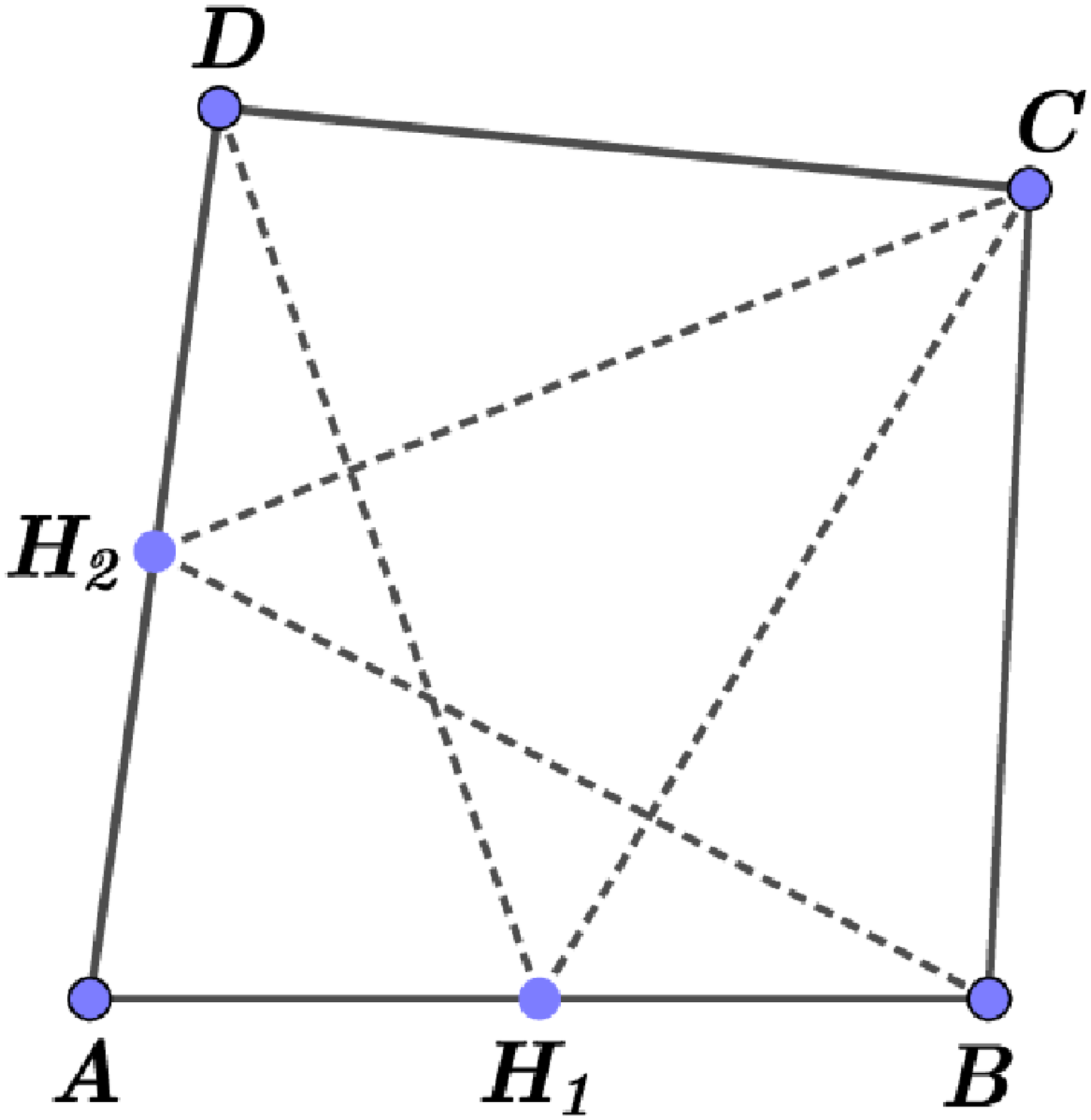}} \\b)
\end{minipage}
\quad
\begin{minipage}[h]{0.28\linewidth}
\center{\includegraphics[width=1\linewidth]{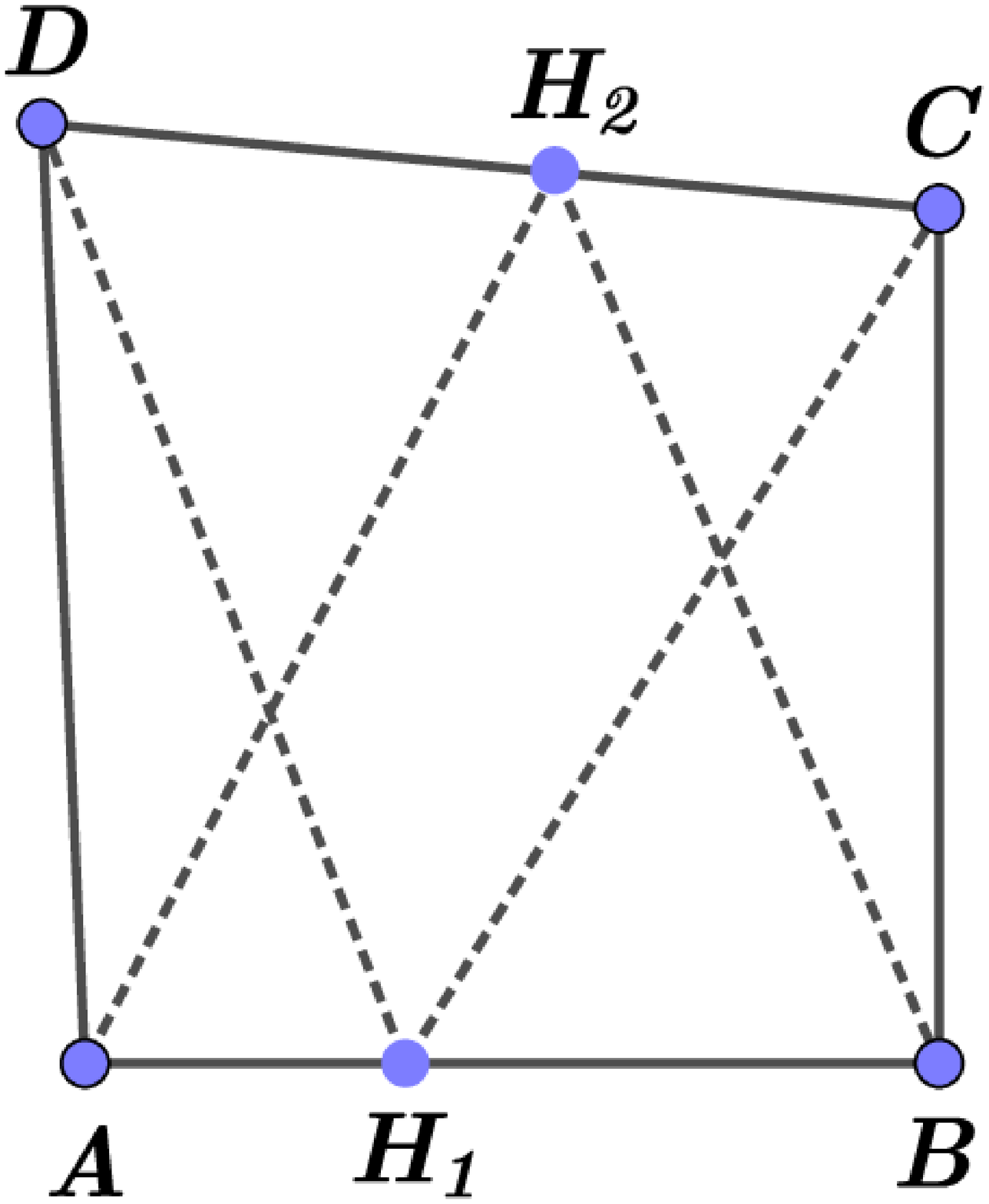}} c) \\
\end{minipage}
\vfill
\begin{minipage}[h]{0.34\linewidth}
\center{\includegraphics[width=1\linewidth]{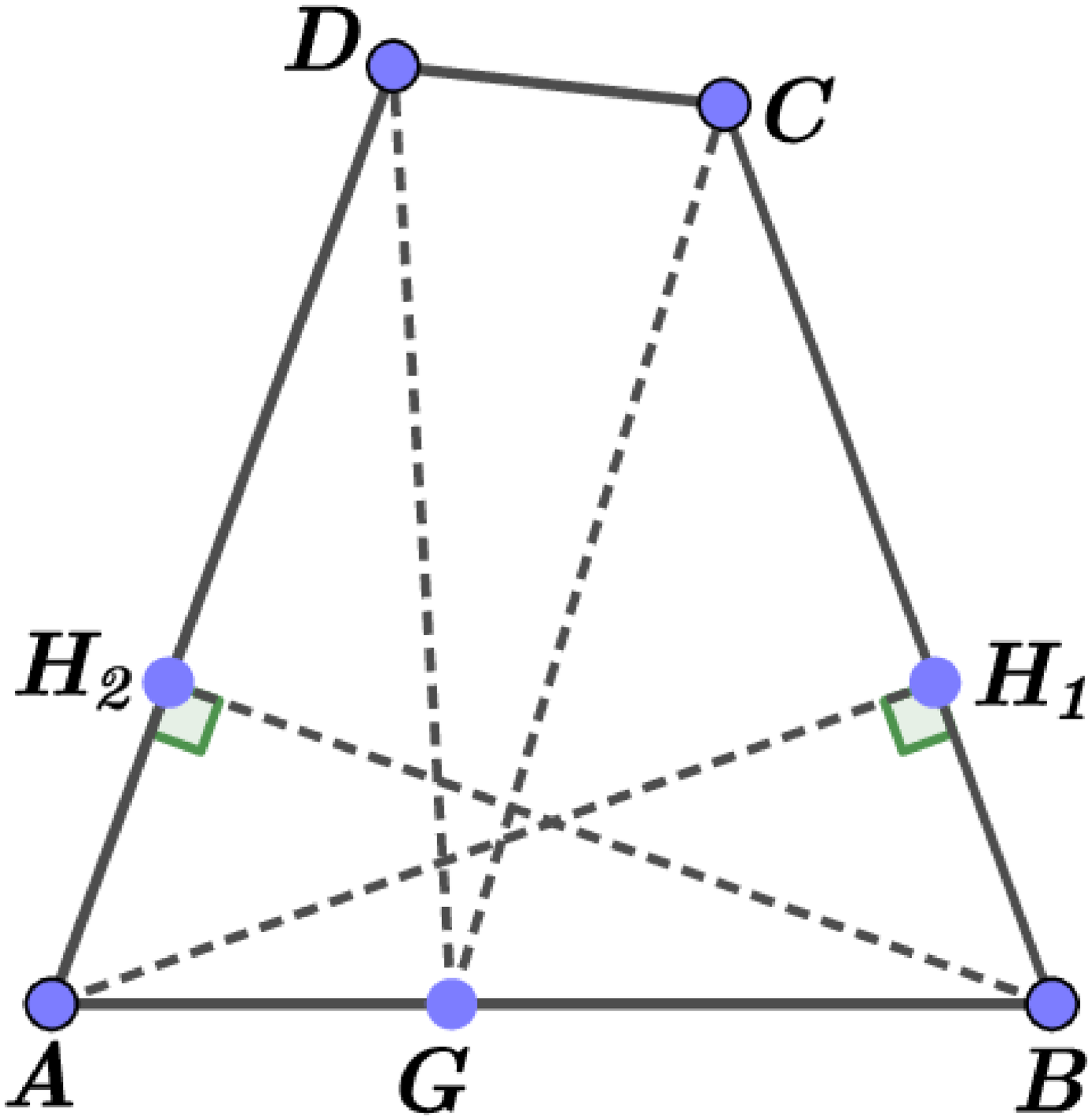}} d) \\
\end{minipage}
\quad
\begin{minipage}[h]{0.35\linewidth}
\center{\includegraphics[width=1\linewidth]{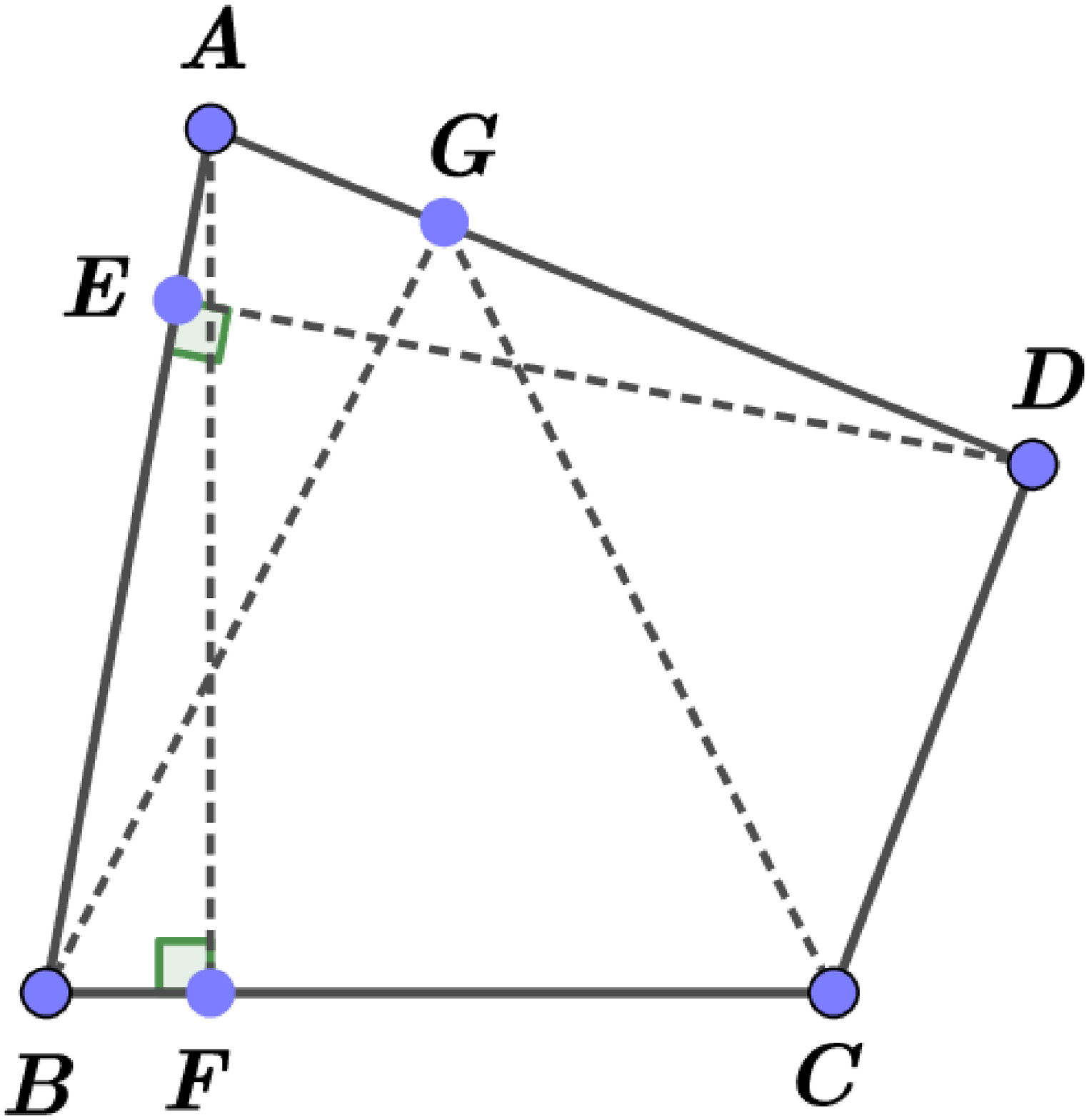}} e) \\
\end{minipage}
\vfill
\begin{minipage}[h]{0.35\linewidth}
\center{\includegraphics[width=1\linewidth]{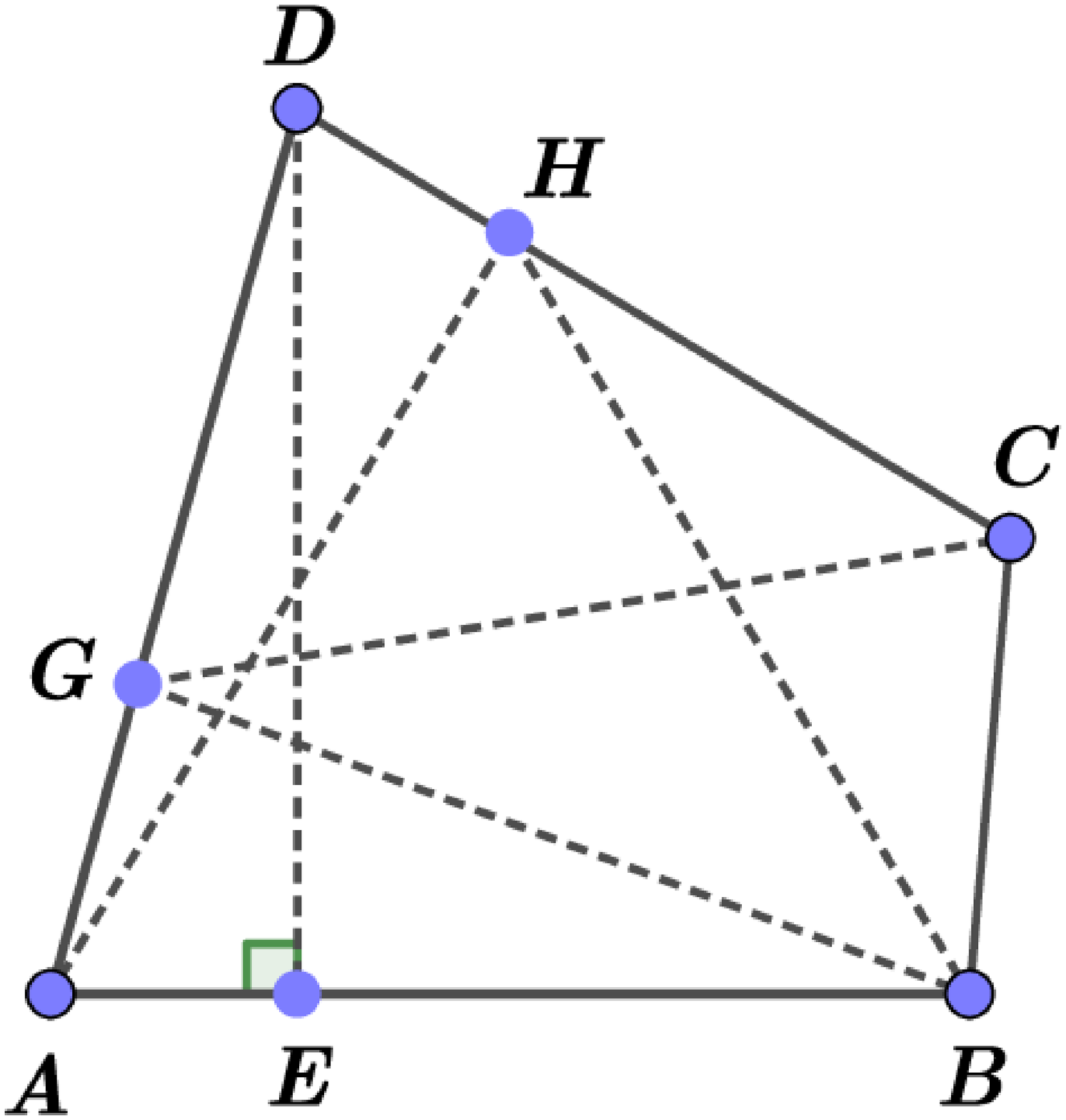}} f) \\
\end{minipage}
\quad
\begin{minipage}[h]{0.35\linewidth}
\center{\includegraphics[width=1\linewidth]{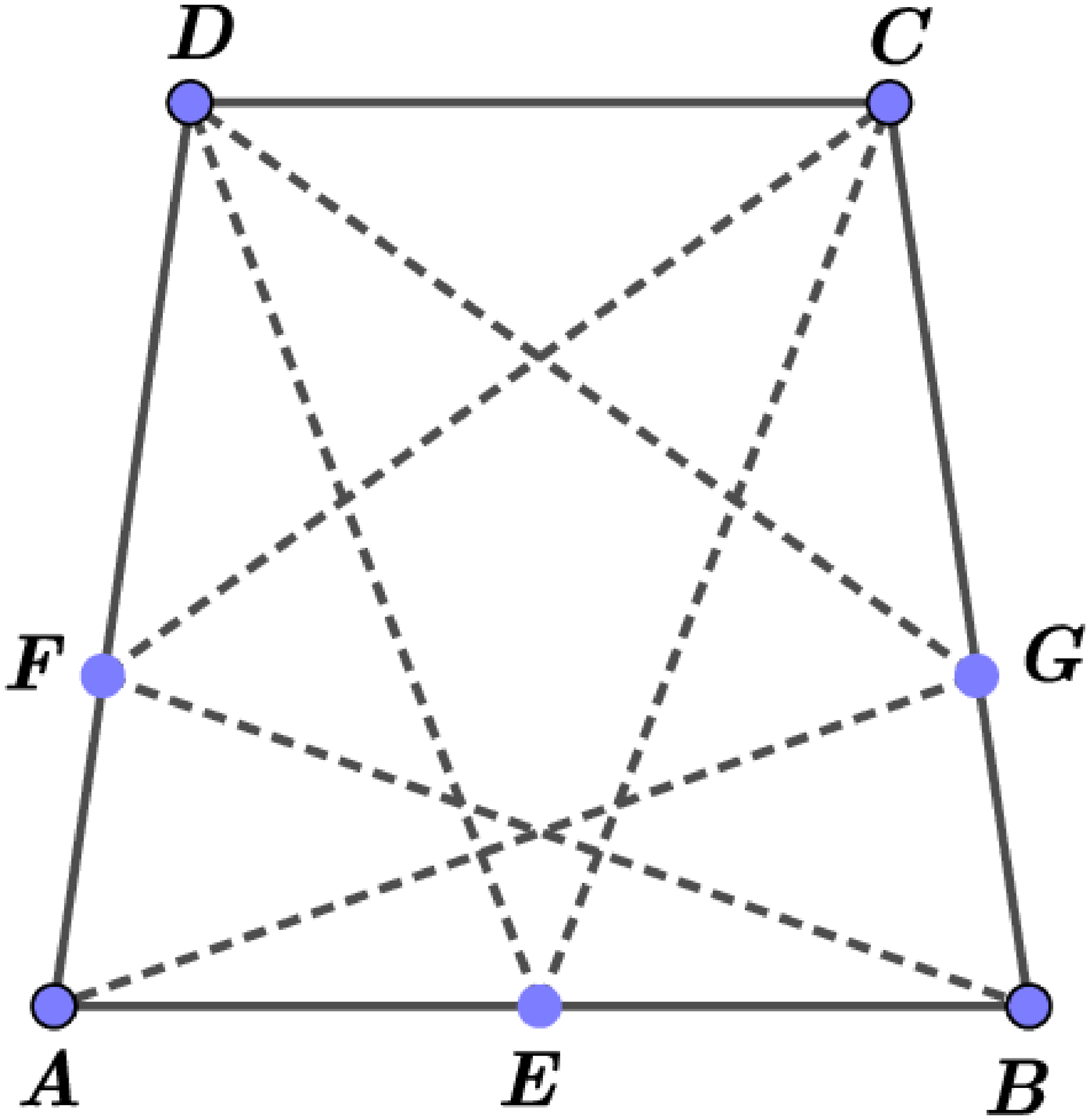}} g) \\
\end{minipage}
\caption{a) Case 2.1, b)
Case 2.2, c) Case 2.3, d) Case 3.1, e) Case 3.2,  f) Case 3.3, g) Case 3.4.}
\label{ris}
\end{figure}

Denote by $\|G_{\min}\|$  {\it the number of points in the set} $G_{\min}$, or, in other words, the number of self Chebyshev centers for $\Gamma$ (see \eqref{eq_fuct3}).
\label{mumchebcenters}
For any $i \in \mathbb{Z}_4$, the set $G_{\min} \cap [A_i, A_{i+1}]$ is either empty or has exactly one point.

We denote the elements of $G_{\min}$ by capital roman letter (sometimes with indices).
The possible subcases depends not only of cardinality of $G_{\min}$ but also on  how many elements of the set $G_{\min}$
satisfy the conditions $\|{\mathcal F}(\,\cdot\,)\| \geq 2$ and $\|{\mathcal F}(\,\cdot\,)\|=1$ (see \eqref{eq.fath1})
and the relative position of these points.
\smallskip

In what follows, a self Chebyshev center $x \in \Gamma$ is called {\it simple} if $\|{\mathcal F}(x)\|=1$ and {\it non-simple}
if $\|{\mathcal F}(x)\|\geq 2$. \label{simple_nonsimple}

On the other hand, $G_{\min}\neq \emptyset$ and, therefore, $1\leq \|G_{\min}\| \leq 4$. Moreover, $G_{\min}$ does not contain
any vertex of~$\Gamma$ by Corollary \ref{cor.simple1}.
We distinguish four cases according to the value of $\|G_{\min}\|$.
{\bf Case $k$} is the case when $\|G_{\min}\|=k$, $k=1,2,3,4$.

\section{Case 1}\label{2.5}

\begin{prop}\label{pr.case1}
Case 1 does not provide extremal quadrilaterals.
\end{prop}

\begin{proof}
Suppose that $\|G_{\min}\|=1$ and $G_{\min}=\{B\}$, where
$B$ is an interior point of the segment (say) $[A_1,A_2]$.
Hence, $[A_2,A_3]\cap G_{\min}=[A_3,A_4]\cap G_{\min}=[A_4,A_1]\cap G_{\min}=\emptyset$ and
$A_1, A_2, A_3, A_4 \in F(B)$ by Proposition \ref{pr_extr1}. In particular, $d(A_1,A_2)=d(A_1,B)+d(A_2,B)=2 \delta(\Gamma)$.
By the triangle inequality we have also that $d(A_2,A_3)+d(A_3,A_4)+d(A_4, A_1) \geq d(A_1,A_2)=2 \delta(\Gamma)$.
Therefore, $L(\Gamma)\geq 4\delta(\Gamma)$.
Obviously, $4 > \frac{4}{3}\sqrt{2\sqrt{3}+3}= 3.389946...$.
Therefore, Case 1 does not provide an extremal quadrilateral.
\end{proof}

\begin{figure}[t]
\center{\includegraphics[width=0.45\textwidth]{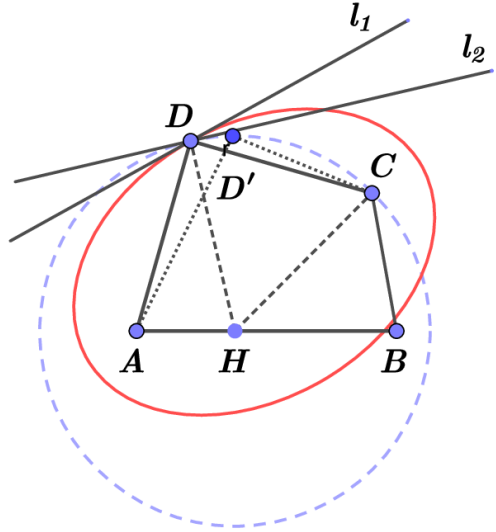}}\\
\caption{The ellipse and the circle for the quadrilateral $ABCD$.}
\label{Ellipse}
\end{figure}

\begin{remark}\label{rem4}
Recall that we use the term ``extremal'' in this paper only for polygons with {\bf minimal} perimeter among all
polygons with a given self Chebyshev radius of the boundary. It should be noted that the polygon with {\bf maximal} perimeter among all
polygons with a given self Chebyshev radius of the boundary was determined in \cite[Theorem 4]{BMNN2021}.
The boundary of such a polygon has only one self Chebyshev center. For $n=4$, such a polygon is
a half of a regular $6$-gon.
\end{remark}

In the following three sections, we will consider Cases 2, 3, and 4 in turn.
For all this cases, it is reasonable to consider separately some subcases.
In Fig.~\ref{ris}, we show all the essential subcases of Cases 2 and 3.
\smallskip

The following auxiliary result is very useful.

\begin{lemma} \label{ellipse}
Let $ABCD$ be a quadrilateral with $\delta(\bd(ABCD))=1$ for the self Chebyshev radius of its boundary.
Suppose that there is a point $H\in G_{\min} \cap [A,B]$ such that ${\mathcal F}(H)=\{C,D\}$ and $\angle AHD <\pi/2$.
If, moreover, $[A,D]\cap G_{\min} =\emptyset$ {\rm(}respectively, $G_{\min} \subset [A,B] \cup [B,C]${\rm)} and $D \not \in {\mathcal F}(H')$ for any
$H'\in G_{\min}\setminus \{H\}$,
then either $\angle HDC \leq \angle HDA$ {\rm(}respectively, $\angle HDC = \angle HDA${\rm)}
or there exists a point $D'$ arbitrarily close to the point~$D$, such that
$\delta(\bd(ABCD')) = 1$ and  $L(\bd(ABCD'))<L(\bd(ABCD))$.
\end{lemma}

\begin{proof} At first, we consider the case $[A,D]\cap G_{\min} =\emptyset$. Let us suppose that $\angle HDC > \angle HDA$.
We consider an ellipse $\mathcal{E}$ with the foci at the points $A$ and $C$,
that contains the point $D$ (see Fig. \ref{Ellipse}).
Let $\mathcal{S}$ be a circle of radius $1$ centered at the point~$H$.
Recall that $C, D\in \mathcal{S}$.

Let $l_1$ be the tangent line to the ellipse $\mathcal{E}$ at the point $D$ and
$l_2$ the tangent line to the circle $\mathcal{S}$ at the point $D$.
Note that $l_2\bot DH$.
Since $\angle HDC > \angle HDA$,
then the angle between the tangent $l_2$ and the segment $DC$ is less than the angle between the tangent $l_2$ and the segment $DA$.

By a well known property of the ellipse, the angle between the tangent $l_1$ and the line segment $[D,C]$ is equal to the angle between the tangent
$l_1$ and the the segment $[D,A]$.
Therefore, the arc of a circle $\mathcal{S}$ between the points $D$ and $C$ is such that any its point $D'\neq D$, that is sufficiently close to the point $D$,
lies inside the ellipse $\mathcal{E}$.
Therefore, we have $d(A,D')+d(D',C)<d(A,D)+d(D,C)$ for such a point $D'$ according to another one well known property of the ellipse.
Hence, $L(\bd(ABCD'))<L(\bd(ABCD))$. It is easy to see that the self Chebyshev radius of the boundary of the quadrilateral $ABCD'$
is $1$ if $D'$ is sufficiently close to $D$.

In the case $G_{\min} \subset [A,B] \cup [B,C]$ we can repeat the above arguments with one addition. In this case, we can take also
the point $D'$ outside the arc between $D$ and $C$
(but sufficiently close to $D$). Since $[C,D]\cap G_{\min} =\emptyset$, then such a variation of the quadrilateral $ABCD$ does not decrease the value
of the self Chebyshev radius of the boundary. On the other hand, if $\angle HDC < \angle HDA$, such a variation decreases the perimeter. Therefore,
either $\angle HDC = \angle HDA$ or there exists a point $D'$ arbitrarily close to the point~$D$, such that
$\delta(\bd(ABCD')) = 1$ and  $L(\bd(ABCD'))<L(\bd(ABCD))$.
\end{proof}
\smallskip

We obviously get the following important corollary.

\begin{figure}[t]
\center{\includegraphics[width=0.33\textwidth, trim=0mm 1mm 0mm 0mm, clip]{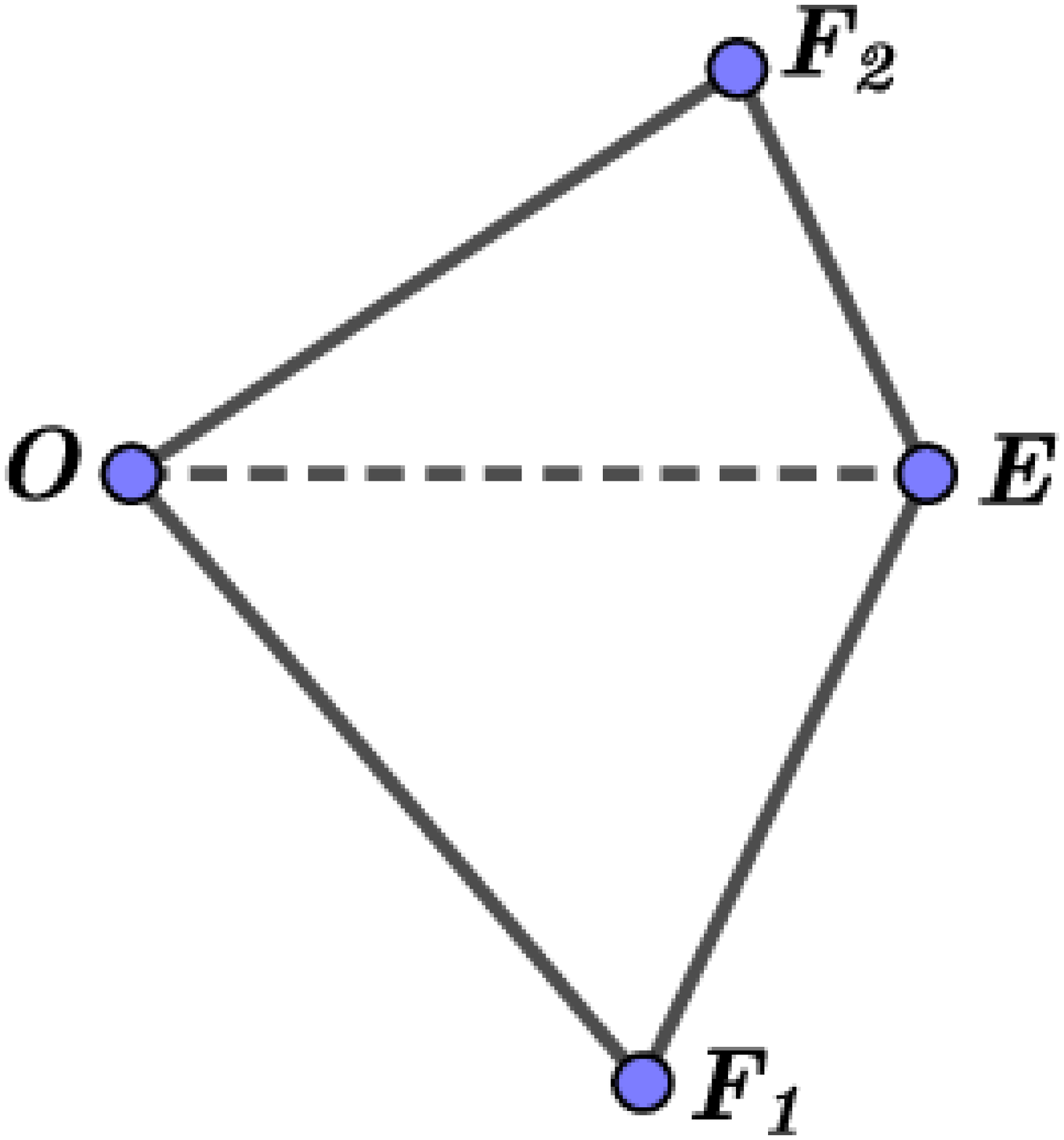}}\\
\caption{Lemma \ref{stor}.}
\label{Stor}
\end{figure}

\begin{corollary}\label{cor.angle2}
If the quadrilateral $ABCD$ is extremal in the assumptions of Lemma \ref{ellipse}, then $\angle CDH \leq \angle ADH$ {\rm(}respectively, $\angle CDH = \angle ADH${\rm)}.
\end{corollary}

The following lemma is important for further considerations.

\begin{lemma}\label{stor}
Let $OF_1EF_2$ be a quadrilateral such that $d(O,E)=d(O,F_1)=1$, $d(O,F_2)\leq1$,
$\angle OF_1E=\angle OEF_1=\angle OEF_2=:\beta\geq\pi/4$ and $\beta<\pi/2$.
We consider an ellipse $\mathcal{E}$ with the foci at the points $F_1$ and $F_2$,
that contains the point $E$.
Let $\mathcal{S}$ be a circle of radius $1$ centered at the point $O$.
If $d(E,F_2)>\frac{1}{3}d(E,F_1)$
then all points $E'\neq E$ on the circle $\mathcal{S}$ sufficiently close to the point $E$ lie inside the ellipse $\mathcal{E}$.
\end{lemma}

\begin{proof}
In a suitable Cartesian coordinate system, the points we need have the following coordinates:
$$
O=(0,0), \quad E=(1,0), \quad F_1=\bigl(\cos(\alpha),\sin(\alpha)\bigr), \quad F_2=\bigl(l\cos(\alpha)+1-l,-l\sin(\alpha)\bigr),
$$
where $\alpha:=\pi-2\beta$ (see Fig. \ref{Stor}).
Let us consider the equation of the ellipse $\mathcal{E}$:
$$
f(x,y):=d\bigl((x,y),F_1\bigr)+d\bigl((x,y),F_2\bigr)-d(E,F_1)-d(E,F_2)=0.
$$
By direct computations, we get
$$
f\bigl(\cos(t),\sin(t)\bigr)=2\bigl(1-\cos(t)\bigr)\cdot g(t),
$$
where $g(t)=-(1+l)(3l-1)\bigl(1-\cos(\alpha)\bigr)-(1-l^2)\sin(\alpha)t+o(t)$ as $t\rightarrow 0$.
Therefore, $l>1/3$ (that is equivalent to $d(E,F_2)>\frac{1}{3}\,d(E,F_1)$) implies $g(t)<0$ for $t$ sufficiently close to $0$.
\end{proof}

\section{Case 2}\label{sect.3}

Suppose that $\|G_{\min}\|=2$.
Our general assumptions are as follows.
Let $ABCD$ be a quadrilateral with $\|G_{\min}\|=2$ and with $\delta(\bd(ABCD))=1$ for the self Chebyshev radius of its boundary.
We consider the set
$$
G_{\min}=\{x\in \bd(ABCD)~|~ \mu(x)=1\}=\{H_1,H_2\},
$$
where both points $H_1$ and $H_2$ are interior points of some sides of $ABCD$, see \eqref{eq_fuct3}.
In other words, $u\in G_{\min}$ if and only if  $\mathcal{F}(u)=\{y\in \bd(ABCD)~|~ d(u,y)=1\}$.

We consider several subcases according to Fig. \ref{ris}.

\subsection{Case 2.1}
Without loss of generality, we may suppose that $G$ is an interior point of $[A,D]$,
$E$ is an interior point of $[A,B]$ such that $DE \perp [A,B]$, and
$\delta(\bd(ABCD))=d(G,B)=d(G,C)=d(E,D)=1$ (see Fig. \ref{case2_1}).
Moreover, we have $\angle DGC \leq \pi/2$ and $\angle BGA \leq \pi/2$ by Proposition \ref{locmin_mu1}.

Let us suppose that  the quadrilateral $ABCD$ is extremal.
Then we have $\angle GCD=\angle GCB$ by Corollary \ref{cor.angle2}. Let us introduce the following notation:  $\beta:=\angle GBC$, $\varphi=\angle ADE$, $\psi=\angle GBA$.

If $\beta \leq \pi/4$, then $\angle CGB \geq \pi/2$ and we obtain the following simple estimate for the perimeter:
$L(\bd(CGB)) \geq 2+\sqrt{2}$. Hence, we have  $L(\bd(ABCD)) > 2+\sqrt{2}$, that is impossible for an extremal quadrilateral. Therefore, $\beta > \pi/4$.

\begin{figure}[t]
\center{\includegraphics[width=0.35\textwidth]{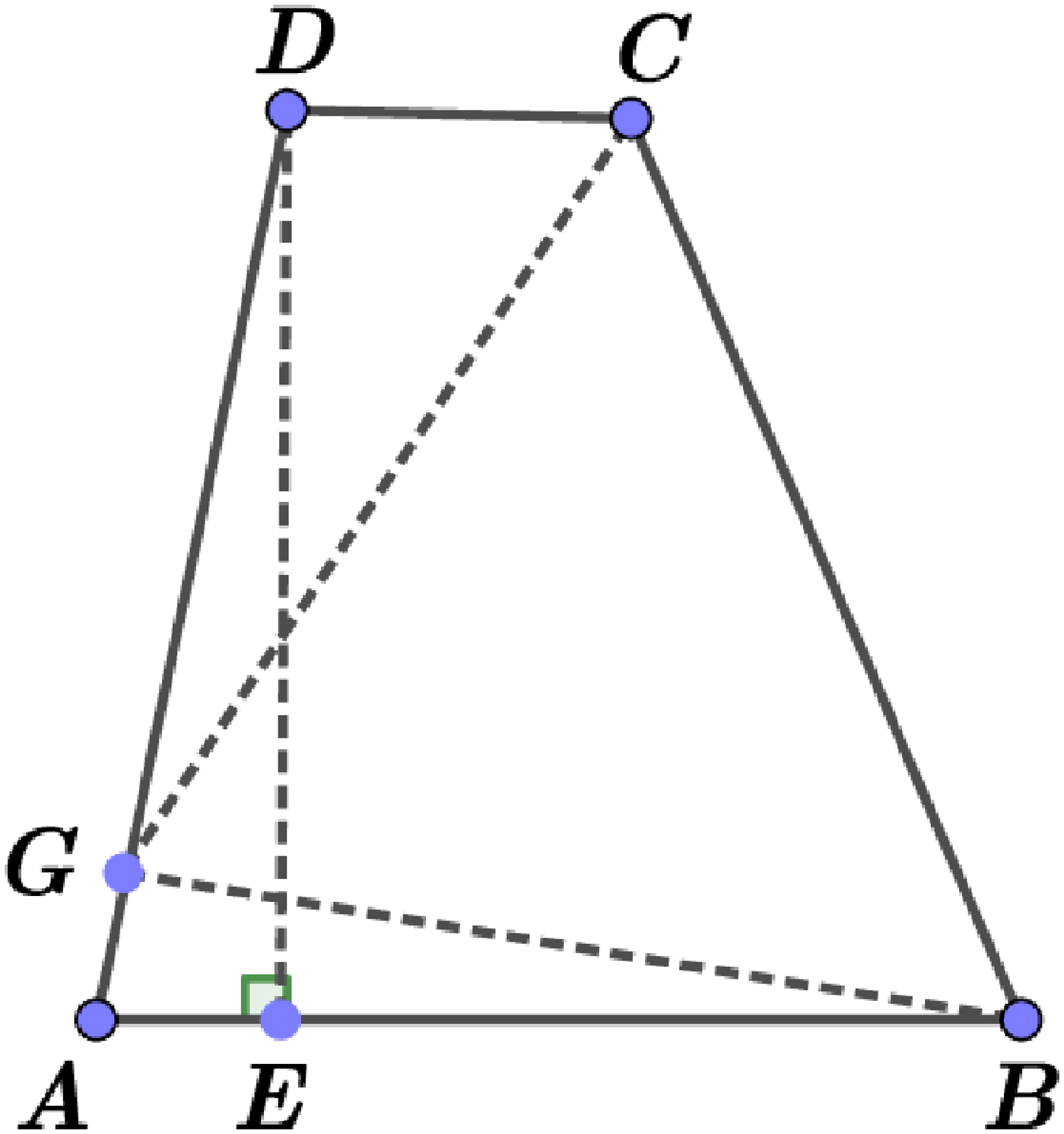}}\\
\caption{Case 2.1.}
\label{case2_1}
\end{figure}

Since $\angle AGB=\pi/2-(\psi-\varphi)\leq \pi/2$, then $\psi\geq \varphi\geq 0$.
It is easy to see that $\angle EDC=3\pi/2-3\beta-\psi$. Hence
$\angle GDC=\varphi+\angle EDC=3\pi/2-3\beta-(\psi-\varphi)\geq\pi/2$. It implies
$\beta\leq\pi/3-(\psi-\varphi)/3\leq\pi/3$.

It is important that  $d(C,D) \leq \frac{1}{3}d(B,C)$ (otherwise the quadrilateral $ABCD$ is not extremal by Lemma~\ref{stor}).
Since $d(B,C)=2\cos(\beta)$, we get $d(C,D)=2l\cos(\beta)$, where $l\in(0,1/3]$.

If $\beta\in[\pi/4,\pi/3]$ and $l\in[0,1/3]$, then $4\cos^2(\beta)\leq2$ and $l(1-l)\leq 2/9$.
These inequalities imply $4\cos^2(\beta)l(1-l)\leq 4/9$ and, consequently,
$$
d(G,D)=\sqrt{1-4\cos^2(\beta)\cdot l(1-l)}\geq\frac{\sqrt{5}}{3}
$$
On the other hand, $d(G,D)=\frac{1-\sin(\psi)}{\cos(\varphi)}$. Hence, $1-\sin(\psi)\geq\frac{\sqrt{5}}{3}\cos(\varphi)$ or, equivalently, $\frac{\sqrt{5}}{3}\cos(\varphi)+\sin(\psi)\leq1$.
We have $\sin(\varphi)\leq\sin(\psi)$ and $\cos(\varphi)\geq\cos(\psi)$ due to $0\leq\varphi\leq\psi$.
It implies that $\varphi$ and $\psi$ both satisfy the inequality $\frac{\sqrt{5}}{3}\cos(t)+\sin(t)\leq1$, which solutions for $t\in[0,\pi/2]$ are as follows:
$t\in[0,t_0]$, where $t_0=\arcsin(2/7)$.
Therefore, $0\leq\varphi\leq\psi\leq t_0<\pi/10$.

Put $\theta:=\psi-\varphi$. It is easy to see that $\angle DGC=\pi-\beta-\angle GDC=2\beta+\theta-\pi/2$.
The equations (the low of sines for $\triangle CDG$)
$$
\frac{d(G,D)}{\sin(\beta)}=\frac{1}{\sin(\frac{3\pi}{2}-3\beta-\theta)}=\frac{2l\cos(\beta)}{\sin(2\beta+\theta-\pi/2)}
$$
imply $2l\cos(\beta)=\frac{\cos(2\beta+\theta)}{\cos(3\beta+\theta)}$ or, equivalently,
$$
l=\frac{\cos(2\beta+\theta)}{2\cos(3\beta+\theta)\cos(\beta)}=:f(\beta,\theta).
$$
It is easy to check that
$$
\frac{\partial f}{\partial \theta}(\beta,\theta)=\frac{\sin(\beta)}{2\cos^2(3\beta+\theta)\cos(\beta)}>0,
$$
for $\beta\in [\pi/4,\pi/3]$ and $\theta\in[0,\pi/10]$.
Hence
$\frac{1}{3}\geq l\geq\frac{\cos(2\beta)}{2\cos(3\beta)\cos(\beta)}=\frac{\cos(2\beta)}{\cos(2\beta)+\cos(4\beta)}$.
It implies $\cos(4\beta)\leq 2\cos(2\beta)$. We get $2\eta^2-2\eta-1\leq0$ for $\eta:=\cos(2\beta)$.
The solution of this inequality is $\eta\in \left[\frac{1-\sqrt{3}}{2},0 \right]$.
Therefore, $\beta\in(\pi/4,\beta_0]$, where $\beta_0=\frac{\pi}{2}-\frac{1}{2}\arccos \left(\frac{1-\sqrt{3}}{2}\right)$.
\smallskip

{\bf Now we are going to show that the inequality $d^2(E,C)\leq1$ does not hold for $\beta\in(\pi/4,\beta_0]$}
(hence, $E \not\in G_{\min}$, and we get a contradiction).
If we suppose that $d^2(E,C)\leq1$, then $\angle EDC<\pi/2$ and, consequently,  $3\beta+\psi>\pi$.
In particular, $\psi>\pi-3\beta_0=\psi_0$ and $\sin(3\beta+\psi)<0$.
On the other hand,
$$
d(B,C)\cdot \sin(\angle EBC)+d(C,D)\cdot \cos(\angle EDC)=
2\cos(\beta)\bigl(\sin(\beta+\psi)-l\sin(3\beta+\psi)\bigr)=1
$$ i.e.
$$
-l\sin(3\beta+\psi)=\frac{1}{2\cos(\beta)}-\sin(\beta+\psi)=:F(\beta,\psi).
$$
Therefore, $\sin(3\beta+\psi)<0$ if and only if $F(\beta,\psi)>0$.
It is easy to check that
$$
 \frac{\partial F}{\partial \beta}(\beta,\psi)=\frac{\sin(\beta)}{2\cos^2(\beta)}-\cos(\beta+\psi)>
 \frac{\sin(\beta)}{2\cos^2(\beta)}-\cos(\beta)=\frac{\sin(\beta)-2\cos^3(\beta)}{2\cos^2(\beta)}>0
 $$
for $\beta\in(\pi/4,\beta_0]$.
Thus, the function $F(\beta,\psi)$ is strictly increasing with respect to the
variable $\beta$ for any $\psi\in(\psi_0,\pi/10]$.
Further, it is easy to see that
 $$
 \frac{\partial F}{\partial \psi}(\beta,\psi)=-\cos(\beta+\psi)<0
 $$
for $\psi\in(\psi_0,\pi/10]$.
Thus, the function $F(\beta,\psi)$ is strictly decreasing with respect to the
variable $\psi$ for any $\beta\in(\pi/4,\beta_0]$.

\begin{figure}[t]
\center{\includegraphics[width=0.36\textwidth]{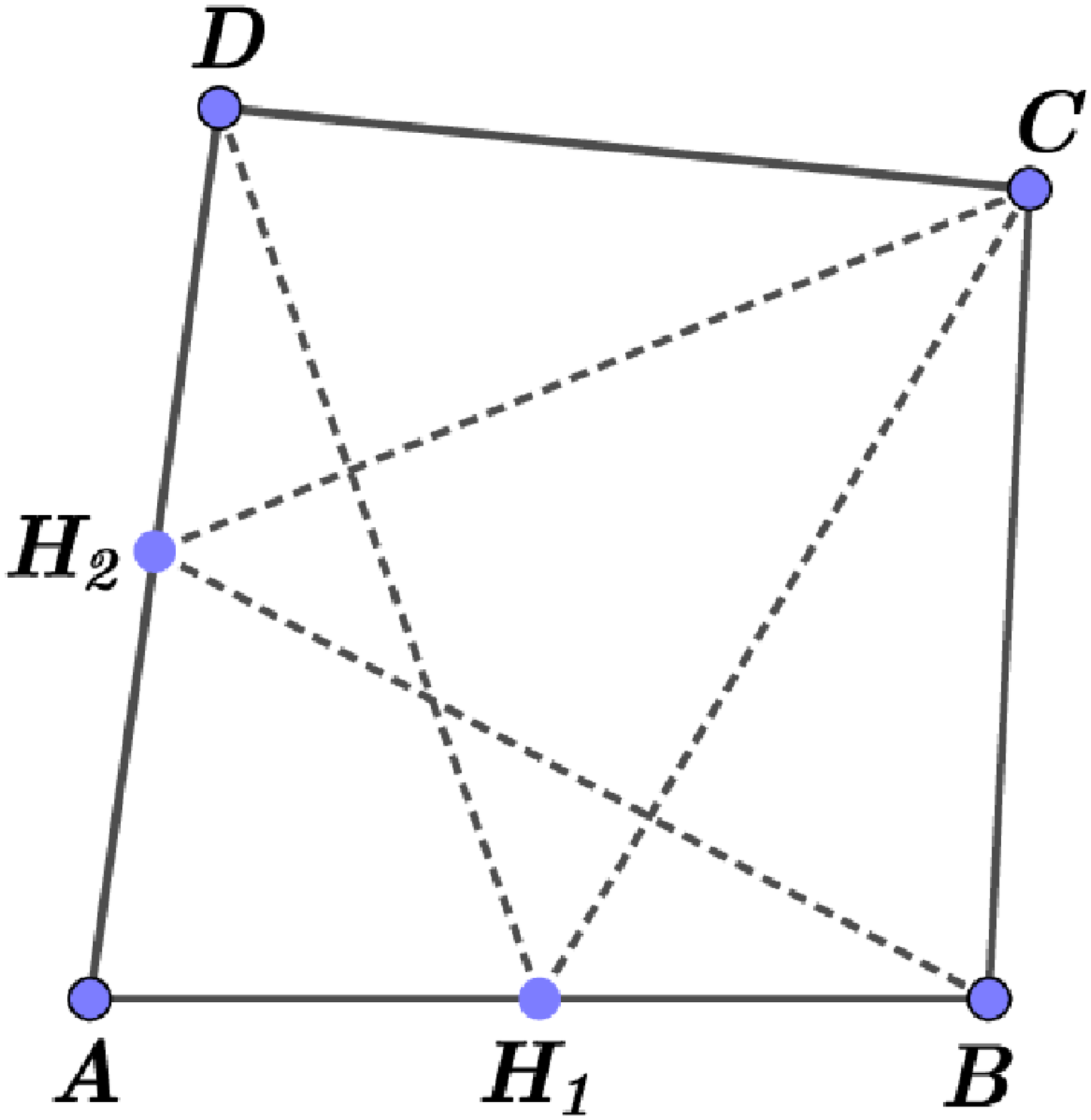}}\\
\caption{Case 2.2.}
\label{case2_2}
\end{figure}

Therefore, $F(\beta_0,\psi_0)>F(\beta,\psi)$. Now, it is easy to check that $F(\beta_0,\psi_0)<0$.
Indeed,
$$
1-2\cos(\beta_0)\sin(\beta_0+\psi_0)=
1-\sqrt{2\sqrt{3}}\sin\left(\frac{1}{2}\arccos\left(\frac{\sqrt{3}-1}{2}\right)\right)=-0.047891316... <0.
$$
Consequently, $F(\beta,\psi)<0$ and we get a contradiction. Hence, Case~2.1 does not provide extremal quadrilaterals.
\medskip

\subsection{Case 2.2}
Let us suppose that the quadrilateral under consideration is extremal.
Assume, for definiteness, that $H_1$ is an interior point of $[A,B]$,
$H_2$ is an interior point of $[A,D]$ and
$\delta(\bd(ABCD))=d(H_1,C)=d(H_1,D)=d(H_2,B)=d(H_2,C)=1$ (see Fig. \ref{case2_2}).

Then
$G_{\min} \cap [B,C]=\emptyset$ and
$G_{\min} \cap [D,C]=\emptyset$.
By Corollary \ref{cor.angle2}
we have $\angle DCH_1 \leq \angle BCH_1$ and
$\angle BCH_2 \leq \angle DCH_2$, respectively.
However, $\angle DCH_2 < \angle DCH_1$ and $\angle BCH_1 < \angle BCH_2$.
This contradiction shows that our
assumption (that $ABCD$ is extremal) is wrong.
\medskip

\subsection{Case 2.3} Let us suppose that the quadrilateral under consideration is extremal.
Assume for definiteness
that $H_1$ is an interior point of $[A,B]$,
$H_2$ is an interior point of $[C,D]$,  and
$\delta(\bd(ABCD))=d(H_1,C)=d(H_1,D)=d(H_2,A)=d(H_2,B)=1$ (see the left panel of Fig. \ref{case2_3}).
Without loss of generality, we may suppose that $\angle CDA\leq \pi/2$.

We know that $\angle DH_1A \leq \pi/2$. Suppose at first that $\angle DH_1A < \pi/2$.
Then we have the inequality $\angle H_1DC\leq \angle H_1DA$ (for an extremal quadrilateral) by Corollary~\ref{cor.angle2}.
Since $\angle H_1DC +\angle H_1DA=\angle CDA\leq \pi/2$, we get
$\angle H_1DC = \angle H_1CD\leq \pi/4$, then $\angle DH_1C\geq \pi/2$ and we obtain the following simple estimate for the perimeter:
$L(\bd(DCH_1))\geq 2+\sqrt{2}$. Hence, we have  $L(\bd(ABCD)) > 2+\sqrt{2}$, that is impossible for an extremal quadrilateral.

Now, let us suppose that $\angle DH_1A = \pi/2$.
It implies  $\angle DAB=\angle DAH_1 < \pi/2$ and
we can repeat the above arguments
changing the points $A,B,C,D, H_1,H_2$ to the points $C,D,A,B,H_2,H_1$, respectively.
Hence, if $\angle AH_2D < \pi/2$ we again get $L(\bd(ABCD)) > 2+\sqrt{2}$ by Corollary \ref{cor.angle2}.

Therefore, the remaining case, called {\bf Case~2.3a}, is $\angle AH_2D=\pi/2$ that (together with $\angle DH_1A = \pi/2$) implies $\angle CDA = \angle BAD$.

Thus $ABCD$ is an isosceles trapezoid with the bases $DA$ and $CB$ (see the right panel of Fig. \ref{case2_3}).
Put  $\alpha:=\angle CDA = \angle BAD \in \left( 0, \pi/2\right)$.
It is easy to see that $\angle H_2AD = \angle H_1DA =  \pi /2-\alpha$ and
$$
\angle H_2AB = \angle H_2BA =  \angle H_1DC = \angle H_1CD = \alpha - \left(\frac{\pi}{2}-\alpha\right) = 2\alpha - \frac{\pi}{2}\,.
$$
In particular, $\alpha > {\pi}/{4}$.

\begin{figure}[t]
\begin{minipage}[h]{0.31\textwidth}
\center{\includegraphics[width=0.99\textwidth]{case2_3.eps}}    a) \\
\end{minipage}
\quad\quad
\begin{minipage}[h]{0.38\textwidth}
\center{\includegraphics[width=0.99\textwidth]{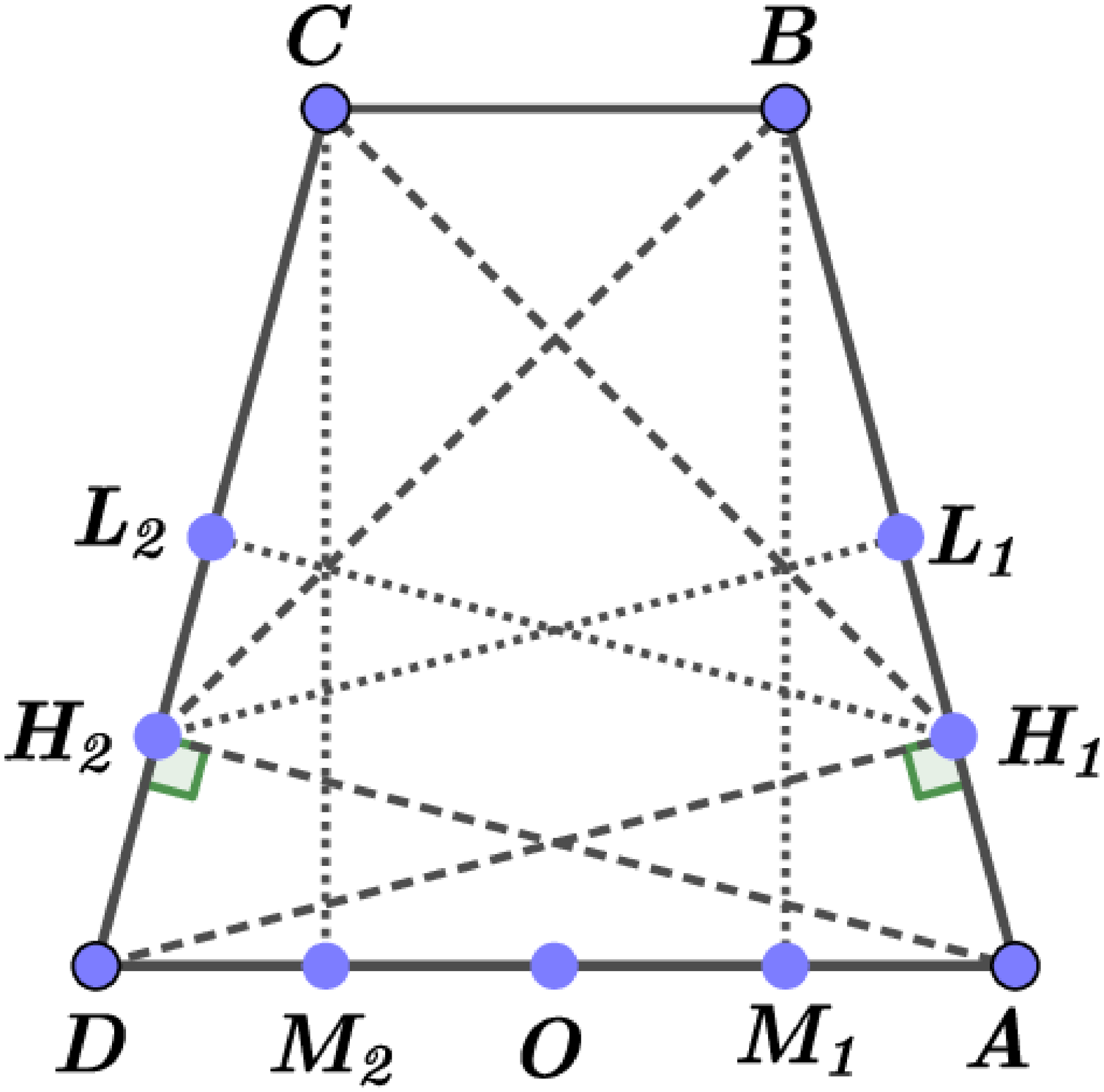}} b) \\
\end{minipage}
\caption{a) Case 2.3; b) Case 2.3a.}
\label{case2_3}
\end{figure}

Since the straight line $H_2L_1$ is orthogonal to the straight line $AB$, then
$\angle BH_2L_1 = \angle AH_2L_1= {\pi}/{2}-\left(2\alpha-{\pi}/{2}\right)=\pi-2\alpha$.
Therefore, $d(A,L_1)=\sin(\pi-2\alpha)=\sin(2\alpha)$ and $d(A,B)=d(D,C)=2\sin(2\alpha)$.

Since the straight line $BM_1$ is orthogonal to the straight line $AD$, we get
\begin{eqnarray*}
d(A,M_1)&=&d(A,B)\cos(\alpha)=2\sin(2\alpha)\cos(\alpha)=4\sin(\alpha)\cos^2(\alpha),\\
d(B,M_1)&=&d(A,B)\sin(\alpha)=2\sin(2\alpha)\sin(\alpha)=4\sin^2(\alpha)\cos(\alpha).
\end{eqnarray*}
From $d(A,M_1) < \frac{1}{2}d(A,D)$, we get
$4\sin(\alpha)\cos^2(\alpha) < \frac{1}{2\sin(\alpha)}$ or $\sin^2(2\alpha)<1/2$. Therefore, $\sin(2\alpha) < \frac{1}{\sqrt{2}}$.
Solving this inequality and taking into account that $\pi/2 < 2\alpha < \pi$, we get $\alpha\in \left( 3\pi/8,\pi/2\right)$.

Let $O$ be the midpoint of the line segment  $[D,A]$. Then
$$
d^2(O,B)=d^2(O,C)=d^2(O,M_1)+d^2(B,M_1)
$$
$$
=\left(\frac{1}{2\sin(\alpha)}-4\sin(\alpha)\cos(\alpha)\right)^2+16\sin^4(\alpha)\cos^2(\alpha)
$$
$$
=\frac{1}{4\sin^2(\alpha)}-4\cos^2(\alpha)+16\sin^2(\alpha)\cos^2(\alpha)=:h(\alpha).
$$

It is easy to check that the function $h(\alpha)$ is strictly decreasing for $\alpha\in \left(3\pi/8,\pi/2\right)$
and takes the value $h(\alpha)\geq 1$ for
$\alpha\in \left[ 3\pi/8,\alpha_0\right]$, where $\alpha_0=\arccos\left(\frac{\sqrt{6}-\sqrt{2}}{4}\right)$.
It follows from the following facts: $h(\alpha_0)=1$, $h^{\prime}(\alpha)=-\sin(2\alpha)\left(-16\cos(2\alpha)-4+\frac{1}{(1-\cos(2\alpha))^2}\right)$,
$-\cos(2\alpha)> 1/\sqrt{2}$ and $\sin(2\alpha)>0$ for $\alpha \in \left(3\pi/8,\pi/2\right)$.

Consider the perimeter
$$
L\bigl(\bd(ABCD)\bigr)=\frac{2}{\sin(\alpha)}-8\sin(\alpha)\cos^2(\alpha)+8\sin(\alpha)\cos(\alpha)=:g(\alpha).
$$
It is easy to see that the function $g(\alpha)$ is strictly decreasing for $\alpha\in \left[3\pi/8,\alpha_0\right]$.
Indeed,  $g^{\prime}(\alpha)=-8 f \bigl(\cos(\alpha)\bigr)-2\frac{\cos(\alpha)}{\sin^2(\alpha)}$, where $f(t)=1-2t+2t^2-3t^3$.
It is easy to check that the maximal value of $f(t)$ for $t\in [\cos(\alpha_0), \cos(3\pi/8)]$ is $f\bigl(\cos(3\pi/8)\bigr)=-0.1098679896...$
Hence, $g^{\prime}(\alpha)<-2\frac{\cos(\alpha)}{\sin^2(\alpha)}<0$ on the segment $\left[3\pi/8,\alpha_0\right]$.

Therefore, $g(\alpha)$ achieves its minimal value on $\left[ 3\pi/8,\alpha_0\right]$ exactly
at the point $\alpha=\alpha_0$.
This minimal value is
$$
g\left(\arccos\left(\frac{\sqrt{6}-\sqrt{2}}{4}\right)\right)=\frac{3(\sqrt{6}-\sqrt{2})}{2}+2=3.55291....
$$
Therefore, the quadrilateral $ABCD$ is not extremal.

\section{Case 3}\label{sect.4}
In this section we deal with the case $\|G_{\min}\|=3$, i.~e. $\bd(ABCD)$ has three self Chebyshev centers.
We consider several subcases according to Fig. \ref{ris}.

\subsection{Case 3.1}
Let us suppose that $\|G_{\min}\|=3$ and $G_{\min}=\{H_1, H_2, G\}$.
Without loss of generality, we may suppose that $H_1$ is an interior point of $[B,C]$,
$H_2$ is an interior point of $[A,D]$, $G$ is an interior point of $[A,B]$ such that $d(A,G) \leq d(B,G)$, and
$\delta(\bd(ABCD))=d(H_1,A)=d(H_2,B)=d(G,D)=d(G,C)=1$. It is clear that $AH_1$ is orthogonal to $[B,C]$ and $BH_2$ is orthogonal to $[A,D]$.
Moreover, $\angle AGD \leq \pi/2$ and $\angle BGC \leq \pi/2$ by Proposition \ref{locmin_mu1}. We put $\alpha:=\angle DAB=\angle CBA$.

\begin{lemma}\label{le.c.3.1}
In the conditions as above, we get $\alpha\geq \pi/3$. If, in addition, $d(B,G)\geq 1$, then $L\bigl(\bd(ABCD)\bigr)\geq 2+\sqrt{2}$.
\end{lemma}

\begin{proof}
Let us consider the quadrilateral $ABCD$ as a subset of the isosceles triangle $ABE$ (see the left panel of Fig.~\ref{case3_1}).
We put $x=d(A,G)$,
$y=d(G,B)$, where $x\leq y$.

Since $\angle ADB\leq \pi/2$, then we get $\angle ADG\leq \pi/2$ and $\angle GDE\geq \pi/2$.
Since $\angle AGD \leq \pi/2$ and $d(A,G) \leq d(B,G)$,
then there is a point $D'\in [D,E]$, such that the straight line $D'G$ is orthogonal to the straight line $AB$.
The inequality $\angle GDE\geq \pi/2$ implies $d(G,D')\geq d(G,D)=1$. Since
$\frac{\cos(\alpha)}{\sin(\alpha)}=\frac{d(A,G)}{d(G,D')}\leq d(A,G)=x$, then we get $x\geq\frac{\cos(\alpha)}{\sin(\alpha)}$.

On the other hand, $d(A,B)=x+y=\frac{1}{\sin(\alpha)}$, hence $x\leq \frac{1}{2\sin(\alpha)}$.
Therefore,
$$
\frac{\cos(\alpha)}{\sin(\alpha)}\leq x \leq \frac{1}{2\sin(\alpha)}.
$$
Hence, we get $\cos(\alpha)\leq 1/2$ or, equivalently, $\alpha\geq  \pi/3$.

Now, if $d(B,G)\geq 1$, then the perimeter of the triangle $BGD$ is as least $2+\sqrt{2}$.
Indeed, it easily follows from the relations $d(G,D)=1$, $d(G,B) \geq 1$, and $\angle BGD \geq \pi/2$. Since quadrilateral $ABCD$ contains the triangle $BGD$, then
$L(\Gamma)\geq 2+\sqrt{2}$.
\end{proof}
\medskip

\begin{figure}[t]
\begin{minipage}[h]{0.35\textwidth}
\center{\includegraphics[width=0.99\textwidth, trim=0mm 1mm 0mm 0mm, clip]{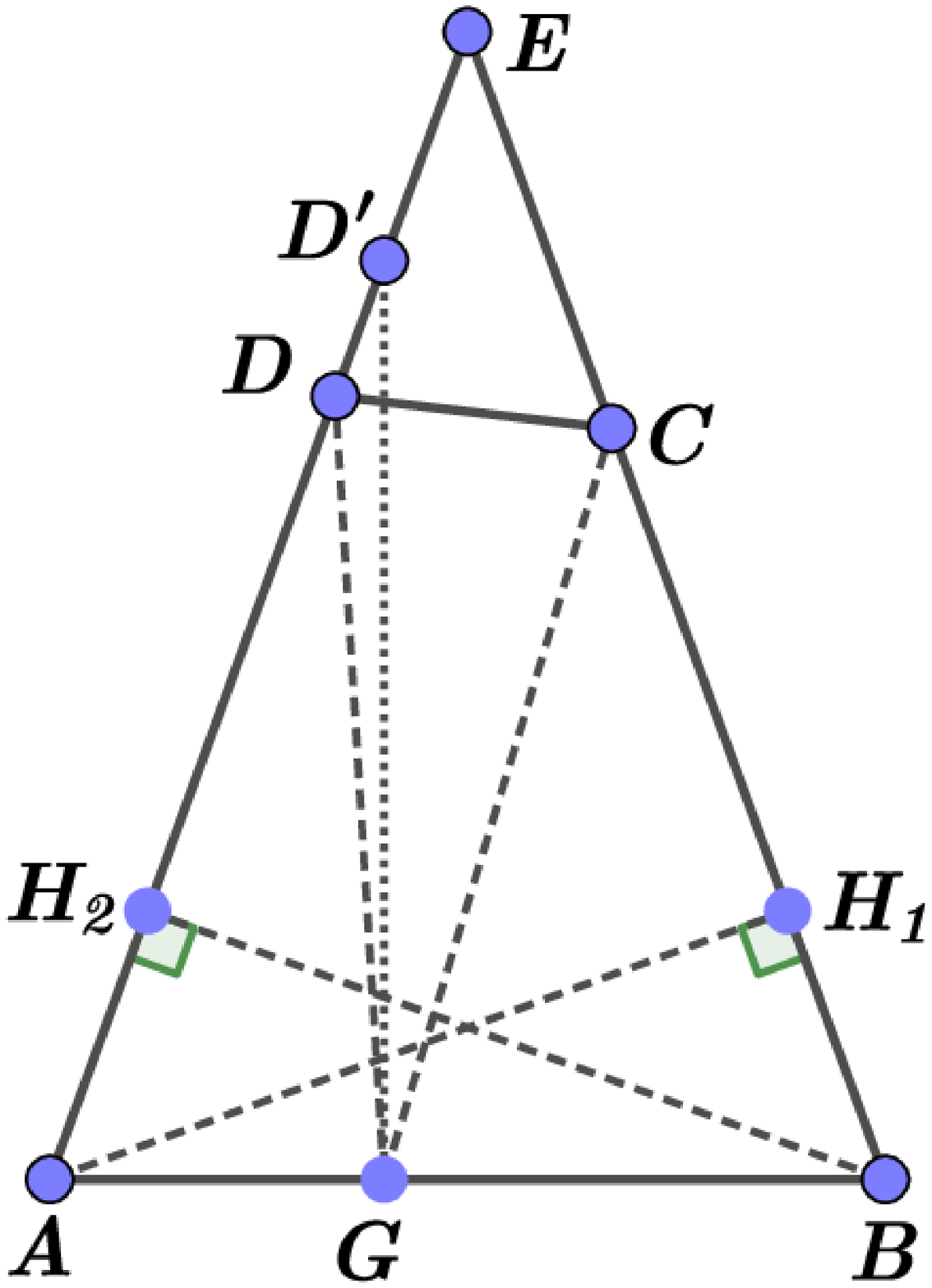}}    a) \\
\end{minipage}
\quad\quad
\begin{minipage}[h]{0.35\textwidth}
\center{\includegraphics[width=0.99\textwidth, trim=0mm 1mm 0mm 0mm, clip]{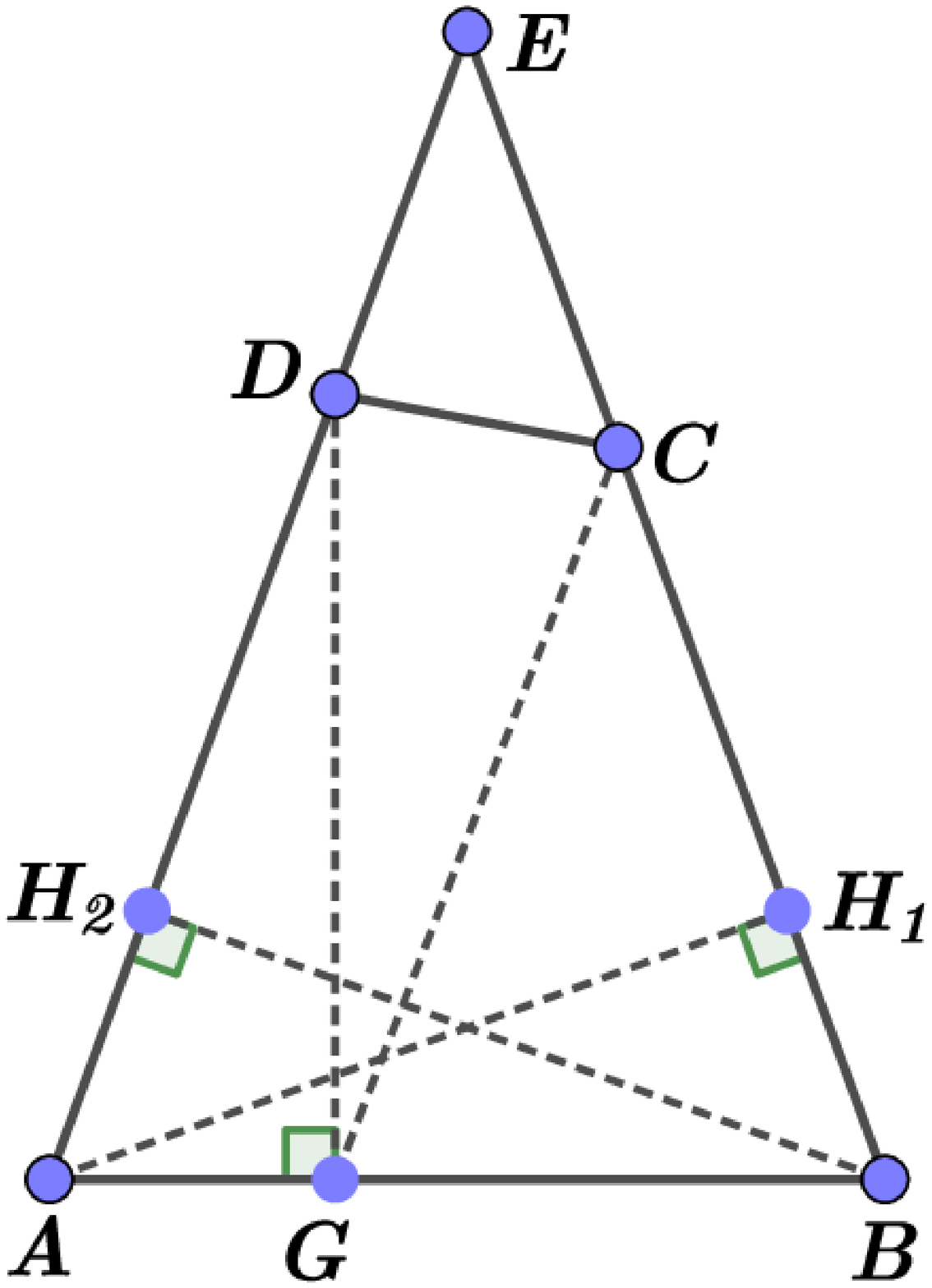}} b) \\
\end{minipage}
\caption{a) Case 3.1; b) Case 3.1a.}
\label{case3_1}
\end{figure}

Now we consider, the following class $\mathcal{P}_{\operatorname{spec}}$ of quadrilaterals: \label{specquadr}
a quadrilateral $ABCD$ is in $\mathcal{P}_{\operatorname{spec}}$ if and only if the following conditions hold:
$\alpha:=\angle DAB=\angle CBA \in [\pi/3,\pi/2)$ and there are
$H_1 \in [B,C]$, $H_2\in [A,D]$, and $G \in [A,B]$ such that
$d(H_1,A)=d(H_2,B)=d(G,D)=d(G,C)=1$, $d(A,G)\leq d(B,G)\leq 1$, $AH_1 \perp [B,C]$, $BH_2 \perp [A,D]$, $\angle AGD \leq \pi/2$, $\angle BGC \leq \pi/2$.

Lemma \ref{le.c.3.1} implies that any quadrilateral in Case 3.1 with $d(G,B) \leq 1$ is in the class $\mathcal{P}_{\operatorname{spec}}$.
Now, it suffices to prove the following

\begin{prop}\label{pr.c.3.1}
For any quadrilateral $P=ABCD$ in the class $\mathcal{P}_{\operatorname{spec}}$, we have $\delta(\Gamma)\leq 1$ and  $L(\Gamma)\geq 2+\sqrt{2}$.
In particular, there is no extremal quadrilateral in $\mathcal{P}_{\operatorname{spec}}$.
\end{prop}

\begin{proof}
Since the distances from the points $A, B, C, D$ to $G$ is at most $1$, then $\delta(\Gamma)\leq 1$. Let us prove that $L(\Gamma)\geq 2+\sqrt{2}$.
We put $x:=d(A,G)$, $y:=d(G,B)$, $\beta:=\angle DGA$, and $\gamma:=\angle CGB$.
We have $x\leq y$ by our assumptions.
Note that $d(A,B)=x+y=\frac{1}{\sin(\alpha)}$, hence, $x \leq \frac{1}{2\sin(\alpha)}$.
It is easy to see also that $x\geq\frac{\cos(\alpha)}{\sin(\alpha)}$ (see e.g. the proof of Lemma \ref{le.c.3.1})

Let us consider $E$, the intersection point of  the straight lines $AD$ and $BC$ (the quadrilateral $ABCD$ is a subset of the isosceles triangle $ABE$), see
the left panel of Fig.~\ref{case3_1}.

Since $d(D,G)=d(C,G)=1$, then we have $\angle GDC=\angle GCD=\frac{\beta+\gamma}{2}$ and $d(C,D)=2\cos \left(\frac{\beta+\gamma}{2}\right)$.
The equalities $\frac{d(A,D)}{\sin(\beta)}=\frac{1}{\sin(\alpha)}$ and $\frac{d(B,C)}{\sin(\gamma)}=\frac{1}{\sin(\alpha)}$ imply
$d(A,D)+d(B,C)=\frac{\sin(\beta)+\sin(\gamma)}{\sin(\alpha)}$.
Therefore, we get
$$
L(\Gamma)=d(A,B)+d(B,C)+d(C,D)+d(D,A)=
\frac{1}{\sin(\alpha)}+\frac{\sin(\beta)+\sin(\gamma)}{\sin(\alpha)}+2\cos\left(\frac{\beta+\gamma}{2}\right).
$$

The equalities
$\frac{x}{\sin(\alpha+\beta)}=\frac{1}{\sin(\alpha)}$
and $\frac{y}{\sin(\alpha+\gamma)}=\frac{1}{\sin(\alpha)}$
imply $x=\frac{\sin(\alpha+\beta)}{\sin(\alpha)}$ and
$y=\frac{\sin(\alpha+\gamma)}{\sin(\alpha)}$, respectively.
Since $x+y=\frac{1}{\sin(\alpha)}$, then we get $\sin(\alpha+\beta)+\sin(\alpha+\gamma)=1$.

Now, we introduce new variables $t$ and $s$ such that $t:=\frac{\beta+\gamma}{2}$ and $s:=\frac{\beta-\gamma}{2}$.
We get
$$
\sin(\beta)+\sin(\gamma)=
2\sin\left(\frac{\beta+\gamma}{2}\right)\cos\left(\frac{\beta-\gamma}{2}\right)=2\sin(t)\cos(s).
$$
The equation
$$
1=\sin(\alpha+\beta)+\sin(\alpha+\gamma)=
2\sin\left(\alpha+\frac{\beta+\gamma}{2}\right)\cos\left(\frac{\beta-\gamma}{2}\right)=2\sin(\alpha+t)\cos(s),
$$
implies $\cos(s)=\frac{1}{2\sin(\alpha+t)}$.
Therefore, we get
\begin{eqnarray*}
L(\Gamma)&=&\frac{1}{\sin(\alpha)}+\frac{\sin(\beta)+\sin(\gamma)}{\sin(\alpha)}+2\cos\left(\frac{\beta+\gamma}{2}\right)\\
&=&\frac{1}{\sin(\alpha)}+\frac{\sin(t)}{\sin(\alpha)\sin(\alpha+t)}+2\cos(t)=:F(\alpha,t).
\end{eqnarray*}

We will now find useful upper and lower bounds for the value $t$.
The equation $\frac{x}{\sin(\alpha+\beta)}=\frac{1}{\sin(\alpha)}$ implies
$\sin(\alpha+\beta)=x\sin(\alpha)$ and $\pi-\alpha-\beta=\arcsin(x\sin(\alpha))$, that could be rewritten as
$$
\beta=\pi-\alpha-\arcsin\bigl(x\sin(\alpha)\bigr).
$$
Similarly, we get
$$
\gamma=\pi-\alpha-\arcsin(y\sin(\alpha))=\pi-\alpha-\arcsin\bigl(1-x\sin(\alpha)\bigr).
$$
The above equality obviously implies
$$
t=\frac{\beta+\gamma}{2}=\pi-\alpha-\frac{1}{2}\arcsin(x\sin(\alpha))-\frac{1}{2}\arcsin\bigl(1-x\sin(\alpha)\bigr)=:u(x).
$$
It is easy to check that
$$
u'(x)=\frac{\sin(\alpha)}{2}\Bigl(\arcsin'\bigl(1-x\sin(\alpha)\bigr)-\arcsin'\bigl(x\sin(\alpha)\bigr)\Bigr).
$$
Therefore, $u'(x)\geq 0$ if and only if
$1-x\sin(\alpha)\geq x\sin(\alpha)>0$ or, equivalently, $0< x\leq\frac{1}{2\sin(\alpha)}$.

Since
$
\frac{\cos(\alpha)}{\sin(\alpha)}\leq x \leq \frac{1}{2\sin(\alpha)},
$
we will find the value of the function $u$ at the boundary points:
\begin{eqnarray*}
u\left(\frac{1}{2\sin(\alpha)}\right)&=&\pi-\alpha-\arcsin\left(\frac{1}{2}\right)=\frac{5\pi}{6}-\alpha.\\
u\left(\frac{\cos(\alpha)}{\sin(\alpha)}\right)&=&\pi-\alpha-\frac{1}{2}\arcsin \bigl(\cos(\alpha)\bigr)-\frac{1}{2}\arcsin \bigl(1-\cos(\alpha)\bigr)\\
&=&\frac{3}{4}\pi-\frac{1}{2}\alpha+\frac{1}{2}\arcsin \bigl(\cos(\alpha)-1\bigr)=:l(\alpha).
\end{eqnarray*}

Therefore, $l(\alpha)\leq t \leq \frac{5}{6}\pi-\alpha$ for a given $\alpha \in [\pi/3,\pi/2]$.
Now, we will prove that the minimal value of $F(\alpha,t)$
on the set
\begin{equation}\label{eq.setS}
S:=\left\{(\alpha,t)\,\left|\,\,\frac{\pi}{3}\leq \alpha \leq \frac{\pi}{2},\,  l(\alpha) \leq t \leq\frac{5\pi}{6}-\alpha\right.\right\}
\end{equation}
is $2+\sqrt{2}$. It will be suffices to prove the proposition.
\smallskip

It is easy to see that
$l'(\alpha)=-\frac{1}{2}\left(1+\frac{\sin(\alpha)}{\sqrt{1-(\cos(\alpha)-1)^2}}\right)<0$
for $\alpha\in[0,\frac{\pi}{2})$.
Therefore, the function $l$ is strictly decreasing on the interval $[0,\frac{\pi}{2}]$, $l(0)=3\pi/4$, $l(\pi/2)=\pi/4$.
By direct computations, we get
\begin{eqnarray*}
\frac{\partial F}{\partial \alpha}(\alpha,t)&\!\!\!\!=\!\!\!\!&
-\frac{\cos(\alpha)}{\sin^2(\alpha)}-\frac{\sin(t)\cos(\alpha)}{\sin^2(\alpha)\sin(\alpha+t)}-\frac{\sin(t)\cos(\alpha+t)}{\sin(\alpha)\sin^2(\alpha+t)}\\
&\!\!\!\!=\!\!\!\!&\frac{\cos(\alpha)+1}{\sin^2(\alpha)\sin^2(\alpha+t)}\bigl(\cos^2(\alpha+t)-\cos(\alpha)\bigr)\\
&\!\!\!\!=\!\!\!\!&\frac{\cos(\alpha)+1}{\sin^2(\alpha)\sin^2(\alpha+t)}\left(2\sin^2\left(\frac{\alpha}{2}\right)-\sin^2(\alpha+t)\right)\\
&\!\!\!\!=\!\!\!\!&\frac{\cos(\alpha)+1}{\sin^2(\alpha)\sin^2(\alpha+t)}
\left(\sqrt{2}\sin\left(\frac{\alpha}{2}\right)+\sin(\alpha+t)\right) \left(\sqrt{2}\sin\left(\frac{\alpha}{2}\right)-\sin(\alpha+t)\right).
\end{eqnarray*}

The derivative $\frac{\partial F}{\partial \alpha}(\alpha,t)$ has the same sign as
the expression $\left(\sqrt{2}\sin\left(\frac{\alpha}{2}\right)-\sin(\alpha+t)\right)$.
Solving equation $\sqrt{2}\sin\left(\alpha/2\right)=\sin(\alpha+t)$ with respect to $t$, we obtain
$$
t=\pi-\alpha-\arcsin\left(\sqrt{2}\sin\bigl(\alpha/2\bigr)\right).
$$
It is easy to show that the graph of this function does not
intersect a two-dimensional connected set $S$ (see \eqref{eq.setS}).
Since $\sqrt{2}\sin\left(\pi/4\right)-\sin\left(\pi/2+\pi/4\right)=\frac{2-\sqrt{2}}{2}>0$, then
we get $\frac{\partial F}{\partial \alpha}(\pi/2, \pi/4)>0$.
Thus, the function $F(\alpha,t)$ is strictly  increasing with respect to the variable $\alpha$ for any $t\in \left[l(\alpha),5\pi/6-\alpha \right]$.
\smallskip

Therefore, the function $F(\alpha,t)$ achieves its minimum value on $S$ at one of the points of the curve $t=l(\alpha)$, that corresponds
to the following configuration, which will be called {\bf Case 3.1a} (see the right panel of Fig.~\ref{case3_1}):
$d(A,H_1)=d(B,H_2)=d(D,G)=d(C,G)=1$,
$\angle AGD = \pi/2$.
Now we are going to find the explicit expression for $F(\alpha,l(\alpha))$, $\alpha\in [\pi/3,\pi/2)$.

We have $d(A,B)=d(A,D)=\frac{1}{\sin(\alpha)}$, $d(A,G)=\frac{\cos(\alpha)}{\sin(\alpha)}$
and $d(G,B)=\frac{1-\cos(\alpha)}{\sin(\alpha)}$ for $t=l(\alpha)$.
Let $\angle GDC = \angle GCD =\beta$.
Since $\angle DGC \leq \pi/2$, we have $\beta\in \left[\pi/4,\pi/2\right]$.
It~is clear that $\angle CGB = \pi/2-(\pi-2\beta)=2\beta-\pi/2$.
The equations
$$
\frac{d(B,C)}{\sin(2\beta-\pi/2)}=\frac{d(B,C)}{-\cos(2\beta)}=
\frac{d(C,G)}{\sin(\alpha)}=\frac{1}{\sin(\alpha)}
$$
imply $d(B,C)=-\frac{\cos(2\beta)}{\sin(\alpha)}$.
Since $d(C,D)=2\cos(\beta)$, then we have
$$
F(\alpha,l(\alpha))=d(A,B)+d(B,C)+d(C,D)+d(D,A)=
\frac{2-\cos(2\beta)}{\sin(\alpha)}+2\cos(\beta).
$$

It is clear that $\angle GCB=\pi-\left(2\beta-\pi/2+\alpha\right)=3\pi/2-(\alpha+2\beta)$.
Since $\angle GCB \in [0,\pi/2]$, then $\alpha+2\beta \in \left[\pi,3\pi/2\right]$.
The equations
$$
\frac{d(C,G)}{\sin(\alpha)}=\frac{1}{\sin(\alpha)}=
\frac{d(G,B)}{\sin(\angle GCB)}=\frac{d(G,B)}{\cos(\alpha+2\beta-\pi)}=\frac{1-\cos(\alpha)}{\sin(\alpha)\cos(\alpha+2\beta-\pi)}
$$
imply $\cos(\alpha+2\beta-\pi)=1-\cos(\alpha)$. Therefore,
$2\beta=\pi-\alpha+\arccos \bigl(1-\cos(\alpha)\bigr)$ and, consequently,
$$
\cos(2\beta)=-\cos\bigl(\alpha-\arccos(1-\cos(\alpha))\bigr)=\cos^2(\alpha)-\cos(\alpha)-
\sqrt{2\cos(\alpha)-\cos^2(\alpha)}\sin(\alpha).
$$

Since $\cos(\beta)=\sqrt{\frac{1+\cos(2\beta)}{2}}$, then we get
\begin{eqnarray*}
F(\alpha,l(\alpha))&=&
\frac{2+\cos(\alpha)-\cos^2(\alpha)}{\sin(\alpha)}+
\sqrt{2\cos(\alpha)-\cos^2(\alpha)}\\
&&+\sqrt{2+2\cos^2(\alpha)-2\cos(\alpha)-2\sqrt{2\cos(\alpha)-\cos^2(\alpha)}\sin(\alpha)}=:g(\alpha).
\end{eqnarray*}

By direct computations we get that the function $g(\alpha)$ achieves its maximal value on
$\left[\pi/3,\pi/2\right]$
at the point $\alpha=\arccos(t_0)= 1.416644291...$, where $t_0$ is the root of the polynomial
$f(t)=16t^7-120t^6+346t^5-485t^4+348t^3-118t^2+18t-1$ on the interval $[0.15,0.16]$.
This maximal value is
$$
g \bigl(\arccos(t_0)\bigr)=3.51731....
$$

By direct computations we get also that the function $g(\alpha)$ achieves its minimal value on
$\left[ \pi/3,\pi/2\right]$ exactly
at the point $\alpha=\pi/2$.
This minimal value is
$g\left(\pi/2\right)=2+\sqrt{2}$.
Therefore, $F(\alpha,t) \geq F(\alpha,l(\alpha))= g(\alpha) \geq g(\pi/2)=2+\sqrt{2}$  for all $(a,t) \in S$.
The proposition is proved.
\end{proof}

\begin{remark}
Note that minimal value of $F(\alpha,t)$ on $S$ is achieved at the point $(\alpha,t)=(\pi/2, l(\pi/2))=(\pi/2, \pi/4)$, that corresponds to a degenerate
quadrilateral $ABCD$, where $B=C$, $d(A,B)=d(A,D)=1$,  and $\angle BAD =\pi/2$. The self Chebyshev radius of the boundary of this degenerate quadrilateral (in fact, a triangle)
is $1/\sqrt{2} <1$.
\end{remark}

\medskip

\subsection{Case 3.2}
Without loss of generality, we may suppose that $G_{\min}=\{E, G, F\}$ (see Fig. \ref{case3_2}),
where $E$ is an interior point of $[A,B]$,
$F$ is an interior point of $[B,C]$, $G$ is an interior point of $[A,D]$ such that $d(A,G) \leq d(D,G)$, and
$\delta(\bd(ABCD))=d(E,D)=d(F,A)=d(G,B)=d(G,C)=1$. It is clear that $DE$ is orthogonal to $[A,B]$ and $AF$ is orthogonal to $[B,C]$.
Moreover, $\angle AGB \leq \pi/2$ and $\angle DGC \leq \pi/2$ by Proposition \ref{locmin_mu1}.

Now we are going to show that this case is impossible.

\begin{remark}\label{re.3.2}
It should be noted that it does not matter for the arguments below whether there is a fourth self Chebyshev center for $\bd(ABCD)$ or not.
Therefore, we can use this arguments also for some quadrilaterals with four self Chebyshev centers.
\end{remark}

\begin{figure}[t]
\begin{minipage}[h]{0.41\textwidth}
\center{\includegraphics[width=0.99\textwidth, trim=0mm 0mm 0mm 1mm, clip]{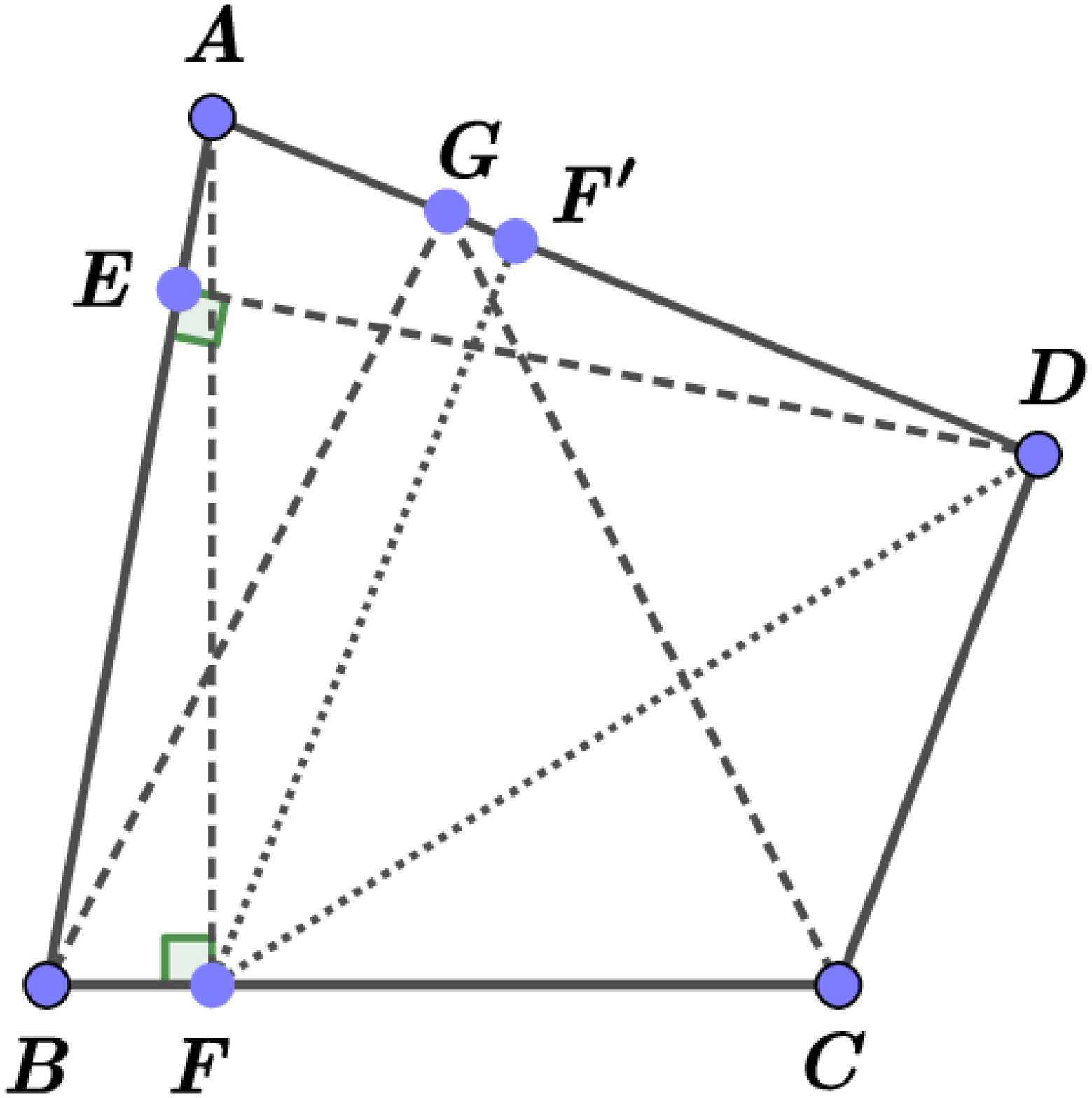}}    a) \\
\end{minipage}
\quad\quad
\begin{minipage}[h]{0.41\textwidth}
\center{\includegraphics[width=0.99\textwidth, trim=0mm 0mm 0mm 1mm, clip]{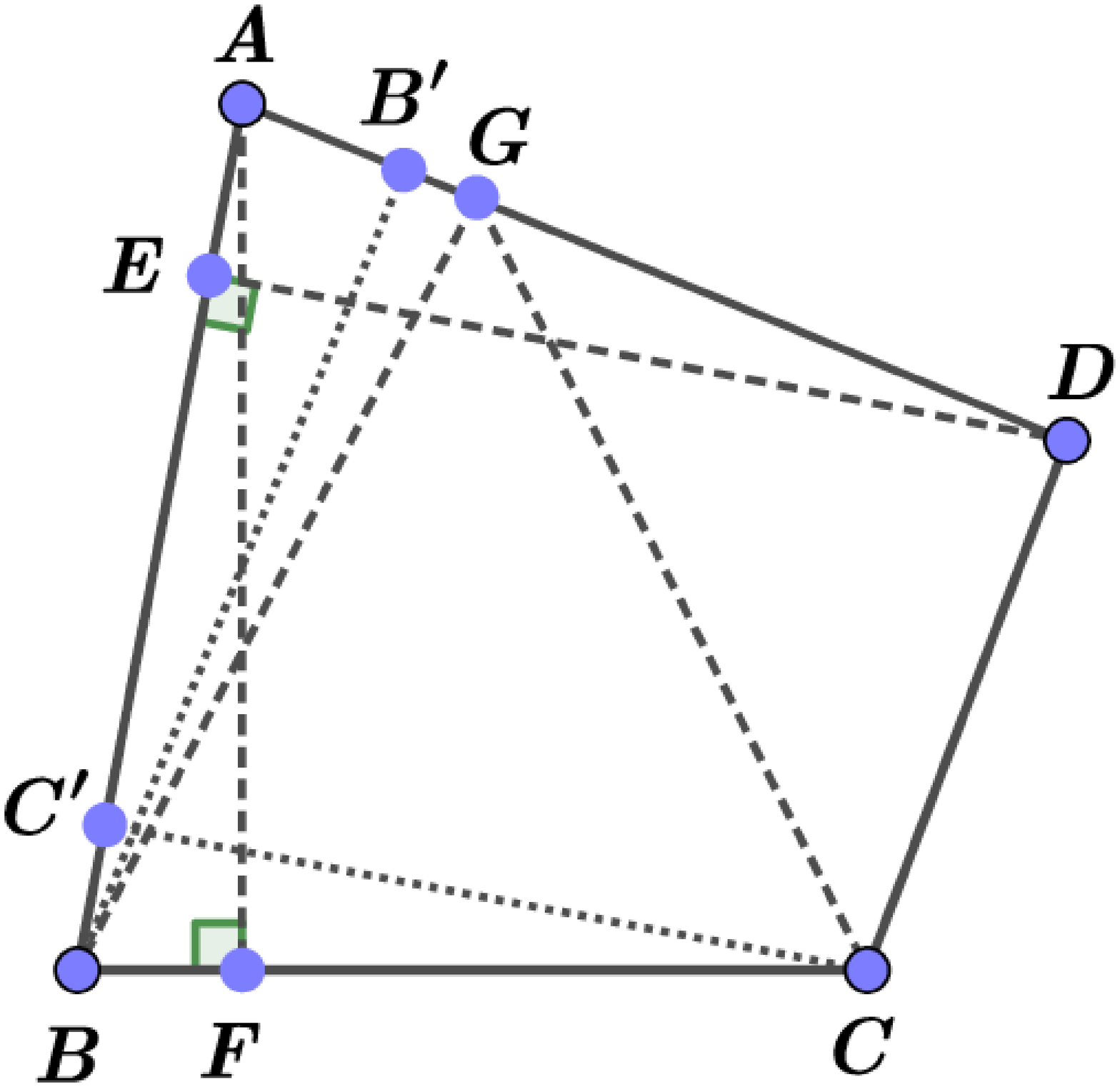}} b) \\
\end{minipage}
\caption{a) Case 3.2a; b) Case 3.2b.}
\label{case3_2}
\end{figure}

We put  $\angle GCB=\angle GBC=:\varphi$,  $\angle BAD=:\alpha$ and $\angle BAF=:\beta$ (see the left panel of Fig. \ref{case3_2}).
Let $F'$ is an interior point of $[A,D]$ such that the straight line $FF'$ is orthogonal to the straight line $[A,D]$.

Since $d(A,D)=1/\sin(\alpha)$ and
$d(A,F')=d(A,F)\cos(\alpha-\beta)=\cos(\alpha-\beta)$, then we have
$d(F',D)=d(A,D)-d(A,F')=1/\sin(\alpha)-\cos(\alpha-\beta)$.
It is clear that
$d(F',F)=d(A,F)\sin(\alpha-\beta)=\sin(\alpha-\beta)$.
Therefore, we get
$$
d^2(F,D)=d^2(F',D)+d^2(F',F)=
\left(\frac{1}{\sin(\alpha)}-\cos(\alpha-\beta)\right)^2+\sin^2(\alpha-\beta)
$$
$$
=\frac{1}{\sin^2(\alpha)}-\frac{2\cos(\alpha-\beta)}{\sin(\alpha)}+1=\frac{1-2\cos(\alpha-\beta)\sin(\alpha)}{\sin^2(\alpha)}+1.
$$
Hence, $d^2(F,D)\leq 1$ if and only if $2\cos(\alpha-\beta)\sin(\alpha)\geq 1$.

Since $\angle AGB \leq \pi/2$,
there is a point $B'\in [A,G]$, such that the straight line $B'B$ is orthogonal to the straight line $AD$
(see the right panel of Fig. \ref{case3_2}).
It is easy to see that $\angle B'BG= \angle ABF - \angle ABB' - \angle GBC= \pi/2-\beta-(\pi/2-\alpha)-\varphi=\alpha-\beta-\varphi$.
Hence,  $\angle BGA=\pi/2 -(\alpha-\beta-\varphi)$.

The equations
$$
\frac{d(B,G)}{\sin(\alpha)}=\frac{1}{\sin(\alpha)}=\frac{d(A,B)}{\sin(\angle BGA)}=\frac{1}{\cos(\beta)\cos(\alpha-\beta-\varphi)}
$$
imply $\cos(\alpha-\beta-\varphi)=\sin(\alpha)/\cos(\beta)$. Therefore,
$\varphi=\alpha-\beta-\arccos\left(\frac{\sin(\alpha)}{\cos(\beta)}\right)$
and, consequently,
$$
\cos(\varphi)=\cos(\alpha-\beta)\frac{\sin(\alpha)}{\cos(\beta)}+\sin (\alpha-\beta)\sin \left(\arccos\left(\frac{\sin(\alpha)}{\cos(\beta)}\right)\right)
$$
$$
=\frac{1}{\cos(\beta)}(\cos(\alpha-\beta)\sin(\alpha)+\sin(\alpha-\beta)\sqrt{\cos^2(\beta)-\sin^2(\alpha)}).
$$

Let $C'$ is an interior point of $[A,B]$ such that the straight line $CC'$ is orthogonal to the straight line $AB$
(see the right panel of Fig. \ref{case3_2}).
The equations
$$
d(B,C)=2\cos(\varphi)=\frac{d(C,C')}{\sin(\pi/2-\beta)}=\frac{d(C,C')}{\cos(\beta)}
$$
imply $d(C,C')=2\cos(\varphi)\cos(\beta)$.

It is easy to see that
$d(B,C')=d(B,C)\cos(\pi/2-\beta)=2\cos(\varphi)\sin(\beta)$,
$d(A,E)=\cot(\alpha)$ and
$d(A,B)=\frac{1}{\cos(\beta)}$.
Hence,
$$
d(E,C')=d(A,B)-d(A,E)-d(B,C')=\frac{1}{\cos(\beta)}-\cot(\alpha)-2\cos(\varphi)\sin(\beta).
$$

Therefore, we get
$$
d^2(E,C)=d^2(E,C')+d^2(C,C')
$$
$$
=\left(\frac{1}{\cos(\beta)}-\cot(\alpha)-2\cos(\varphi)\sin(\beta)\right)^2+4\cos^2(\varphi)\cos^2(\beta).
$$

Let us now show that the inequalities $d^2(F,D)\leq 1$ and $d^2(E,C)\leq 1$ cannot hold simultaneously.
For this goal, we prove the following simple result.

\begin{lemma}\label{le.c.3.2.1}
Let $M$ be a topological space and $X, Y$ are closed subsets of $M$.
Let the following conditions be satisfied:
\begin{itemize}
\item $Y$  be a connected set;
\item $\bd(X)\cap Y=\emptyset$;
\item There exists $x\in Y$ such that $x\notin X$.
\end{itemize}
Then $X\cap Y=\emptyset$.
\end{lemma}

\begin{proof}
Let $X_1=X{\setminus}\bd(X)$ and $X_2=M{\setminus}X$. It is clear that $Y\subset X_1\cup X_2$ and
$X_1, X_2$ are open and disjoint sets. As $Y$ is a connected set, we have $Y\subset X_1$ or $Y\subset X_2$. By the third property, $Y\subset X_1$ is impossible.
Hence, $Y\subset X_2$
and $X\cap Y=\emptyset$.
\end{proof}
\smallskip

Let $Y$ the connected set of solutions $(\alpha,\beta) \in [0,\pi/2]^2$ to the inequality $d^2(F,D)\leq 1$ (i.e. $2\cos(\alpha-\beta)\sin(\alpha)\geq 1$)
and $X$ the set of solutions to the inequality $d^2(E,C)\leq 1$.
If $2\cos(\alpha-\beta)\sin(\alpha)=1$ then $\beta=-\arccos \left(\frac{1}{2\sin(\alpha)}\right)+\alpha$
and
$$
d^2(C,C')=\left( \sqrt {\frac {4\sin^2(\alpha)-1}{\sin^2(\alpha)}}\sqrt{
 \cos \left( -\arccos \left(\frac{1}{2\sin(\alpha)}\right) +\alpha\right) ^2-\sin^2(\alpha)}+1 \right) ^2>1.
$$
Hence, $d^2(E,C)>1$ and $\bd(X)\cap Y=\emptyset$. It is easy to see that  $x\in Y$ and $x\notin X$ for the point $x=(\pi/4,\pi/4)$.
Therefore, the inequalities $d^2(F,D)\leq 1$ and $d^2(E,C)\leq 1$ cannot hold simultaneously. Hence, the case under consideration is impossible.

\medskip

\subsection{Case 3.3}
Suppose that $\|G_{\min}\|=3$ and $G_{\min}=\{E, G, H\}$ (see Fig.~\ref{case3_3}).
Without loss of generality, we may suppose that $E$ is an interior point of $[A,B]$,
$G$ is an interior point of $[A,D]$, $H$ is an interior point of $[C,D]$, and
$\delta(\bd(ABCD))=d(E,D)=d(G,B)=d(G,C)=d(H,A)=d(H,B)=1$. It is clear that $ED$ is orthogonal to $[A,B]$.
Moreover, $\angle AGB \leq \pi/2$ and $\angle DGC \leq \pi/2$,
$\angle AHD \leq \pi/2$ and $\angle BHC \leq \pi/2$ by Proposition \ref{locmin_mu1}.

\begin{figure}[t]
\center{\includegraphics[width=0.62\textwidth]{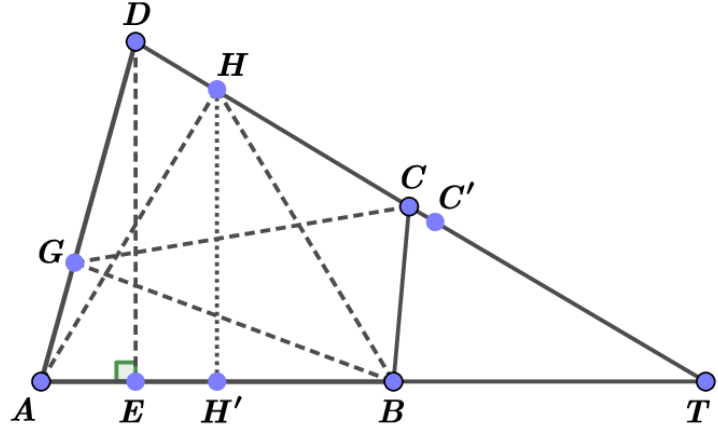}}\\
\caption{Case 3.3.}
\label{case3_3}
\end{figure}

Let us consider the quadrilateral $ABCD$ as a subset of the triangle $ATD$
(see Fig.~\ref{case3_3}).
We put $\alpha:=\angle BAD$, $\beta:=\angle HAB=\angle HBA$, and $\gamma:=\angle ATD$.
Since $\angle BAD > \angle HAB$ and $\angle HBA > \angle ATD$, we get $\alpha > \beta > \gamma >0$.

Since $\beta+\gamma=\angle HAT +\angle HTA =\angle AHD \leq \pi/2$, we get $\gamma \leq \pi/2 -\beta$ and $\cot(\gamma) \geq \cot(\pi/2 -\beta)=\tan(\beta)$.
It is clear that $d(A,D)\cdot \sin \angle ADT$, the distance  from the point $A$ to the straight line $DT$, is at most $d(A,H)=1=d(D,E)=d(A,D) \sin \angle BAD$.
Therefore, $\sin \angle ADT=\sin(\pi-\alpha-\gamma)=\sin(\alpha+\gamma)\leq \sin(\alpha)=\sin \angle BAD$.
It implies that $\alpha+\gamma >\pi/2$ and $\pi-(\alpha+\gamma)\leq \alpha$. Hence, we get $2\alpha \geq\pi - \gamma \geq \pi -(\pi/2-\beta)=\pi/2+\beta$ and
\begin{equation}\label{eq.L1.bound}
\alpha \geq \beta/2 +\pi/4.
\end{equation}

Let $H'$ be the midpoint of $[A,B]$. Since triangles $TH'H$ and $TED$ are similar, then we have $\frac{d(T,H')}{d(E,T)}=\frac{d(H,H')}{d(E,D)}=\sin(\beta)$.
Hence $d(T,H')=\sin(\beta)\cot(\gamma)$.
Further on,
$$
d(E,H')=d(T,E)-d(T,H')=\cot(\gamma)-\sin(\beta)\cot(\gamma)=\bigl(1-\sin(\beta)\bigr)\cot(\gamma),
$$ and
$$
d(E,A)=d(A,H')-d(E,H')=\cos(\beta)-\bigl(1-\sin(\beta)\bigr)\cot(\gamma).
$$
Moreover, we get $\cot(\alpha)=\frac{d(E,A)}{d(E,D)}=\cos(\beta)-\bigl(1-\sin(\beta)\bigr)\cot(\gamma)$, which implies
\begin{equation}\label{eq.gamma.expl}
\gamma= \arccot\left(\frac{\cos(\beta)-\cot(\alpha)}{1-\sin(\beta)}\right).
\end{equation}

From the triangle $ABG$ we have
$$
\frac{d(A,B)}{\sin(\angle AGB)}=\frac{d(A,G)}{\sin(\angle ABG)}=\frac{d(B,G)}{\sin(\angle BAG)}=\frac{1}{\sin(\alpha)}=d(A,D)
$$
and, consequently, $2\cos(\beta)\sin(\alpha)=\sin(\angle AGB)\leq 1$ and
$$
d(A,G)\cdot\sin(\alpha) =\sin(\angle ABG)=\sin\bigl(\alpha+\arcsin(2\sin(\alpha)\cos(\beta))\bigr).
$$
Then we get
\begin{equation}\label{eq.L2.bound}
\beta \geq \arccos \left( \frac{1}{2 \sin(\alpha)}\right),
\end{equation}

$$
d(G,D)=d(A,D)-d(A,G)=\Bigl(1-\sin\bigl(\alpha+\arcsin(2\sin(\alpha)\cos(\beta))\bigr)\Bigr)\cdot \bigl(\sin(\alpha)\bigr)^{-1}.
$$

Let us consider  $\theta:=\angle GCD$ and $p:=1-\sin\bigl(\alpha+\arcsin(2\sin(\alpha)\cos(\beta))\bigr)$.
It is clear that $\angle CGD= \pi-(\pi-\alpha-\gamma)-\theta=\alpha+\gamma-\theta$.
From the triangle $DCG$ we get equalities
$$
\frac{d(G,D)}{\sin(\theta)}=\frac{d(D,C)}{\sin(\alpha+\gamma-\theta)}=\frac{d(G,C)}{\sin(\angle GDC)}=\frac{1}{\sin(\alpha+\gamma)}
$$
imply
$$
\sin(\theta)=\sin(\alpha+\gamma)\cdot d(G,D)=
\frac{\sin(\alpha+\gamma)}{\sin(\alpha)} \cdot p\quad \mbox{and}\quad d(D,C)=\frac{\sin(\alpha+\gamma-\theta)}{\sin(\alpha+\gamma)}\,.
$$

It is clear that $\angle BGC= \pi- \angle AGB-\angle CGD=\pi -\alpha-\gamma+\theta-\arcsin(2\sin(\alpha)\cos(\beta))$. This implies
$$
d(B,C)=2\sin\left(\frac{\angle BGC}{2}\right)=2\cos\left(\frac{\alpha+\gamma-\theta+\arcsin(2\sin(\alpha)\cos(\beta))}{2}\right).
$$

Therefore, we get the following expression for the perimeter of the quadrilateral $ABCD$:
\begin{eqnarray*}
L\bigl(\bd(ABCD)\bigr)=d(A,B)+d(A,D)+d(D,C)+d(B,C)\\
=2\cos(\beta)+\frac{1}{\sin(\alpha)}+
\frac{\sin(\alpha+\gamma-\theta)}{\sin(\alpha+\gamma)}+
2\cos\left(\frac{\alpha+\gamma-\theta+\arcsin(2\sin(\alpha)\cos(\beta))}{2}\right).
\end{eqnarray*}

We can explicitly express $\gamma$ and $\theta$ in terms of $\alpha$ and $\beta$
from \eqref{eq.gamma.expl} and the equality
\begin{eqnarray*}
\sin(\theta)&=&\frac{\sin(\alpha+\gamma)}{\sin(\alpha)} \cdot p=\frac{\sin(\alpha+\gamma)}{\sin(\alpha)}
\Bigl(1-\sin \bigl(\alpha+\arcsin(2\sin(\alpha)\cos(\beta))\bigr)\Bigr).
\end{eqnarray*}

Now we going to obtain some estimation on possible values of $\alpha$ and $\beta$.
Note that
\begin{eqnarray*}
\cot(\alpha)&=\cos(\beta)-\bigl(1-\sin(\beta)\bigr)\cot(\gamma)\leq \cos(\beta)-\bigl(1-\sin(\beta)\bigr)\tan(\beta)=\frac{1-\sin(\beta)}{\cos(\beta)}
\end{eqnarray*}
due to  $\cot(\gamma) \geq \cot(\pi/2-\beta)=\tan(\beta)$.
Note also that
$\cot^2(\alpha)=\frac{1}{\sin^2(\alpha)}-1\geq 4\cos^2(\beta)-1$ (due to $2\sin(\alpha)\cos(\beta)\leq 1$), therefore we have
$(1-\sin(\beta))^2 \geq \cos^2(\beta) \bigl( 4\cos^2(\beta)-1\bigr)$ or, dividing by $1-\sin(\beta)>0$,
$$
1+2\sin(\beta)-2 \sin^2(\beta)-2\sin^3(\beta)\leq 0.
$$
Hence, $\sin(\beta)\geq y_0=0.8546376797...$ and $\beta \geq \arcsin(y_0)=1.024852629...$, where $y_0$ is a unique root of the polynomial $2y^3+2y^2-2y-1$ on
$[0,1]$.
Therefore, we get that $\alpha \in  [\alpha_0, \pi/2]$ and $\beta \in [\beta_0, \pi/2]$,
where  $\alpha_0=\arcsin(y_0)/2+\pi/4=1.29782448...$, $\beta_0=\arcsin(y_0)=1.024852629...$.

\medskip
Now we will find additional important restrictions on the values of $\alpha$, $\beta$.
We are going to prove that any possible pair $(\alpha,\beta)$ should lie in the following set:
\begin{equation}\label{eq.triang.curv}
\Omega=\left\{(\alpha,\beta)\,\,\left|\,\, \pi/2 \geq \alpha \geq \beta/2+\pi/4, \,\,\pi/3 \geq \beta \geq \arccos \left( \frac{1}{2 \sin(\alpha)}\right)\right.  \right\}.
\end{equation}
It is easy to check that $\Omega$ is bounded by the straight lines  $L_1=\left\{(\alpha,\beta)\,|\, \alpha=\beta/2+\pi/4\right\}$, $\beta=\pi/3$, and by the curve
$$
L_2=\left\{(\alpha,\beta)\,|\, 2\sin(\alpha)\cos(\beta)=1,\,\, \pi/2 \geq \alpha \geq \beta\geq\arcsin(y_0)\right\}.
$$
The above argument implies that the point $M:=(\alpha_0,\beta_0)$ lies in the intersection of $L_1$ and $L_2$. It is easy to check also that $L_1\cap L_2 =\{M\}$.

Recall that we have obtained the inequalities $\alpha \geq \beta/2+\pi/4$ and $2\sin(\alpha)\cos(\beta)\leq 1$. Hence it suffices to prove that $\beta \leq \pi/3$.
In any case, all possible pairs $(\alpha,\beta)$ lie in the set
\begin{equation}\label{eq.btriang.curv}
\Pi=\left\{(\alpha,\beta)\,\,\left|\,\, \pi/2 \geq \alpha \geq \beta/2+\pi/4, \,\,\pi/2 \geq \beta \geq \arccos \left( \frac{1}{2 \sin(\alpha)}\right) \right. \right\}.
\end{equation}
Note that the straight line  $L_1$ passes through the points $(\pi/2,\pi/2)$,
hence we can omit the restriction $\beta \leq \pi/2$ from the formal point of view.
\smallskip

Let us consider the point $C'$ on $[C,T]$ such that $d(E,C')=1$ (recall that $d(E,C) \leq 1$ and $\gamma <\pi/4$, hence, $d(E,T) >d(E,D)=1$).
Note that $d(E,C)\leq d(E,C')$ holds if and only if $d(D,C)\leq d(D,C')$ or, equivalently, $d(G,C')\geq 1$.
We can rewrite the latter inequality using the law of cosines the triangle $GDC'$.
It is clear that $d(D,C')=2\sin(\gamma)$ (recall that $d(E,C')=d(E,D)=1$ and $\angle EDC'=\pi/2-\gamma$). Therefore,
$$
d(D,C')^2+2d(G,D)d(D,C')\cos(\alpha+\gamma)+d(G,D)^2=d(G,C')^2\geq 1,
$$
which can be rewritten $\Bigl(d(G,D)\cdot \sin(\alpha)=p=1-\sin\bigl(\alpha+\arcsin(2\sin(\alpha)\cos(\beta))\bigr)\Bigr)$ as
$$
4\sin^2(\gamma)+\frac{4\sin(\gamma)\cos(\alpha+\gamma)}{\sin(\alpha)}\,p+\frac{p^2}{\sin^2(\alpha)}\geq 1.
$$
Now, we define the function

\begin{figure}[t]
\center{\includegraphics[width=0.7\textwidth]{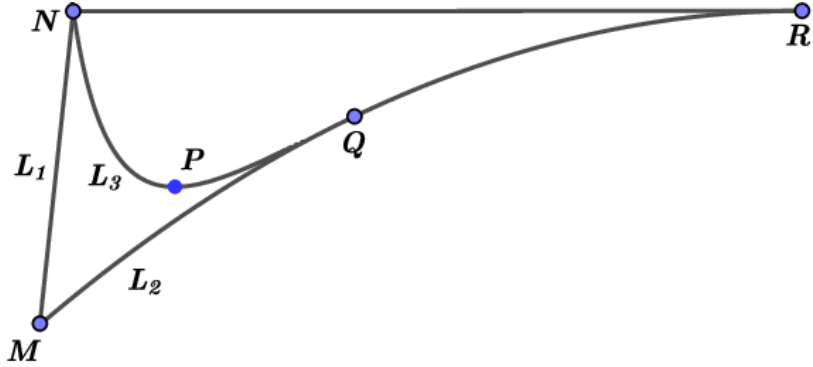}}\\
\caption{The curvilinear triangle MNR.}
\label{treug}
\end{figure}

\begin{equation}\label{eq.ec}
F(\alpha,\beta):=4\sin^2(\gamma)+\frac{4\sin(\gamma)\cos(\alpha+\gamma)}{\sin(\alpha)}\,p+\frac{p^2}{\sin^2(\alpha)}-1,
\end{equation}
where
$$
p=1-\sin\bigl(\alpha+\arcsin(2\sin(\alpha)\cos(\beta))\bigr),\quad \gamma=\arccot\left(\frac{\cos(\beta)-\cot(\alpha)}{1-\sin(\beta)}\right),
$$
and the set
$$
L_3=\left\{(\alpha,\beta)\in \Pi\,|\, F(\alpha,\beta)=0\right\}.
$$

Now we will find the intersection point $N=(\alpha_1,\beta_1)$ of the curve $L_1$ and the set $L_3$.
Substituting $\alpha=\beta/2+\pi/4$ into $\eqref{eq.ec}$,
we easily get that $\beta_1=\pi/3$ and, consequently, $\alpha_1=\pi/6+\pi/4=5\pi/12$.

Now we will find the intersection of the curve $L_2$ and the set $L_3$.
Substituting $\cos(\beta)=\frac{1}{2\sin(\alpha)}$ and $\sin(\beta)=\frac{\sqrt{4\sin^2(\alpha)-1}}{2\sin(\alpha)}$ into $\eqref{eq.ec}$ we get
$$
\frac{\sin(\alpha)\sqrt{3-4\cos^2(\alpha)}\Bigl(2\cos(\alpha)+3\Bigr)+4\cos^3(\alpha)+
8\cos^2(\alpha)-5\cos(\alpha)-5}
{\Bigl(\cos(\alpha)-\cos^2(\alpha)\Bigr)^{-1}\sin^2(\alpha) \Bigl(\sin(\alpha)\sqrt{3-4\cos^2(\alpha)}+\cos(\alpha)^2+\cos(\alpha)-2\Bigr)}=0.
$$

\begin{figure}[t]
\begin{minipage}[h]{0.45\textwidth}
\center{\includegraphics[width=0.95\textwidth, trim=1mm 0mm 0mm 0mm, clip]{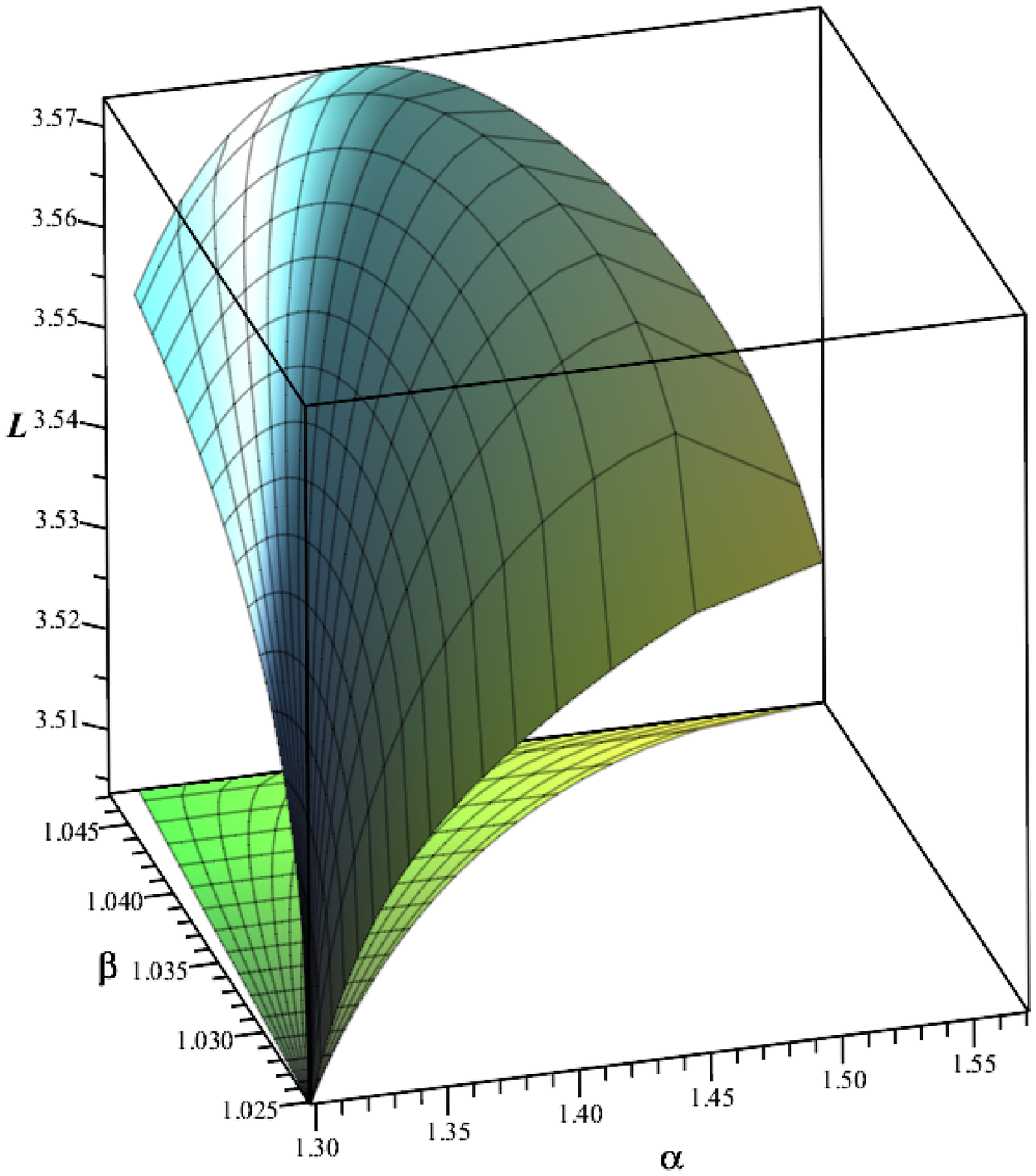}}
\end{minipage}
\hspace{1mm}
\begin{minipage}[h]{0.48\textwidth}
\center{\includegraphics[width=0.9\textwidth, trim=1mm 0mm 0mm 0mm, clip]{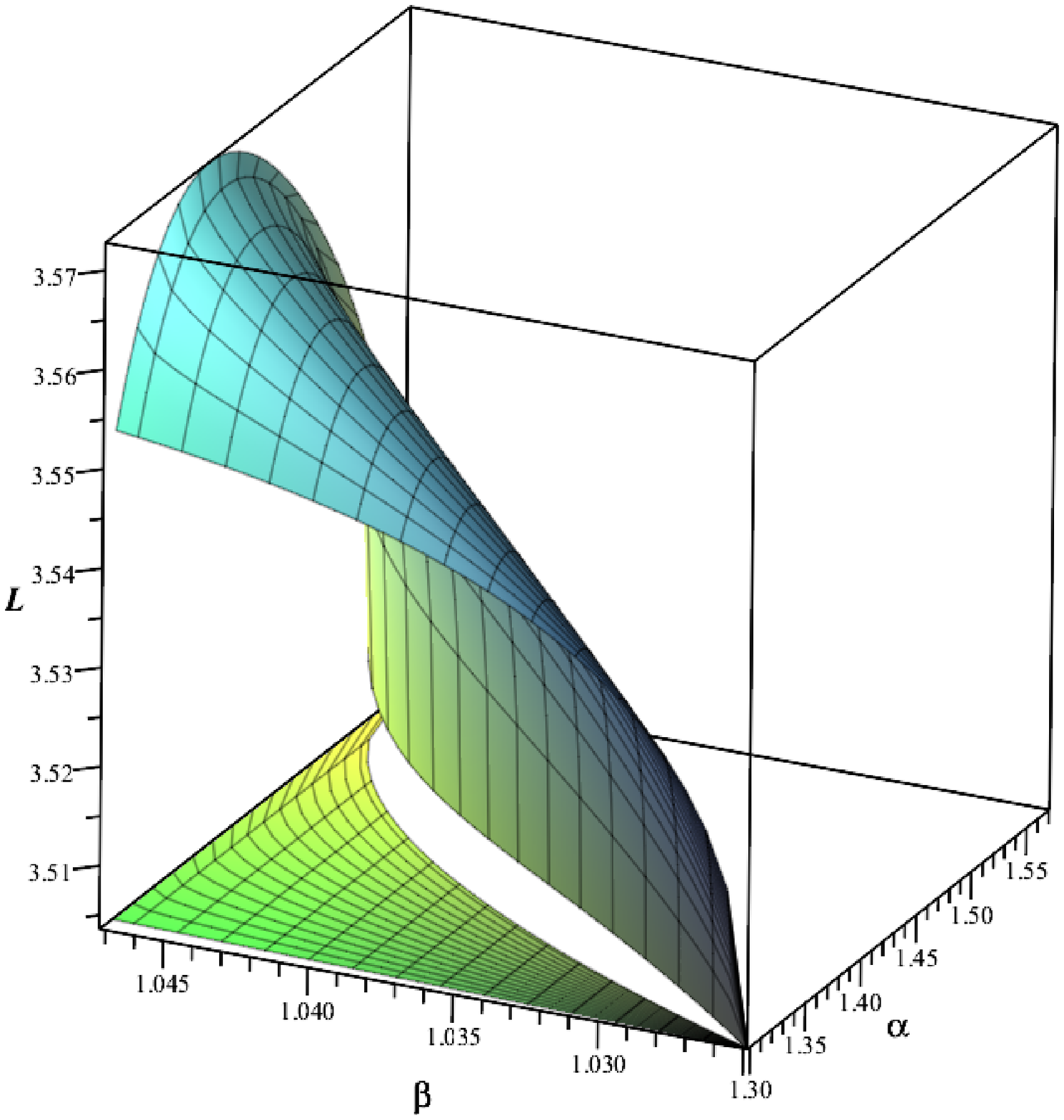}}
\end{minipage}
\caption{The perimeter $L$ of $ABCD$ over MNR.}
\label{3dtreug}
\end{figure}

Hence, either $\alpha=\pi/2$ or
$$
2\bigl(1-2\cos^2(\alpha)\bigr)\Bigl(4\cos^3(\alpha)+4\cos^2(\alpha)-7\cos(\alpha)+1 \Bigr)=0.
$$

By direct computations, we get that
$\alpha=\alpha_2:=\arccos(t')=1.4103331...$,
where $t'=0.1597754...$ is a unique root of the polynomial $f(t)=4t^3+4t^2-7t+1=0$ on the interval $(0,0.2)$.
Consequently, we obtain the intersection points $Q=(\alpha_2,\beta_2)$, where $\beta_2=\arccos\left(\frac{1}{2\sin(\alpha_2)}\right)=1.039667597...$, and $R=(\pi/2,\pi/3)$.

If we consider $h(\alpha):=F(\alpha,\pi/3)$ for $\alpha \in [\alpha_1=5\pi/12, \pi/2]$,
then $h(\alpha_1)=h(\pi/2)=0$ and $h(\alpha)<0$ for $\alpha \in (\alpha_1, \pi/2)$.
Let us suppose that there exist $(\alpha',\beta')\in \Pi$ such that $F(\alpha', \beta') \geq 0$ and $\beta' >\pi/3$.
Then, by the continuity,  there exists $(\alpha'',\beta'')\in \Pi$ such that $F(\alpha'', \beta'') = 0$ and $\beta'' >\pi/3$.
Without loss of generality, we can fix such a point $(\alpha'',\beta'')$ with the maximal possible value of $\beta''$.
Its clear that $\frac{\partial}{\partial \alpha} F(\alpha'',\beta'')=0$.
On the other hand, the system
\begin{equation}\label{eq.inter.diff}
F(\alpha'',\beta'')=\frac{\partial}{\partial \alpha} F(\alpha'',\beta'')=0
\end{equation}
has no solution in $\Pi$ with $\beta \geq \pi/3$.
Since we have $\alpha \geq \beta/2 +\pi/4$ by
\eqref{eq.L1.bound}, hence, we need to check only $(\alpha,\beta)\in [5\pi/12,\pi/2]\times [\pi/3,\pi/2]$.
By numerical calculations (one can use commands {\it Maximize} and {\it Minimize} for the {\it Optimization} package in Maple)
we get that the maximal value of $\frac{\partial}{\partial \alpha} F(\alpha,\beta)$
for $(\alpha,\beta)\in [5\pi/12,\pi/2]\times [11/10,\pi/2]$ is $-1.09119149... <0$.
On the other hand, the minimal value of $F(\alpha,\beta)$ for $(\alpha,\beta)\in [5\pi/12,\pi/2]\times [\pi/3, 6/5]$ is $0.79177743... >0$.
Since $6/5>11/10$, we obtain that system~\eqref{eq.inter.diff} has no solution for $(\alpha,\beta)\in [5\pi/12,\pi/2]\times [\pi/3,\pi/2]$.

Therefore, for all possible pair $(\alpha, \beta) \in \Pi$ we have $\beta \leq \pi/3$.
Thus we have constructed the (curvilinear) triangle
\begin{equation}\label{eq.curvtr.im}
MNR =\left\{ (\alpha,\beta)\,\,\left|\,\, \alpha \geq \beta/2+\pi/4, \,\,\pi/3 \geq \beta \geq \arccos \left( \frac{1}{2 \sin(\alpha)}\right)\right.\right\},
\end{equation}
where $NR$ is the straight line $\beta=\pi/3$ (see Fig. \ref{treug}).

By numerical calculations we get also that the minimum value
of the value $\beta$  on the set $L_3 \cap \Pi$ is $\beta_3=1.034584325...$, that corresponds to
$\alpha_3=1.34546261...$ (see the point $P=(\alpha_3,\beta_3)$ at Fig. \ref{treug}, this point gives a unique solution of \eqref{eq.inter.diff} in $\Pi$).

\medskip

Now, we consider only pair $(\alpha, \beta)$ from the curvilinear triangle $MNR$, since this inclusion follows from the equality $\delta(\bd(ABCD))=1$.

The results of numerical calculations show
that the perimeter $L(\bd(ABCD))$ (on the curvilinear triangle $MNR$) achieves its minimal value
provided that $\angle AGB =\angle AHD=\pi/2$, see Fig. \ref{3dtreug}.
This minimal value is $L(\bd(ABCD))=3.503591801...$. Hence, the quadrilateral $ABCD$ is not extremal.
\medskip

\begin{remark}
It should be noted that in Сase 3.3, in contrast to Сase 3.2, there are real quadrilaterals with all the necessary properties,
therefore, a lower estimate for the value of the perimeter is required.
\end{remark}

\begin{remark}
In Case 3.3 we ended the discussion with a reference to numerical calculations. This is sufficient, since the minimum value
of the perimeter in this case is separated from the hypothetical minimum value (obtained for magic kites) by a distance greater than $1/10$.
It is clear that instead of referring to numerical calculations, one can carry out perimeter estimates (with some margin),
using only symbolic calculations (the function being minimized is analytic and has good properties).
But this will take additional space in the paper, although it is in fact some technical point. The same remark is valid for the following Case 3.4.
\end{remark}

\subsection{Case 3.4}
Without loss of generality, we may suppose that $G_{\min}=\{E, G, F\}$,
where $E$ is an interior point of $[A,B]$,
$G$ is an interior point of $[B,C]$, $F$ is an interior point of $[A,D]$, and
$\delta(\bd(ABCD))=d(E,D)=d(E,C)=d(G,A)=d(G,D)=d(F,B)=d(F,C)=1$
(see Fig. \ref{case3_4}).
Moreover, $\angle AGB \leq \pi/2$, $\angle AFB \leq \pi/2$ and
$\angle AED \leq \pi/2$  by Proposition \ref{locmin_mu1}.

\begin{figure}[t]
\center{\includegraphics[width=0.375\textwidth, trim=0mm 1mm 0mm 0mm, clip]{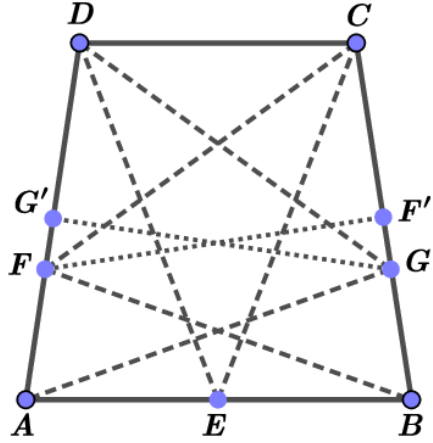}}\\
\caption{Case 3.4.}
\label{case3_4}
\end{figure}

Let us consider a slightly more general case in a suitable system of Cartesian coordinates.
We take the coordinates as follows: $E=(0,0)$,  $C=(\cos(\alpha),\sin(\alpha))$, $D=(-\cos(\alpha),\sin(\alpha))$, where $\alpha\in(0,\pi/2)$.
Let $k$ be a number such that the equation of the straight line  $AB$ is $y=kx$.
Without loss of generality, we may suppose that $k\geq0$.
Then $A=(x_1,kx_1)$, $B=(x_2,kx_2)$ and  $x_1<0<x_2$.

If $\alpha \leq \pi/4$, then $\angle DEC\geq \pi/2$ and we obtain the following simple estimate for the perimeter:
$L(\bd(DEC))\geq 2+\sqrt{2}$. Hence, we have  $L(\bd(ABCD)) > 2+\sqrt{2}$, that is impossible for an extremal quadrilateral. Thus, $\alpha \geq \pi/4$.

Since $\angle DEB\geq\pi/2$ then $(\overrightarrow{ED},\overrightarrow{EB})\leq0$. It implies
$-\cos(\alpha)x_2+\sin(\alpha)kx_2\leq0$ or, equivalently, $k\leq\frac{\cos(\alpha)}{\sin(\alpha)}=\cot(\alpha)$.
Therefore, $0\leq k\leq \cot(\alpha)$.

It is easy to see that
$$
L(\bd(ABCD))=d(A,B)+d(B,C)+d(C,D)+d(D,A)=(x_2-x_1)\sqrt{1+k^2}+2\cos(\alpha)
$$
$$
+\sqrt{(x_2-\cos(\alpha))^2+(kx_2-\sin(\alpha))^2}+\sqrt{(x_1+\cos(\alpha))^2+(kx_1-\sin(\alpha))^2}.
$$

In particular, $L(\bd(ABCD))\geq x_2-x_1=|x_2-x_1|$.
Since $G$ is a point of $[B,C]$ and $F$ is a point of $[A,D]$, there are  $t\in[0,1]$ and $s\in[0,1]$ such that
\begin{eqnarray*}
F&=&\bigl(tx_1-(1-t)\cos(\alpha),\,tkx_1+(1-t)\sin(\alpha)\bigr),\\
G&=&\bigl(sx_2+(1-s)\cos(\alpha),\,skx_2+(1-s)\sin(\alpha)\bigr).
\end{eqnarray*}

Let $F'$ be the midpoint of $[B,C]$ and $G'$ be the midpoint of $[A,D]$. Then
$$
G'=\frac{1}{2}\bigl(x_1-\cos(\alpha),kx_1+\sin(\alpha)\bigr),\quad
F'=\frac{1}{2}\bigl(x_2+\cos(\alpha),kx_2+\sin(\alpha)\bigr).
$$

Since the straight line $G'G$ ($F'F$) is orthogonal to the straight line $AD$ (respectively, $BC$) then
$$
(\overrightarrow{G'G},\overrightarrow{AD})=0\quad\text{and}\quad(\overrightarrow{F'F},\overrightarrow{BC})=0.
$$
Solving these linear equations with respect to $s$ and $t$, we get some explicit expressions
$$
s=g(x_1,x_2,k,\alpha),\quad
t=f(x_1,x_2,k,\alpha).
$$

Since $d(G,D)=d(G,A)=1$ and $d(F,C)=d(F,B)=1$, we have the equations
\begin{equation*}
d(G,D)=1,\quad d(F,C)=1,
\end{equation*}
with respect to $x_1,x_2,k,\alpha$. Thus, here we have four variables and two constraints, that is, roughly speaking, two degrees of freedom.
Recall also that $x_1<0<x_2$, $\alpha\in [\pi/4,\pi/2)$, $0\leq k\leq \cot(\alpha)$, moreover, $s=g(x_1,x_2,k,\alpha), t=f(x_1,x_2,k,\alpha) \in [0,1]$,
and we can assume that $|x_2-x_1|\leq 2+\sqrt{2}$
(otherwise the perimeter of $ABCD$ is at least $2+\sqrt{2}$, and $ABCD$ is not extremal).

The results of numerical calculations show
that the perimeter $L\bigl(\bd(ABCD)\bigr)$ for the quadrilaterals under consideration reaches its minimal value
exactly for $\angle AED =\angle AFB=\pi/2$.
This minimal value is $L(\bd(ABCD))=3.510690988...$.
Hence, the quadrilateral $ABCD$ is not extremal.

\section{Case 4}\label{sect.5}

Let us suppose that $\|G_{\min}\|=4$ and $G_{\min}=\{E, F, G, H\}$.
Without loss of generality we may suppose that
$E\in [A,D]$, $F\in [A,B]$,  $G\in [B,C]$ and $H\in [C,D]$.
In all subsaces of this case we have only one degree of freedom.

Recall that a self Chebyshev center $x$ for $\bd(ABCD)$ is {\it simple} if $\|{\mathcal F}(x)\|=1$ and {\it non-simple}
if $\|{\mathcal F}(x)\|\geq 2$ (see (\ref{eq.fath1})).

\begin{figure}[t]
\begin{minipage}[h]{0.31\textwidth}
\center{\includegraphics[width=0.99\textwidth]{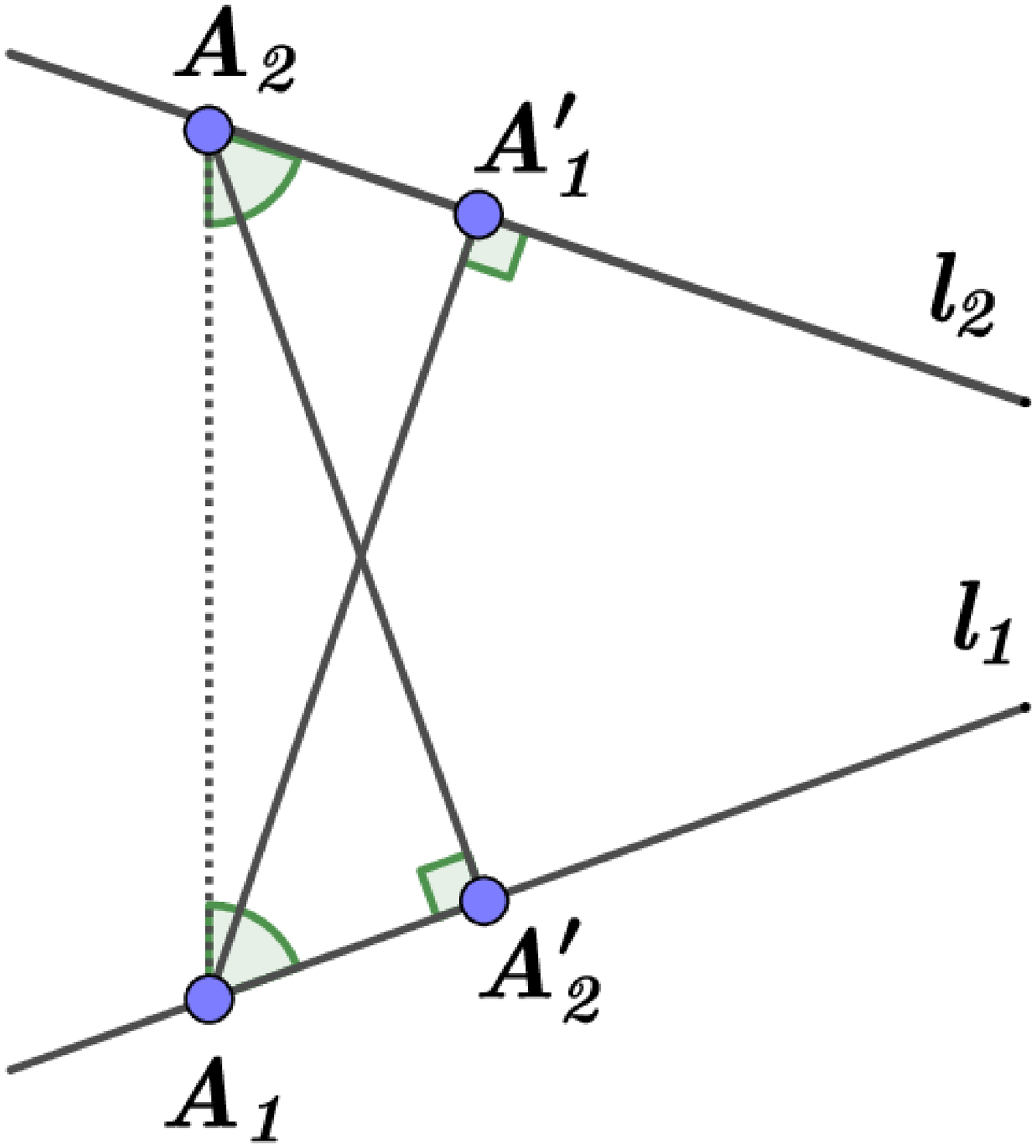}}   a) \\
\end{minipage}
\quad\quad
\begin{minipage}[h]{0.38\textwidth}
\center{\includegraphics[width=0.99\textwidth]{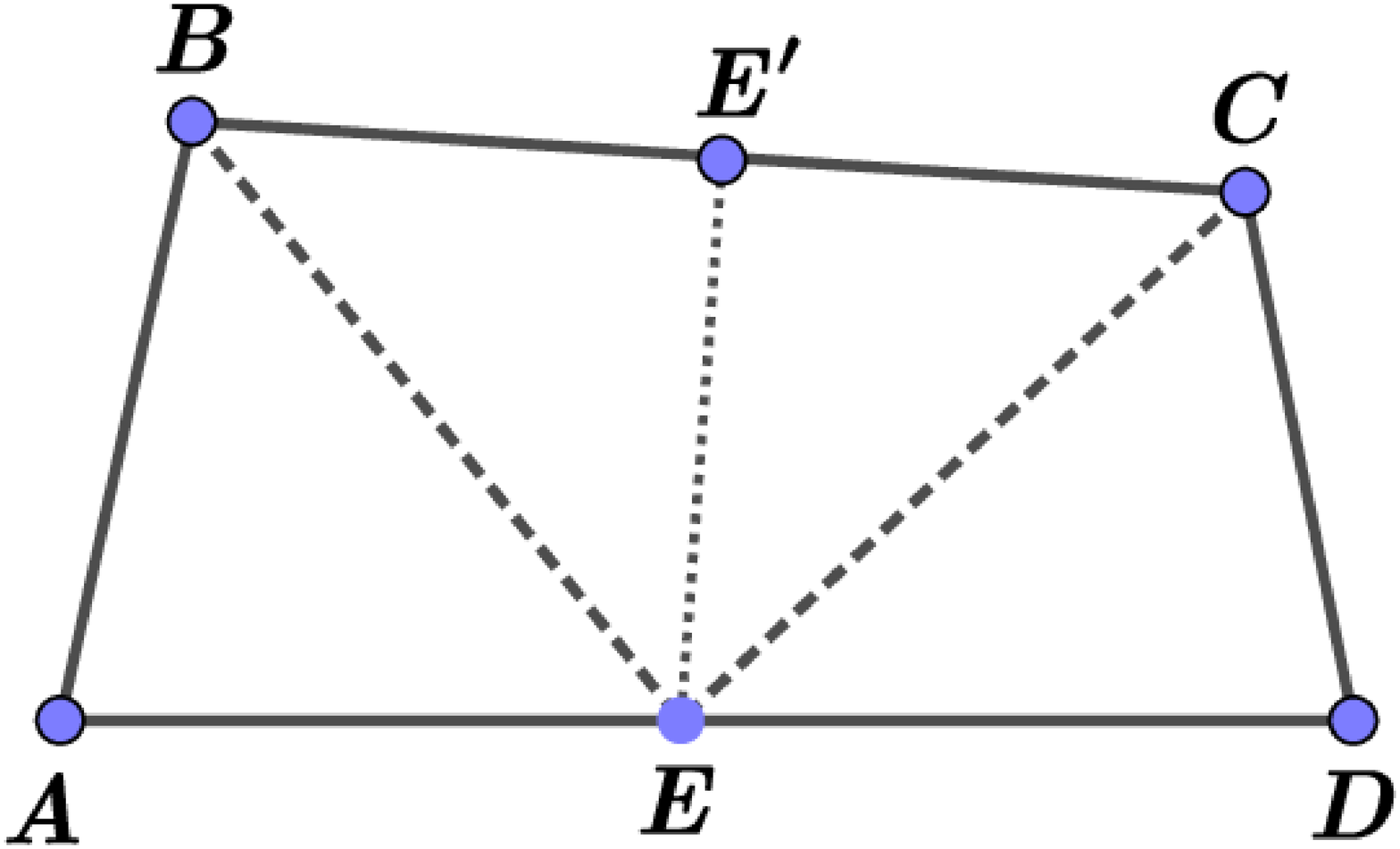}} b) \\
\end{minipage}
\caption{a) Two straight lines; b) A ``thin'' quadrilateral.}
\label{case4}
\end{figure}

\bigskip

Now, we are going to study a special case that we will call Case 4.0.

\subsection{Case 4.0.}
We consider a quadrilateral $ABCD$ with $\delta(\bd(ABCD))=1$ and $\|G_{\min}\|=4$, such that there are two opposite sides of $ABCD$ with
non-simple self Chebyshev centers of $\bd(ABCD)$. We will prove that such a quadrilateral can not be extremal. In other words, the following
proposition holds.

\begin{prop}\label{pro.case4.2nsimple}
Suppose that a quadrilateral $ABCD$ with $\delta(\bd(ABCD))=1$ is extremal and $\|G_{\min}\|=4$. Then there are no two opposite sides of $ABCD$
such that each of them has a non-simple self Chebyshev center of $\bd(ABCD)$.
\end{prop}

To prove this result, we need some lemmas.

\begin{lemma}
Let $l_1, l_2$ are straight lines with points $A_1\in l_1$, $A_2\in l_2$, $A'_1\in l_2$, $A'_2\in l_1$
that $A_1A'_1 \perp l_2$, $A_2A'_2 \perp l_1$ and $d(A_1,A'_1)=d(A_2,A'_2)$ (see the left panel of Fig. \ref{case4}).
If $[A_1,A'_1]\cap [A_2,A'_2]=\emptyset$ then $A_1A'_1A_2A'_2$ is rectangle, otherwise $A_1A_2A'_1A'_2$
is isosceles trapezoid.
\end{lemma}

\begin{proof}
The triangle $A_1A'_1A_2$  is congruent to triangle $A_2A'_2A_1$
in the common hypotenuse ($[A_1,A_2]$) and legs $(d(A_1,A'_1)=d(A_2,A'_2))$.
Hence $\angle A_2A_1A'_2=\angle A_1A_2A'_1$.
It implies that if $[A_1,A'_1]\cap [A_2,A'_2]=\emptyset$ then $A_1A'_1A_2A'_2$ is rectangle, otherwise $A_1A_2A'_1A'_2$ is isosceles trapezoid.
\end{proof}

\begin{lemma}\label{le.10}
Let $ABCD$ be a quadrilateral with $\delta(\bd(ABCD)) = 1$ for the self
Chebyshev radius of its boundary. If $d(A,B)\leq \frac{1}{\sqrt{2}}$ and $d(C,D)\leq \frac{1}{\sqrt{2}}$ then
either \linebreak $L(\bd(ABCD))>2+\sqrt{2}$ or $G_{\min} \cap [A,D]=\emptyset$ and $G_{\min} \cap [B,C]=\emptyset$.
\end{lemma}

\begin{proof}
Suppose that $E\in G_{\min}\cap [A,D]$.
If $\|{\mathcal F}(E)\|=1$ then ${\mathcal F}(E)=\{B\}$ or ${\mathcal F}(E)=\{C\}$. Hence $d(B,E)=1$ and $BE \perp [A,D]$
or $d(C,E)=1$ and $CE \perp [A,D]$.
But this is not possible because $\frac{1}{\sqrt{2}}\geq d(A,B)\geq d(B,E)=1$ or
$\frac{1}{\sqrt{2}}\geq d(C,D)\geq d(C,E)=1$. Thus, $\|{\mathcal F}(E)\|=2$ and ${\mathcal F}(E)=\{B,C\}$ (see the right panel Fig. \ref{case4}).

Let us consider the perpendiculars from the points $A,E,D$ to the straight line $BC$.
Denote by $A',E',D'\in BC$ the foots of these perpendiculars, respectively.
Since $d(A,A')\leq d(A,B)\leq \frac{1}{\sqrt{2}}$ and $d(D,D')\leq d(D,C)\leq \frac{1}{\sqrt{2}}$
then $d(E,E')\leq \max\{d(A,A'),d(D,D')\}\leq \frac{1}{\sqrt{2}}$.
It implies $d(B,E')=d(C,E')=\sqrt{1-d^2(E,E')}\geq  \frac{1}{\sqrt{2}}$.
Therefore, $L(\bd(ABCD))>L(\bd(BEC))\geq 2+\sqrt{2}$.

The same arguments show that $G_{\min} \cap [B,C]\neq \emptyset$  also implies $L(\bd(ABCD))> 2+\sqrt{2}$.
\end{proof}
\smallskip

The assertion of the following lemma is well known and belongs to folklore. At one time it was proposed as an Olympiad problem in mathematics
(see \cite[1997 Olympiad, Level~C, Problem~2]{Olymp}).

\begin{lemma}\label{le.le8}
Among quadrilaterals $ABCD$ with given diagonal lengths and the angle~$\theta$ between them, a parallelogram has the smallest perimeter.
\end{lemma}

\begin{figure}[t]
\begin{minipage}[h]{0.45\textwidth}
\center{\includegraphics[width=0.99\textwidth]{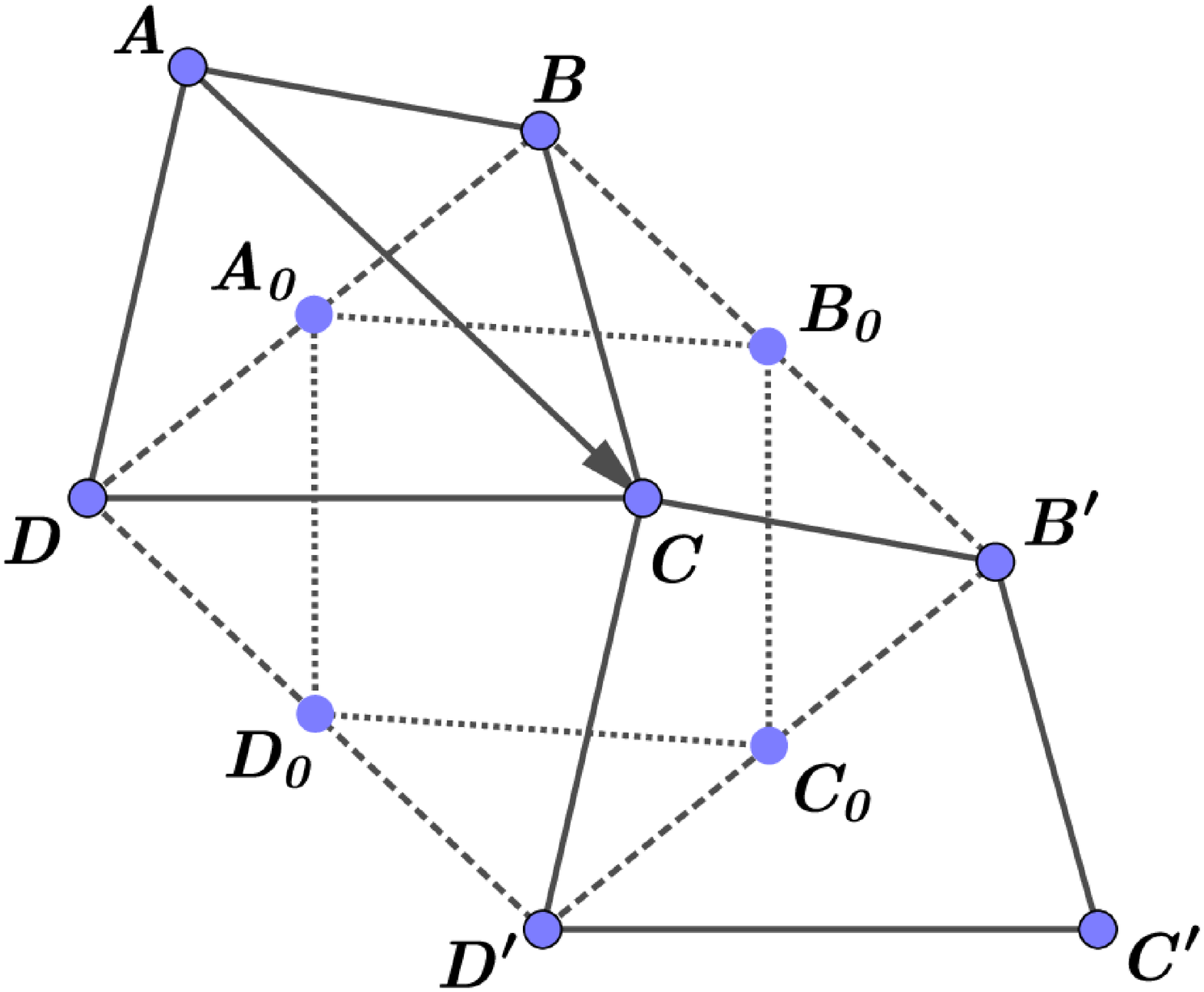}}    a) \\
\end{minipage}
\quad\quad
\begin{minipage}[h]{0.35\textwidth}
\center{\includegraphics[width=0.99\textwidth]{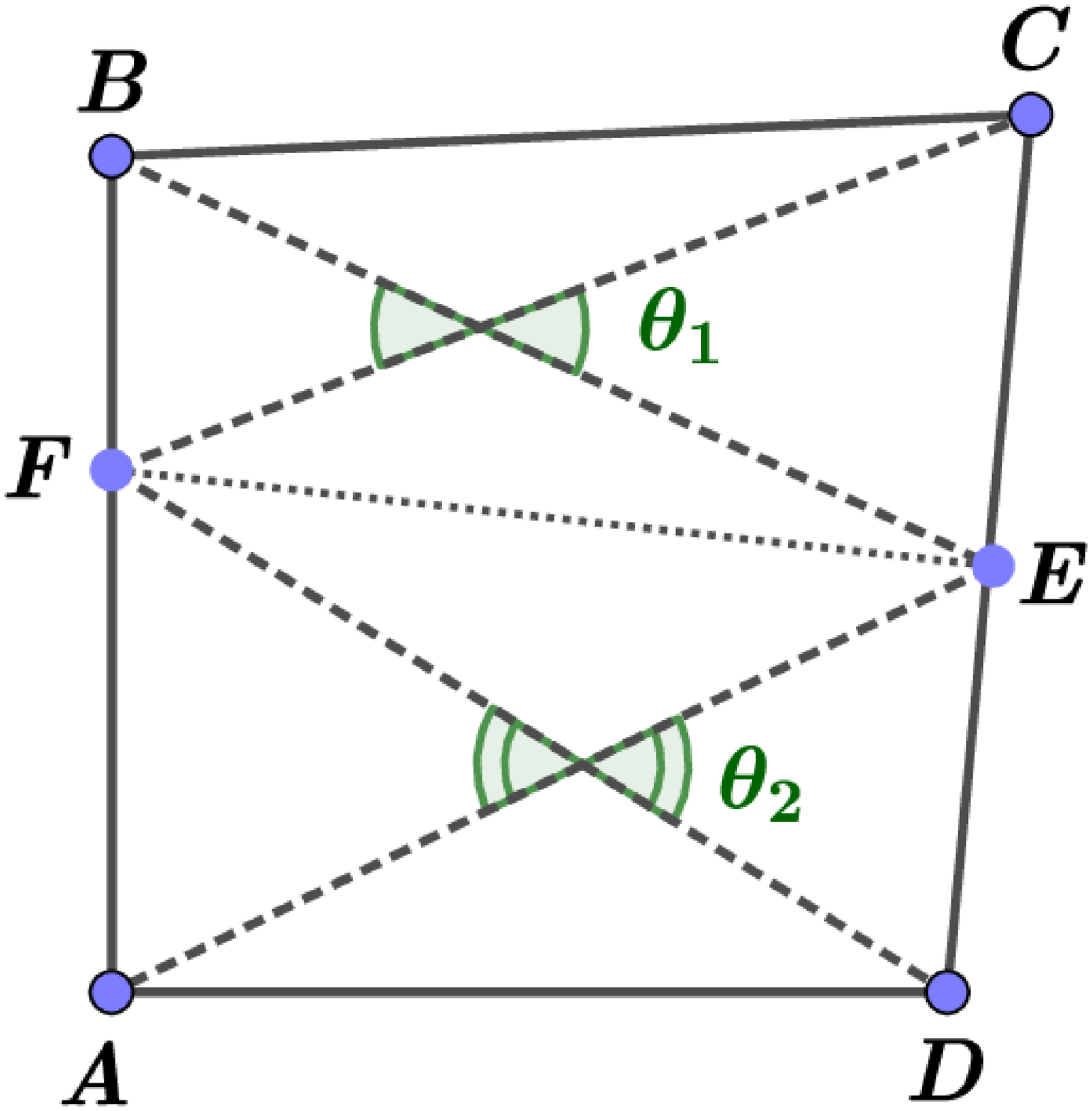}} b) \\
\end{minipage}
\caption{a) The proof of  Lemma \ref{le.le8};  b) The proof of Proposition \ref{pro.case4.2nsimple}.}
\label{olymp_f}
\end{figure}


\begin{proof}
We present the proof for completeness. Let us translate the quadrilateral $ABCD$ by vector $AC$ (see the left panel Fig. \ref{olymp_f}). We get the quadrilateral $A'B'C'D'$,
where $A'=C$ and the quadrilateral $BB'D'D$ is a parallelogram.

Let the points $A_0, B_0, C_0$ and $D_0$ are midpoints of $[B,D], [B,B'], [B'D']$ and $[D',D]$, respectively.
It is easy to see that the quadrilateral $A_0,B_0,C_0,D_0$ is a parallelogram. Moreover,
the sum of the lengths of the diagonals and the angle between them of which are equal
to the sum of the lengths of the diagonals and the angle between them of the quadrilateral $ABCD$.

By the triangle inequality $d(B,C)+d(C,D')\geq d(B,D')$
and $d(B',C)+d(C,D)\geq d(B',D)$.
Adding up these inequalities we get $L(\bd(ABCD))\geq d(B,D')+d(B',D)=L(\bd(A_0B_0C_0D_0))$
what was required to be shown.
\end{proof}

Lemma \ref{le.le8} obviously implies the following result.

\begin{corollary}
We have $L \bigl(\bd(ABCD)\bigr)\geq 2\sqrt{2} \cdot \sin(\theta/2+\pi/4)$ for any quadrilateral $ABCD$ with given diagonal lengths $d(A,C)=d(B,D)=1$
and the angle $\theta$ between them.
\label{olymp}
\end{corollary}


\begin{proof}[Proof of Proposition \ref{pro.case4.2nsimple}]
Let us suppose the contrary. Without loss of generality, we may suppose that there are at least two non-simple self Chebyshev centers $F$ and $E$, where
$F$ is an interior point of $[A,B]$,
$E$ is an interior point of $[C,D]$, while $\delta(\bd(ABCD))=d(F,D)=d(F,C)=d(E,A)=d(E,B)=1$
(see the right panel Fig.~\ref{olymp_f}).

By Corollary \ref{olymp} we have
$$
L(\bd(BCEF))\geq 2f(\theta_1)  \quad \mbox {and} \quad
L(\bd(ADEF))\geq 2f(\theta_2),
$$
where $f(\theta):=\cos\bigl(\theta/2\bigr)+\sin\bigl(\theta/2\bigr)$.
Therefore, we get
$$
L(\bd(ABCD))=L(\bd(BCEF))+L(\bd(ADEF))-d(E,F)\geq 2(f(\theta_1)+f(\theta_2))-1.
$$

Let us consider $\varphi_1:=\angle BEA$ and $\varphi_2:=\angle CFD$. It is easy to see that $\theta_1+\theta_2=\varphi_1+\varphi_2$.
Moreover, $\varphi_1/2=\arcsin(d(A,B)/2)$ and $\varphi_2/2=\arcsin(d(C,D)/2)$. Thus, we have
$$
\frac{d(A,B)}{2}+\frac{d(C,D)}{2}=\sin(\varphi_1/2)+\sin(\varphi_2/2)\leq \frac{\varphi_1+\varphi_2}{2}=\frac{\theta_1+\theta_2}{2}
$$
or $d(A,B)+d(C,D)\leq \theta_1+\theta_2$.

Suppose that $d(A,B)+d(C,D)\geq \frac{1}{\sqrt{2}}$. It implies $\theta_1+\theta_2\geq \frac{1}{\sqrt{2}}$.
Consider the function $F(\theta_1,\theta_2):=f(\theta_1)+f(\theta_2)$ on the set
\begin{equation}
S:=\left\{(\theta_1,\theta_2)\,\left|\,\,0\leq \theta_1,\theta_2 \leq \frac{\pi}{2},\,\theta_1+\theta_2\geq \frac{1}{\sqrt{2}} \right.\right\}.
\end{equation}
Since $f'(\theta)=\sqrt{2}\sin'(\theta/2+\pi/4)=\frac{1}{\sqrt{2}}\cos(\theta/2+\pi/4)>0$ for $\theta\in[0,\pi/2)$ then the function $f(\theta)$ is strictly increasing at a given interval.
On the other hand $f''(\theta)=\frac{1}{\sqrt{2}}\cos'(\theta/2+\pi/4)=-\frac{1}{2\sqrt{2}}\sin(\theta/2+\pi/4)<0$ for $\theta\in[0,\pi/2]$. Thus, the minimum of the function $F(\theta_1,\theta_2)$ is possible only at three points: $(0,1/\sqrt{2})$, $(1/\sqrt{2},0)$ and $(1/\sqrt{2},1/\sqrt{2})$. By direct computations we get
$$
F(0,1/\sqrt{2})=F(1/\sqrt{2},0)=2.2843..., \quad F(1/\sqrt{2},1/\sqrt{2})=2.5687....
$$
Hence, $L(\bd(ABCD))\geq 2F(\theta_1,\theta_2)-1=3.5686...$ and the quadrilateral $ABCD$ is not extremal.

Therefore, $d(A,B)+d(C,D)\leq \frac{1}{\sqrt{2}}$, that implies $d(A,B)\leq \frac{1}{\sqrt{2}}$ and $d(C,D)\leq \frac{1}{\sqrt{2}}$.
By Lemma \ref{le.10}, we get $L(\bd(ABCD))>2+\sqrt{2}$, hence, $ABCD$ is not extremal.
\end{proof}

\bigskip

\smallskip
According to Proposition \ref{pro.case4.2nsimple}, any extremal quadrilateral $ABCD$ with $\|G_{\min}\|=4$
has at least $2$ simple self Chebyshev centers on some adjacent sides
(otherwise there are a pair of opposite sides with non-simple centers, which is impossible due to the mentioned proposition).
Fix a pair of such simple self Chebyshev centers, $E$ and $F$.
\smallskip

There are three possible cases:
\smallskip
\begin{itemize}
\item They have one and the same farthest vertex;

\item They have different farthest vertices that are adjacent;

\item They have different farthest vertices that are not adjacent.
\end{itemize}
\smallskip

We shall call this three possibilities as Cases 4.1, 4.2, and 4.3 respectively.
\smallskip

\begin{figure}[t]
\begin{minipage}[h]{0.42\textwidth}
\center{\includegraphics[width=0.99\textwidth, trim=0mm 1mm 0mm 2mm, clip]{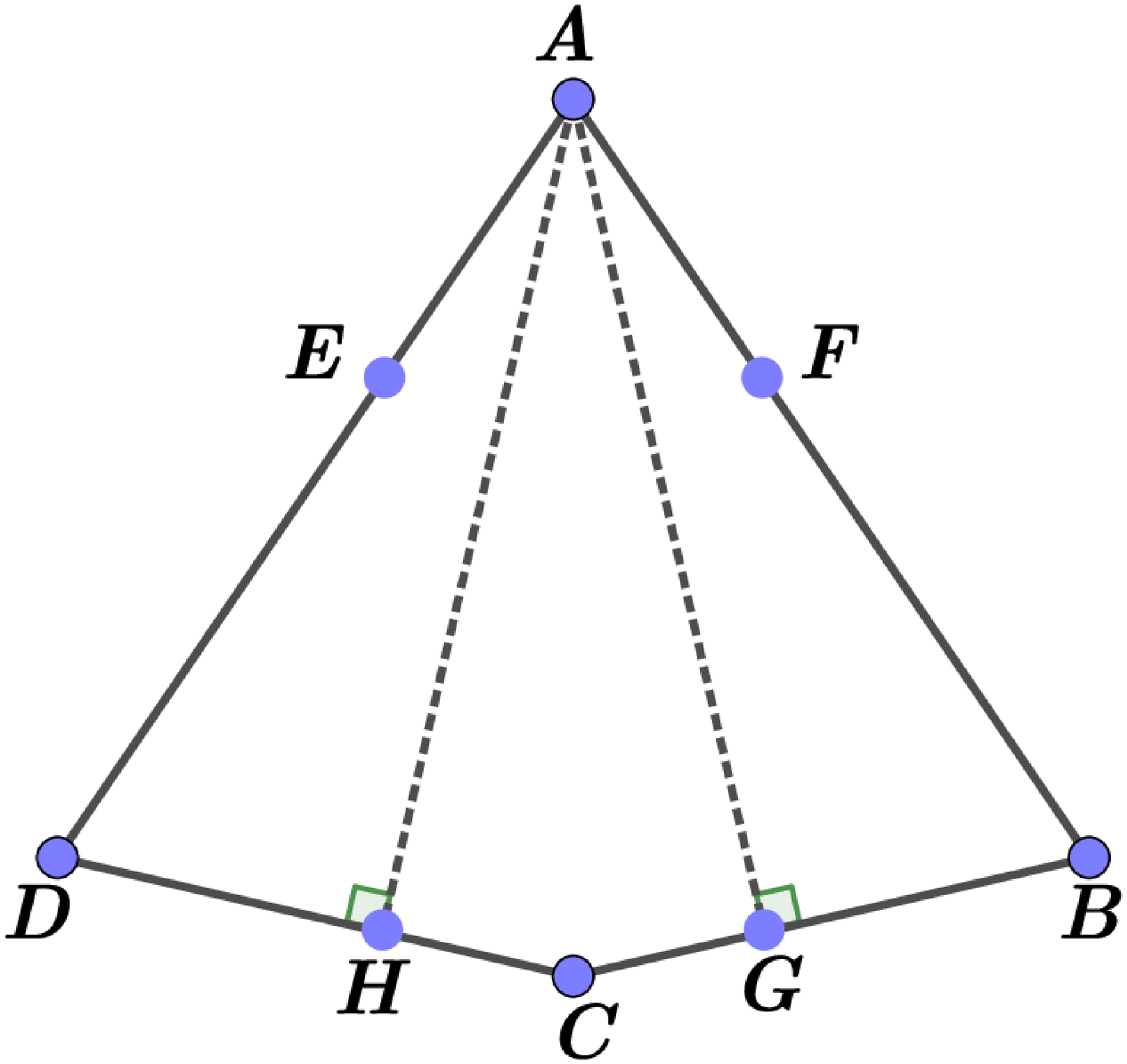}}   a) \\
\end{minipage}
\quad\quad
\begin{minipage}[h]{0.46\textwidth}
\center{\includegraphics[width=0.99\textwidth, trim=0mm 1mm 0mm 2mm, clip]{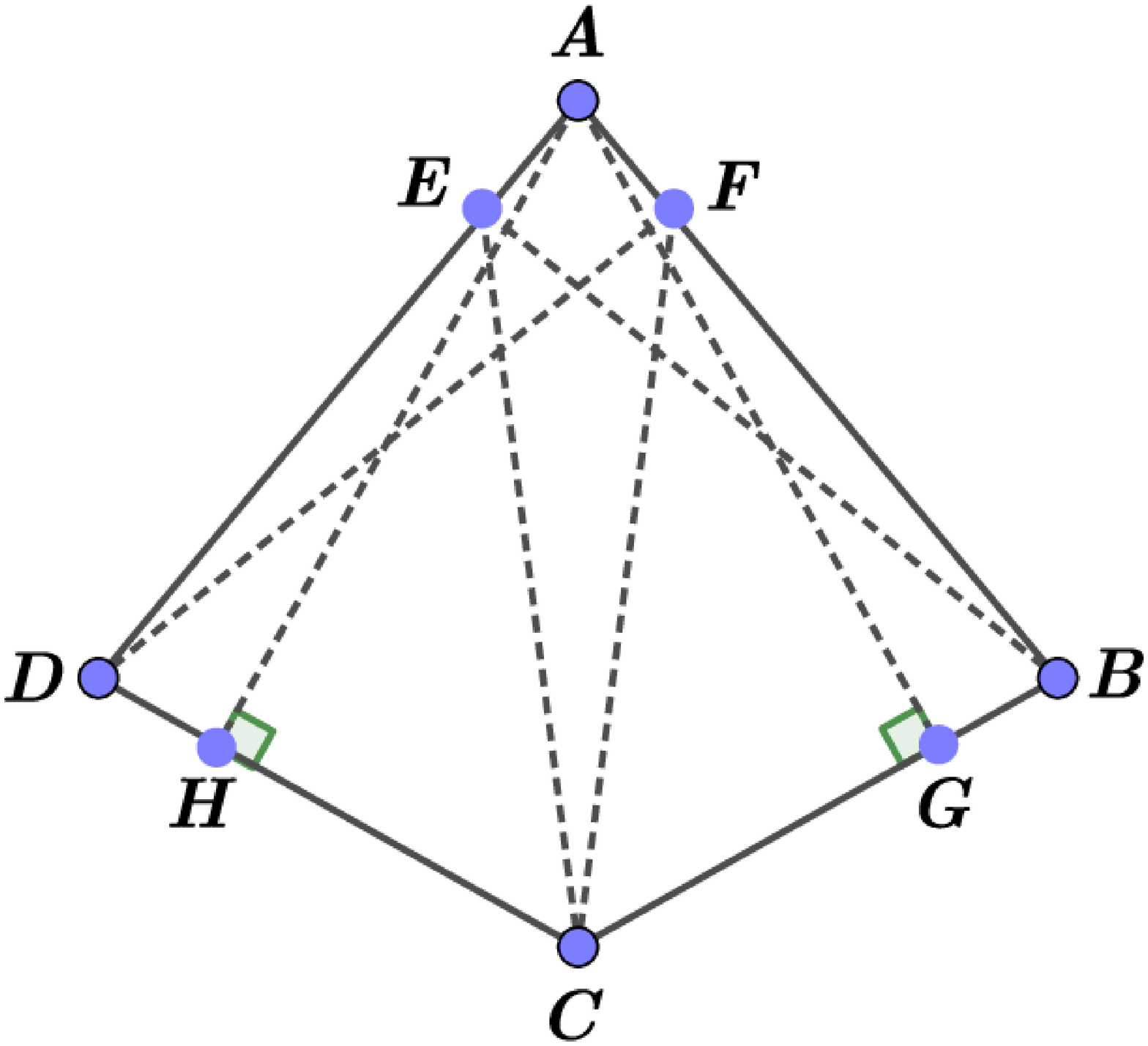}} b) \\
\end{minipage}
\caption{a) Case 4.1; b) Case 4.1a.}
\label{case4_1}
\end{figure}

\subsection{Case 4.1}
Let us consider a quadrilateral $ABCD$ with $\delta(\bd(ABCD))=d(A,H)=d(A,G)=1$
and $AH \perp [C,D]$, $AG \perp [B,C]$
(see the left panel Fig. \ref{case4_1}).

If $\angle BAD \geq \pi/2$, then $L(\bd(ABD))\geq 2+\sqrt{2}$. Hence, we have  $L(\bd(ABCD)) > 2+\sqrt{2}$, that is impossible for an extremal quadrilateral.
Thus, $\angle BAD< \pi/2$.

If $d(C,F)=1$ and $CF \perp [A,B]$ then $HAFC$ is rectangle and $d(H,F)>d(C,F)=1$.
Therefore, $H$ and $F$ are not self Chebyshev centers. Hence $CF$ is not orthogonal to $[A,B]$ as well as  $CE$ is not orthogonal to $[A,D]$.

If $E$ is a simple and $F$ is not simple (or vice versa,
$F$ is a simple and $E$ is not simple) self Chebyshev centers respectively, then the case is impossible (see Remark \ref{re.3.2} for Case~3.2).

Let us consider the case when $E$ and $F$ are not simple self Chebyshev centers,
i.~e. $\delta(\bd(ABCD))=d(D,F)=d(C,F)=d(C,E)=d(B,E)=d(H,A)=d(G,A)=1$
(see the right panel of Fig.~\ref{case4_1}).
Moreover, $\angle DFA \leq \pi/2$ and $\angle BEA \leq \pi/2$ by Proposition \ref{locmin_mu1}.
We show that this case is also impossible.

We put  $\alpha:=\angle DAH$, $\beta:=\angle BAG$, and
$\gamma:=\angle HAC=\angle GAC$.
It is clear that $\alpha<\gamma$ and $\beta<\gamma$.
From the triangle $HAB$ we have $\angle HAB=\beta+2\gamma<\pi/3$.
Similarly, $\angle GAD=\alpha+2\gamma<\pi/3$.
It implies $\alpha+\gamma<\pi/4$.

From $\tan(\alpha+\gamma)=\frac{\tan(\alpha)+\tan(\gamma)}{1-\tan(\alpha)\tan(\gamma)}<\tan(\pi/4)=1$,
we get $\tan(\alpha)+\tan(\gamma)<1$, i.e. $d(C,D)<1$.
Hence, $\angle FDC>\pi/3$.

Since $\angle DAF=\alpha+\angle HAB<\alpha+\pi/3$ and
$\angle ADF=\angle ADH-\angle FDC<\pi/2-\alpha-\pi/3$, then
$\angle DAF+\angle ADF<\pi/2$.
Consequently, $\angle DFA=\pi-\angle DAF-\angle ADF>\pi/2$ and we get a contradiction.

\begin{figure}
\center{\includegraphics[width=0.4\textwidth]{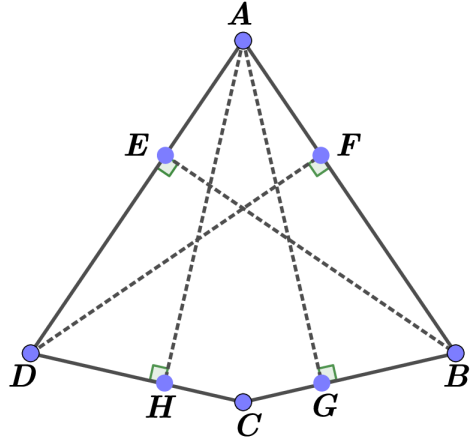}}\\
\caption{Case 4.1b.}
\label{case4_1b}
\end{figure}

We suppose now  that $E$ and $F$ are simple self Chebyshev centers. Hence $BE \perp [A,D]$ and $DF \perp [A,B]$ (see Fig. \ref{case4_1b}).

It is clear that the triangles $DFA$ and $AHD$, and respectively, $BEA$ and $AGB$ are congruent. It implies $\angle CDA= \angle DAB = \angle ABC:=\alpha$
and the triangles $DFA$ and $BEA$ are congruent. Hence $d(A,B)=d(A,D)$ and $d(B,C)=d(C,D)$.

It is clear that $\angle FDC=\angle ADC - \angle ADF=\alpha -(\pi/2-\alpha)=2\alpha-\pi/2$.
Hence $\angle BCD=2\pi-(\pi/2+\alpha+2\alpha-\pi/2)=2\pi-3\alpha$.
Since the triangles $AHC$ and $AGC$ are congruent then $\angle ACD=\angle ACB=\pi-\frac{3}{2}\alpha$.
It implies $d(C,H)=\cot(\pi-3\alpha/2)=-\cot(3\alpha/2)=d(C,G)$
and $d(D,H)=d(B,G)=\cot(\alpha)$.

Therefore, we get
$$
L\bigl(\bd(ABCD)\bigr)=\frac{2}{\sin(\alpha)}+2\cot(\alpha)-2\cot\left(\frac{3\alpha}{2}\right)=:f(\alpha).
$$
Since $\pi-3\alpha/2<\pi/2$ then $3\alpha>\pi$, i.e. $\alpha>\pi/3$.

By direct computations, we obtain
\begin{eqnarray*}
f'(\alpha)&\!\!\!\!=\!\!\!\!&
1-\frac{2\cos(\alpha)}{\sin^2(\alpha)}-\frac{2\cos^2(\alpha)}{\sin^2(\alpha)}+\frac{3\cos^2(\frac{3\alpha}{2})}{\sin^2(\frac{3\alpha}{2})}\\
&\!\!\!\!=\!\!\!\!&\frac{1}{\sin^2(\alpha)\sin^2(\frac{3\alpha}{2})}\left(\sin^2\left(\frac{3\alpha}{2}\right)\left(1-3\cos^2(\alpha)
-2\cos(\alpha)\right)+3\cos^2\left(\frac{3\alpha}{2}\right)\sin^2(\alpha)\right)\\
&\!\!\!\!=\!\!\!\!&-\frac{2}{\sin^2(\frac{3\alpha}{2})}\left(8\cos^4\left(\frac{\alpha}{2}\right)-4\cos^2\left(\frac{\alpha}{2}\right)-1\right).
\end{eqnarray*}

Solving the equation $8\cos^4\left(\frac{\alpha}{2}\right)-4\cos^2\left(\frac{\alpha}{2}\right)-1=0$ with respect to $\alpha$, we obtain
$$
\alpha_0=2\arccos\left(\frac{1}{2}\sqrt{1+\sqrt{3}}\right).
$$

It is easy to check that $f(\alpha)$ achieves its minimal value on $(\pi/3,\pi/2)$ exactly at the point $\alpha_0$.
This minimal value is
$$
f\left(2\arccos\left(\frac{1}{2}\sqrt{1+\sqrt{3}}\right)\right)=\frac{4}{3}\sqrt{3+2\sqrt{3}}=3.389946...
$$
Therefore, if a quadrilateral $ABCD$ is extremal, then it is a magic kite.
\medskip

\subsection{Case 4.2}
Let us consider a quadrilateral $ABCD$ with $\delta(\bd(ABCD))=d(A,H)=d(B,E)=1$
and $AH \perp [D,C]$, $BE \perp [A,D]$
(see the left panel Fig. \ref{case4_2}).

If $d(D,G)=1$ and $DG \perp [B,C]$ then $EBGD$ is rectangle and $d(E,G)>d(D,G)=1$.
Therefore, $E$ and $G$ are not self Chebyshev centers. Hence $DG\not\perp [B,C]$ and $CF\not\perp [A,B]$.

If $d(D,F)=d(C,F)=1$ then this case is impossible (see Remark \ref{re.3.2} for Case 3.2).

Now, we consider the case when $DF \perp [A,B]$ and $d(A,G)=d(D,G)=1$ (see the right panel Fig.~\ref{case4_2}).
We put $\alpha:=\angle BAD=\angle ADC<\pi/2$, $\gamma:=\angle EBC<\pi/2$ and denote
by $G'$ the midpoint of $[A,D]$.

If $\angle AGD \geq \frac{\pi}{2}$, then $L(\bd(AGD))\geq 2+\sqrt{2}$. Hence, we have  $L(\bd(ABCD)) > 2+\sqrt{2}$, that is impossible for an extremal quadrilateral.
Thus, $\angle AGD< \frac{\pi}{2}$. It implies that $\alpha>\angle GAD >\frac{\pi}{4}$.

\begin{figure}[t]
\begin{minipage}[h]{0.46\textwidth}
\center{\includegraphics[width=0.99\textwidth, trim=0mm 1mm 0mm 2mm, clip]{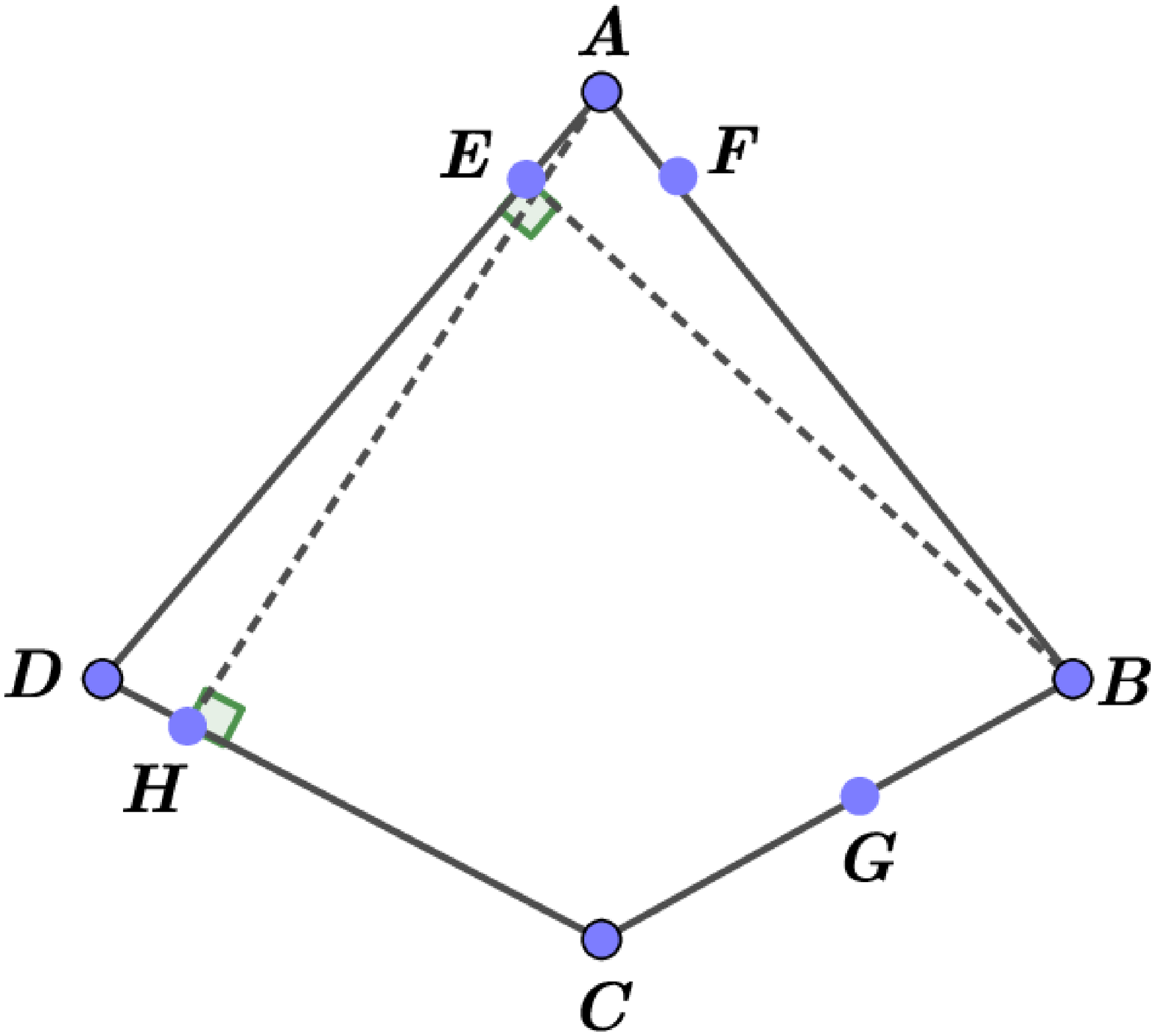}}   a) \\
\end{minipage}
\quad\quad
\begin{minipage}[h]{0.45\textwidth}
\center{\includegraphics[width=0.99\textwidth, trim=0mm 1mm 0mm 2mm, clip]{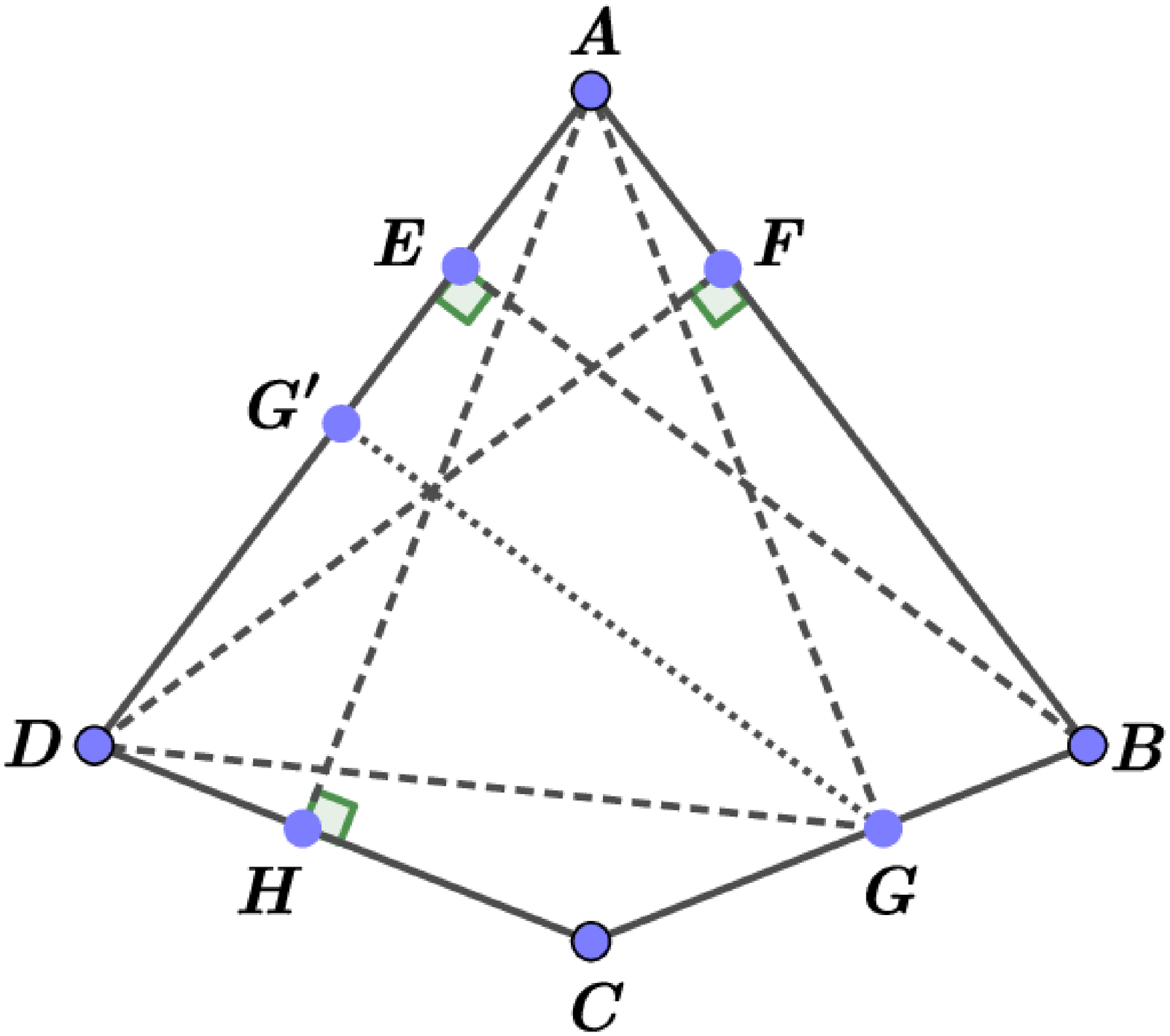}} b) \\
\end{minipage}
\caption{a) Case 4.2; b) Case 4.2a.}
\label{case4_2}
\end{figure}

It is easy to see that $d(A,E)=\cot(\alpha)$ and $d(A,G')=\frac{1}{2}d(A,D)=\frac{1}{2\sin(\alpha)}$. Hence $d(E,G')=\frac{1}{2\sin(\alpha)}-\cot(\alpha)=\frac{1-2\cos(\alpha)}{2\sin(\alpha)}$ and
$d(G,G')=\sqrt{1-d^2(A,G')}=\sqrt{1-\frac{1}{4\sin^2{\alpha}}}$.

From $\cot(\gamma)d(E,G')+d(G,G')=d(E,B)=1$ we get
$$
\cot(\gamma)(1-2\cos(\alpha))+\sqrt{4\sin^2(\alpha)-1}=2\sin(\alpha),
$$
or, equivalently,
$$
\gamma=\arctan\left(\frac{1-2\cos(\alpha)}{2\sin(\alpha)-\sqrt{4\sin^2(\alpha)-1}}\right).
$$

Since $\angle BAH=\angle BAD-\angle DAH=\alpha-(\pi/2-\alpha)=2\alpha-\pi/2$ then from
the triangle $BAH$ we have
$$
1+\frac{1}{\sin^2(\alpha)}-\frac{2}{\sin(\alpha)}\cos(2\alpha-\pi/2)=1+\frac{1}{\sin^2(\alpha)}-4\cos(\alpha)=d^2(B,H)\leq 1.
$$
It implies $1\leq 4\sin^2(\alpha)\cos(\alpha)$. On the other hand
since $\angle GAB= \alpha-\angle G'AG$ we get
$$
\angle AGB=\pi-\gamma-(\pi/2-\alpha)-\alpha+\angle G'AG=\pi/2-\gamma+\angle G'AG\leq \pi/2.
$$
Thus, $\gamma\geq \angle G'AG=\arctan\left(\frac{d(G,G')}{d(A,G)}\right)=\arctan(4\sin^2(\alpha)-1)$
or, substituting the found expression for $\gamma$, we get
$$
\frac{1-2\cos(\alpha)}{2\sin(\alpha)-\sqrt{4\sin^2(\alpha)-1}}\geq 4\sin^2(\alpha)-1,
$$
which can be rewritten as  $1\geq 4\sin^2(\alpha)\cos(\alpha)$.

Solving equation $1= 4\sin^2(\alpha)\cos(\alpha)$ with respect to $\alpha\in(\pi/4,\pi/2)$, we obtain
$\alpha_0=\arccos(t_0)=1.297824478...$, where $t_0=0.2695944364...$ is the root of the polynomial $f(t)=4t^3-4t+1$ on the interval $[0.2,\,0.3]$.
By direct computations, we get $\gamma_0=1.024852628...$ and $\angle AGB= \pi/2$.
Hence $d(C,D)=d(B,C)$.

It is easy to see that $\angle ACB=\pi-(\pi/2-\alpha)-\gamma-\alpha/2=\pi/2+\alpha/2-\gamma$. The equalities
$$
\frac{d(B,C)}{\sin(\alpha/2)}=\frac{d(A,B)}{\sin(\angle ACB)}=\frac{1}{\sin(\alpha)\sin(\pi/2-\gamma+\alpha/2)}=\frac{1}{\sin(\alpha)\cos(\gamma-\alpha/2)},
$$
imply
$$
d(B,C)=\frac{1}{2\cos(\alpha/2)\cos(\gamma-\alpha/2)}=\frac{1}{\cos(\gamma)+\cos(\alpha-\gamma)}.
$$

Therefore, by direct computations, we get
$$
L(\bd(ABCD))=2d(A,B)+2d(B,C))=\frac{2}{\sin(\alpha_0)}+\frac{2}{\cos(\gamma_0)+\cos(\alpha_0-\gamma_0)}=3.426...
$$
Hence, the quadrilateral $ABCD$ is not extremal.
\medskip

\begin{figure}[t]
\begin{minipage}[h]{0.46\textwidth}
\center{\includegraphics[width=0.99\textwidth, trim=0mm 1mm 0mm 2mm, clip]{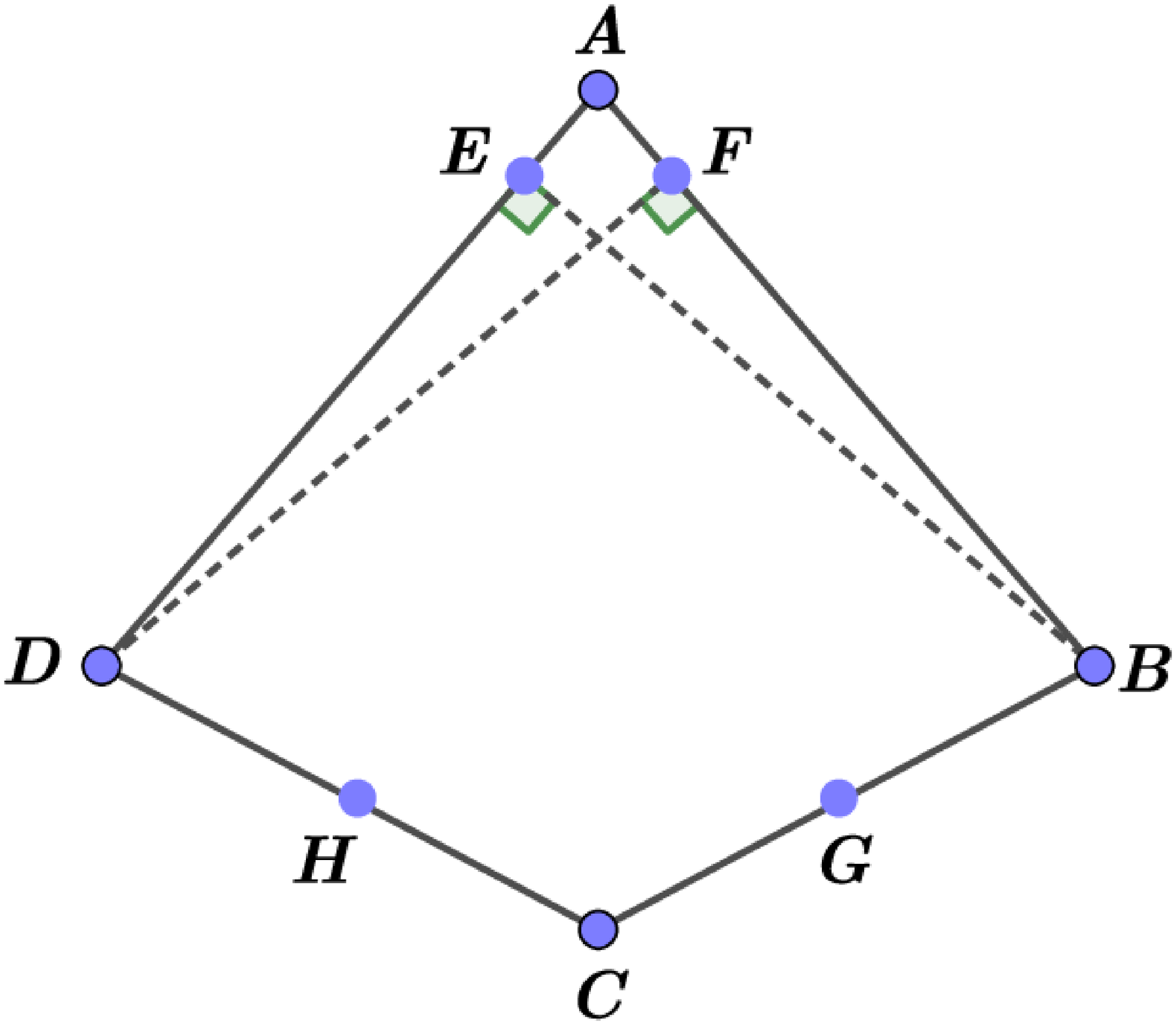}}   a) \\
\end{minipage}
\quad\quad
\begin{minipage}[h]{0.46\textwidth}
\center{\includegraphics[width=0.99\textwidth, trim=0mm 1mm 0mm 2mm, clip]{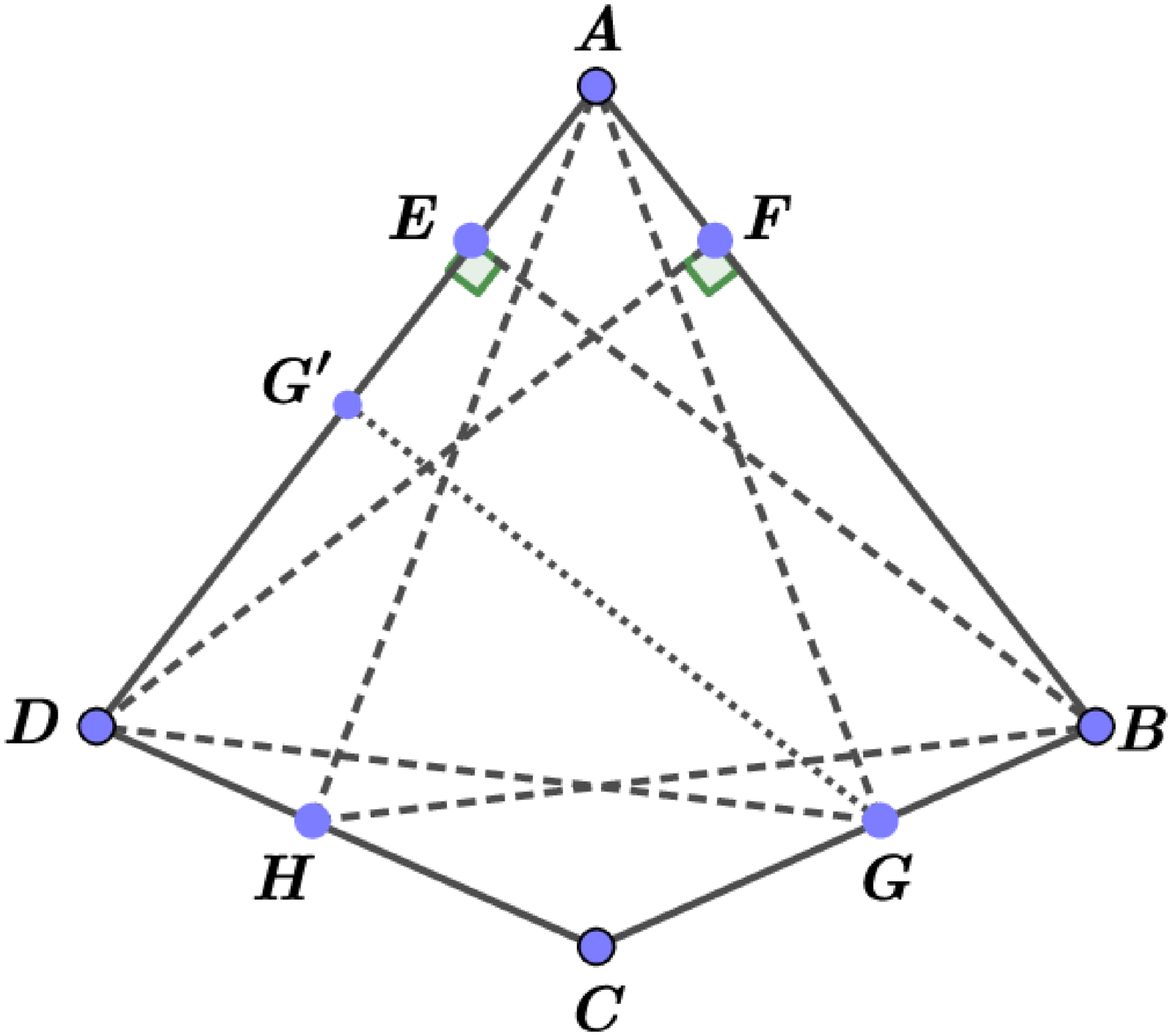}} b) \\
\end{minipage}
\caption{a) Case 4.3; b) Case 4.3a.}
\label{case4_3}
\end{figure}

\subsection{Case 4.3}
Let us consider a quadrilateral $ABCD$ with $\delta(\bd(ABCD))=d(D,F)=d(B,E)=1$
and $DF \perp [A,B]$, $BE \perp [A,D]$ (see the left panel Fig. \ref{case4_3}).

If $d(D,G)=1$ and $DG \perp [B,C]$ then $EBGD$ is rectangle and $d(E,G)>d(D,G)=1$.
Consequently, $E$ and $G$ is not self Chebyshev centers. Therefore, $DG$ not $\perp [B,C]$ and $CF$ not $\perp [A,B]$.

Now, we suppose that $d(A,H)=d(B,H)=1$ and $d(A,G)=d(D,G)=1$ (see the right panel Fig. \ref{case4_3}).
It is easy to see that $d(A,B)=d(A,D)$ and $d(B,C)=d(C,D)$.
We put $\alpha:=\angle BAD<\pi/2$, $\gamma:=\angle EBC<\pi/2$ and
denote by $G'$ be the midpoint of $[A,D]$.

If $\angle AGD \geq \frac{\pi}{2}$, then $L(\bd(AGD))\geq 2+\sqrt{2}$. Hence, we have  $L(\bd(ABCD)) > 2+\sqrt{2}$, that is impossible for an extremal quadrilateral.
Thus, $\angle AGD< \frac{\pi}{2}$. It implies that $\alpha>\angle GAD >\frac{\pi}{4}$.

Repeating the reasoning similar to the previous case we get $1\geq 4\sin^2(\alpha)\cos(\alpha)$.
Solving this inequality for $\alpha\in(\pi/4,\pi/2)$, we obtain $\alpha\in(\alpha_0,\pi/2)$ and
$\alpha_0=\arccos(t_0)=1.297824478...$, where $t_0=0.2695944364...$ is a unique root of the polynomial $4t^3-4t+1$ on the interval $[0.2\,,0.3]$.

Also similar to the previous case, we get
$$
L(\bd(ABCD))=2d(A,B)+2d(B,C)=\frac{2}{\sin(\alpha)}+\frac{2}{\cos(\gamma)+\cos(\alpha-\gamma)}.
$$

The results of numerical calculations show
that the perimeter $L\bigl(\bd(ABCD)\bigr)$ for the quadrilaterals under consideration reaches its minimal value
exactly for $\angle AGB =\angle AHD=\pi/2$ or $\alpha=\alpha_0=1.297824478...$ and $\gamma_0=1.024852628...$.
This minimal value is $L(\bd(ABCD))=3.426...$.
Hence, the quadrilateral $ABCD$ is not extremal.
\bigskip

It should be noted, that we found a quadralateral that can be extremal only in Case~4.1. Since at least one extremal quadrilateral does exist, we conclude
that it is a magic kite from Case 4.1 and from the conjecture by Rolf Walter.
This argument finished the proof of Theorem \ref{t.main}.

\section{Computations and further conjectures}\label{sect.6}

The search for possible extremal polygons for small values of $n$ was undertaken by E.V.~Nikitenko.
His experiments led to the conjecture that {\it for odd $n$, any extremal polygon is a regular $n$-gon} (this is the case for $n=3$).
On the other hand, {\it for even values of $n$, regular $n$-gons are not extremal} in the above sense. The type of an extremal polygon
(quite possibly) depends on the power of the number $2$ with which it enters to $n$ as a multiplier.
See Fig. \ref{Fig3} for hypothetical extreme polygons for $n=6$ and $n=10$.
The calculation of the characteristics of these polygons (in particular, these polygons are equilateral),
as well as numerous computer experiments, lead to the following conjectures:

\begin{figure}[t]
\begin{minipage}[h]{0.45\textwidth}
\center{\includegraphics[width=0.99\textwidth]{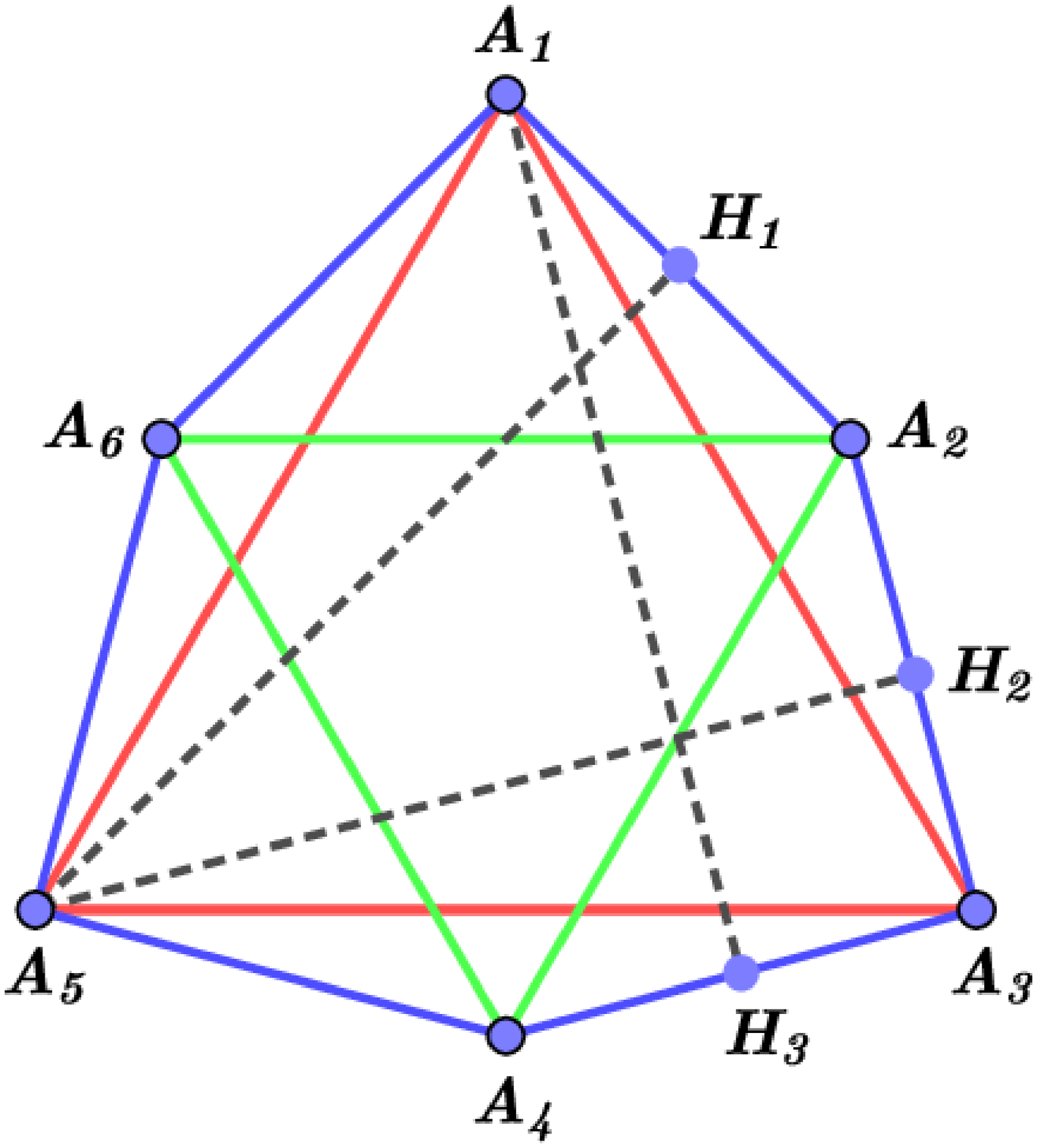}}    a) \\
\end{minipage}
\quad\quad
\begin{minipage}[h]{0.48\textwidth}
\center{\includegraphics[width=0.99\textwidth]{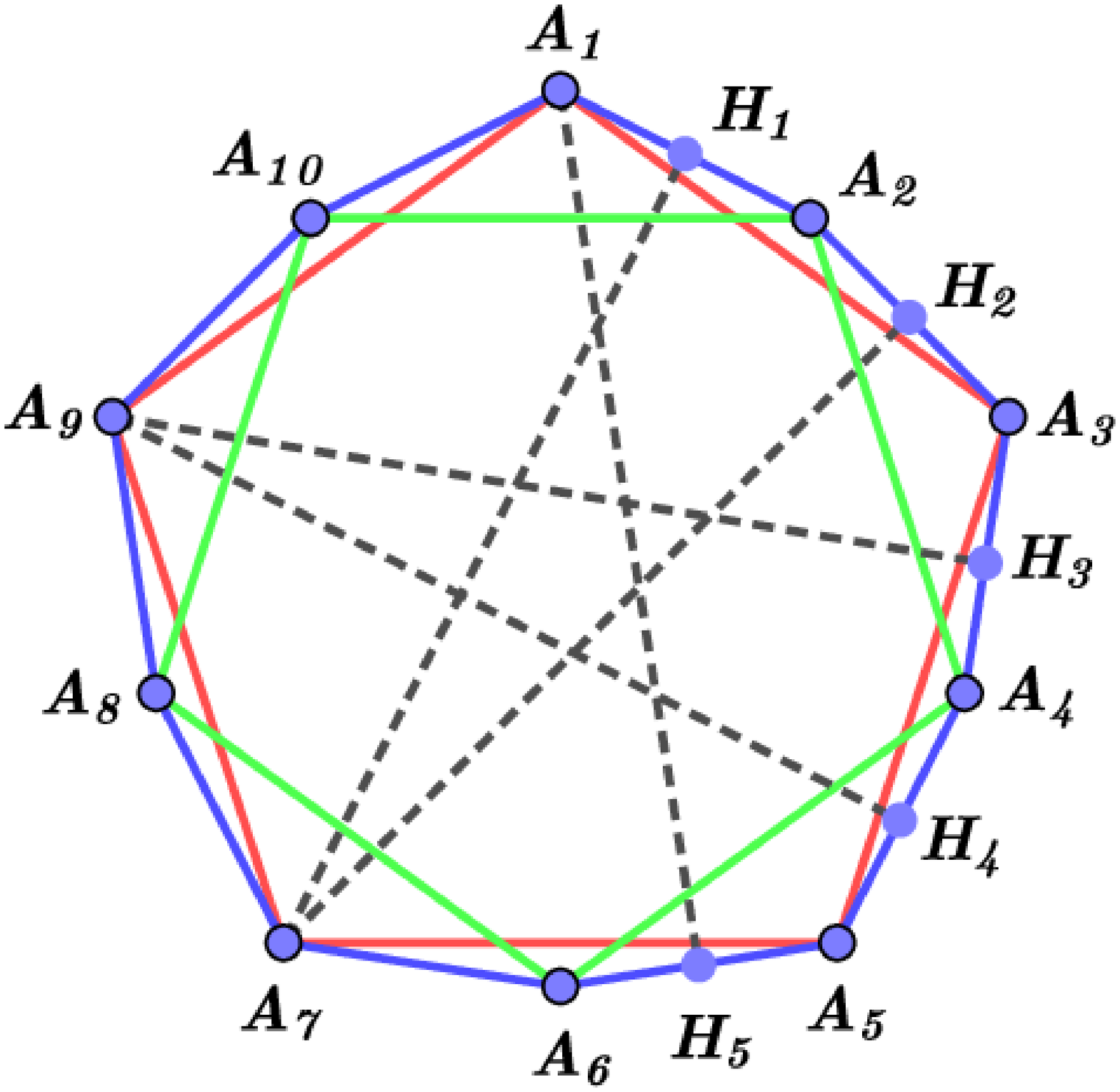}} b) \\
\end{minipage}
\caption{Extreme $n$-gon candidates for: a) $n=6$; b) $n=10$. In both cases, the green and red
polygons are regular, and are negative homothets of each other.}
\label{Fig3}
\end{figure}

\begin{conjecture}
For any convex $6$-gon $P$ in the Euclidean plane, one has
$$
L(\Gamma)\geq 12\left(2- \sqrt{3}~\right)\cdot \delta(\Gamma),
$$
where $\Gamma$ is the boundary of $P$.
\end{conjecture}

\begin{conjecture}
For any convex $10$-gon $P$ in the Euclidean plane, one has
$$
L(\Gamma)\geq 20\left( 1+\sqrt{5}-\sqrt{5+2\sqrt{5}}~\right)\cdot \delta(\Gamma),
$$
where $\Gamma$ is the boundary of $P$.
\end{conjecture}

It should be noted that there are several extremal problems for polygons, with a similar behavior of the solution with respect to the number of sides $n$.
We recall two well-known problems of such kind.

Two classical isodiametric problems for polygons are to determine the maximal area and maximal perimeter of an $n$-gon with unit diameter.
Important results in the study of these two problems were obtained by K.~Reinhardt in \cite{Rein22}.

One of the main his results is the following
\cite[Theorem 1, P. 252]{Rein22}:
{\it Let $n$ be an odd number, $n\geq 3$; then the regular $n$-gon has maximal area among all
simple $n$-gons of diameter $d$.}

For even $n \ge 6$ this is not true.
For $n = 4$ it is true, but the square is not a unique solution:
The diagonals of the square can be slightly shifted relative to each other, leaving them orthogonal.

It should be noted that a $n$-gon of the maximum area among $n$-gons with diameter $d$ is unique if $n$ is an odd number \cite[P. 259]{Rein22}.

The second main result by K.~Reinhardt is the following
\cite[Theorem 2, P. 252]{Rein22}: {\it Let $n$  be an integer that is not a degree of $2$, $n \geq 3$. Then, among the convex $n$-gons with diameter $d$,
the maximum perimeter is achieved if and only if
a $n$-gon is equilateral and inscribed in a Reulaux polygon of diameter $d$ so that the vertices of the
Reulaux polygon are also the vertices of the $n$-gon} (such polygons are called {\it Reinhardt polygons}).

Moreover, this maximum perimeter coincides with the perimeter of the regular $2n$-gon of diameter $d$, as E.~Makai,~J. noticed (a private conversation).
Indeed, if $P_n$ is such a $n$-gon, then the diameter of the polygon $\frac{1}{2}(P_n - P_n)$  (this is the central symmetral of $P_n$,
that have the same perimeter as $P_n$ has) is also equal to $d$, $\frac{1}{2}(P_n - P_n)$ is inscribed
into a circle of diameter $d$, and has $\le 2n$ sides.
So, this perimeter is no more than the perimeter of the regular $2n$-gon incribed into this circle
(that can obtained via the above procedure from  a regular $n$-gon of diameter $d$).

We also note that  for any simple $n \geq 3$, there is only one $n$-gon of the maximum perimeter among $n$-gons with diameter $d$ \cite[P. 263]{Rein22}.

The largest areas of a hexagon and
an octagon, resp., of unit diameter were given by R. Graham \cite{Graham}, and by C. Audet, P. Hansen, F.
Messine, J. Xiong~\cite{AHMX}. The largest perimeter of an octagon of unit diameter was given in C.
Audet, P. Hansen, F. Messine~\cite{AHM07}. The cases of higher even numbers, and
higher powers of $2$ seem to be open.
See also \cite{HaMo2019} for more recent results on the topic.
\smallskip

C.~Audet, P.~Hansen, and F.~Messine  showed that the value $\frac{1}{2n} \cot \left( \frac{\pi}{2n} \right)$ is  an upper bound for the width of any $n$-sided
polygon with unit perimeter.
This bound is reached when $n$ is not a power of $2$,
and the corresponding optimal solutions are the regular polygons when $n$ is odd and clipped regular Reuleaux polygons when $n$ is even but not a power of $2$ \cite{AHM09}.
\smallskip

We hope that these analogies will be useful in further research and will help to find an approach to the conjectures formulated above.
\medskip

For the convenience of readers,
here we provide {\bf the list of some important notations and definitions used in the paper}
(for each object, the page of the paper on which it is introduced is indicated).

\begin{itemize}

\item
For a given set $M$ in a given metric space $(X,d)$, the following concepts are defined on the Page \pageref{defchebrad}: {\it the Chebyshev radius, a Chebyshev center,
the relative Chebyshev radius} (with respect to a given non-empty subset in $X$), {\it the self Chebyshev radius}.

\item
Quadrilaterals called {\it magic kites} are defined on Page \pageref{defmagkite}.

\item
For a given compact set $\Gamma \subset \mathbb{R}^2$ (in particular, a convex curve), the function
$\mu: \Gamma \rightarrow \mathbb{R}$ is defined by  \eqref{mufunk} on Page \pageref{mufunk}.

\item
For a given point $A \in \Gamma$, where $\Gamma$ is  the boundary of a fixed $n$-gon $P$,
the notations $T_+(A)$, $T_-(A)$, and $N(A)$ are introduced on Page \pageref{eq.fath1},
and $\mathcal{F} (A)$ is defined by \eqref{eq.fath1} on the same page.

\item
$G_{\operatorname{locmin}}$, the set of local minimum points of the function $\mu$, is introduced on Page~\pageref{deflocmin}.

\item
The definition of {\it extremal} $n$-gon $P$ is given on Page \pageref{defextrem}, see also Remark \ref{rem4} on Page \pageref{rem4}.

\item
For the boundary $\Gamma$ of a fixed $n$-gon $P$,  the function $t: \mathbb{R}^2 \rightarrow \mathbb{R}$ is define by~\eqref{eq_fuct1} on Page \pageref{eq_fuct1}.
On the same page, the sets
$U_r$ and $L_r$ are defined for any $r>0$, and $G_{\min}$, the set of self Chebyshev centers for $\Gamma$, is introduced by \eqref{eq_fuct3}.

\item
On Page \pageref{mumchebcenters}, the notation $\|G_{\min}\|$ is introduced for the number of points in the set $G_{\min}$, or, in other words, the number
of self Chebyshev centers for $\Gamma$.

\item
{\it Simple}  and {\it non-simple}  self Chebyshev centers for $\Gamma$ are defined on Page \pageref{simple_nonsimple}.

\item
A special class of quadrilaterals $\mathcal{P}_{\operatorname{spec}}$ is introduced on Page \pageref{specquadr}.
\end{itemize}

\vspace{15mm}


\begin{thebibliography}{99}


\bibitem{AlTsa2019}
A.R.~Alimov, I.G.~Tsar’kov,
{\sl Chebyshev centres, Jung constants, and their applications,}
Russ. Math. Surv. 74(5) (2019), 775--849,  {\bf MR}4017575, {\bf Zbl.}1435.41028,

\bibitem{AlNel}
C.~Alsina, R.B.~Nelsen,
{\sl A cornucopia of quadrilaterals,} The Dolciani Mathematical Expositions, 55.
MAA Press, Providence, RI; American Mathematical Society, Providence, RI, 2020,  {\bf MR}4286138, {\bf Zbl.}1443.51001


\bibitem{Am1986}
D.~Amir,
{\sl Characterizations of Inner Product Spaces,}
Operator Theory: Advances and Applications Series Profile, Vol. 20. Basel--Boston--Stuttgart: Birkh\"{a}user--Verlag. Basel, 1986, {\bf MR}0897527, {\bf Zbl.}0617.46030.


\bibitem{AmZig1980}
D.~Amir, Z.~Ziegler,
{\sl Relative Chebyshev centers in normed linear spaces, I,}
J. Approximation Theory 29 (1980), 235--252, {\bf MR}0597471, {\bf Zbl.}0457.41031.



\bibitem{AHM07}
C.~Audet, P.~Hansen, F.~Messine,
{\it The small octagon with longest perimeter},
J. Combin. Th. A, 114 (2007), 135--150, {\bf MR}2275585, {\bf Zbl.}1259.90096.


\bibitem{AHM09}
C.~Audet, P.~Hansen, F.~Messine,
{\it Isoperimetric polygons of maximum width},
Discrete Comput. Geom. 4(1) (2009), 45--60, {\bf MR}2470069, {\bf Zbl.}1160.52001.


\bibitem{AHMX}
C.~Audet, P.~Hansen, F.~Messine, J. Xiong,
{\it The largest small octagon},
J. Combin. Th. A, 98 (2002), 46--59, {\bf MR}1897923, {\bf Zbl.}1022.90013.



\bibitem{BMNN2021}
V.~Balestro, H.~Martini, Yu.G.~Nikonorov, Yu.V.~Nikonorova,
{\sl Extremal problems for convex curves with a given self Chebyshev radius,}
Results in Mathematics, 76(2) (2021), Paper No. 87, 13 pp.,  {\bf MR}4242658, {\bf Zbl.}07369099.


\bibitem{BaCo}
H.H.~Bauschke, P.K.~Combettes,
{\sl Convex analysis and monotone operator theory in Hilbert spaces, second edition,}
Cham: Springer, XIX+619 p. (2017), {\bf Zbl.}1359.26003, {\bf MR}3616647.


\bibitem{BoFe1987}
T.~Bonnesen, W.~Fenchel,
{\sl Theory of Convex Bodies}, BCS Associates, Moscow, ID, 1987. Translated
from the German and edited by L. Boron, C. Christenson and B. Smith., {\bf MR}0920366, {\bf Zbl.}0628.52001.

\bibitem{Fal1977}
K.J.~Falconer,
{\sl A characterisation of plane curves of constant width,}
J. Lond. Math. Soc., II. Ser. 16  (1977), 536--538, {\bf MR}0461287, {\bf Zbl.}0368.52003.

\bibitem{Olymp}
R.~Fedorov (ed.), A.~Belov (ed.), A.~Kovaldzhi (ed.), I.~Yashchenko (ed.), S.~Levy (ed.),
{\sl Moscow Mathematical Olympiads, 1993–1999.}
Translation of the 2006 Russian original. MSRI Mathematical Circles Library, 4. 
Providence, RI: American Mathematical Society (AMS); Berkeley, CA: Mathematical Sciences Research Institute (MSRI), 2011, {\bf MR}2866883, {\bf Zbl.}1236.00022.


\bibitem{Graham}
R.~Graham, {\sl The largest small hexagon,}
J. Combin. Th. A, 18 (1975), 165--170, {\bf MR}360353, {\bf Zbl.}0299.52006.


\bibitem{HaWa}
K.H.~Hang, H.~Wang,
{\sl Solving problems in geometry. Insights and strategies,}
Mathematical Olympiad Series 10. Hackensack, NJ: World Scientific, 2017, {\bf Zbl.}1372.00006.

\bibitem{HaMo2019}
K.G.~Hare, M.J.~Mossinghoff,
{\sl Most Reinhardt polygons are sporadic},
Geom. Dedicata 198 (2019), 1--18, {\bf MR}3933447,  {\bf Zbl.}1412.52011.

\bibitem{Landau}
R.H.~Landau,
{\sl A first course in scientific computing.
Symbolic, graphic, and numeric modeling using Maple, Java, Mathematica, and Fortran90}.
With contributions by R.~Wangberg, K.~Augustson, M.J.~Paez, C.C.~Bordeianu and C.~Barnes.
Princeton University Press, Princeton, NJ, 2005,  {\bf MR}2136880, {\bf Zbl.}1090.68120.

\bibitem{MaMoOl}
H.~Martini, L.~Montejano, D.~Oliveros,
{\sl Bodies of Constant Width. An Introduction to Convex Geometry with Applications.} Birkh\"auser/Springer, Cham, 2019, {\bf MR}3930585, {\bf Zbl.}06999635.

\bibitem{Pra2006}
V.V.~Prasolov
{\sl Problems in Plane Geometry},  5th ed., rev. and compl. (Russian),
Moscow, Russia: The Moscow Center for Continuous Mathematical Education, 2006.
\\Online access:
\url{https://mccme.ru/free-books/prasolov/planim5.pdf}

\bibitem{Rein22}
K.~Reinhardt,
{\sl Extremale Polygone gegebenen Durchmessers},
Jahresbericht der Deutschen Mathematiker-Vereinigung, 31 (1922), 251--270,

\bibitem{Walter2017}
R.~Walter,
{\sl On a minimax problem for ovals,}
Minimax Theory Appl. 2(2) (2017), 285--318, {\bf MR}0920366, {\bf Zbl.}1380.51012, see also  arXiv:1606.06717.



\end{thebibliography}
\end{document}